\documentclass[12pt]{article}
\usepackage[english]{babel}
\usepackage{sectsty}
\usepackage{amsmath}
\usepackage{amssymb}
\usepackage{amsthm}
\usepackage[inline]{enumitem}
\usepackage[hidelinks]{hyperref}
\usepackage{bbm}
\usepackage{bussproofs}
\usepackage{multicol}
\usepackage{slashbox}
\usepackage{tikz}
\usepackage{graphicx}
\usepackage{subcaption}
\usepackage{geometry}
\usepackage{amsmath,amssymb}
\usepackage{iftex}
\usepackage{longtable,booktabs,array}
\usepackage{calc}
\usepackage{etoolbox}
\usepackage{ragged2e}

\usepackage[bb=pazo]{mathalpha}
\usepackage{multirow}
\usepackage{todonotes}
\usepackage[linesnumbered,ruled,vlined]{algorithm2e}
\usepackage{float}

\newsavebox\ltmcbox
\newcounter{entryno}
\def\tabline{Test & \the\value{entryno} & Description\addtocounter{entryno}{1}\\}

 \usepackage{mathpazo}

\allowdisplaybreaks

\newtheorem{theorem}{Theorem}[section]
\newtheorem{corollary}{Corollary}[theorem]
\newtheorem{lemma}[theorem]{Lemma}
\newtheorem{definition}[theorem]{Definition}
\newtheorem{proposition}[theorem]{Proposition}

\newtheorem{remark}[theorem]{Remark}
\newtheorem{example}[theorem]{Example}
\newtheorem{notation}[theorem]{Notation}

\newcommand\restr[2]{{
  \left.\kern-\nulldelimiterspace 
  #1 
  \vphantom{\big|} 
  \right|_{#2} 
  }}

\newcommand {\co} {\textsf{c}}

\newcommand{\bT}{\mathbf{T}}
\newcommand{\bU}{\mathbf{U}}
\newcommand{\bF}{\mathbf{F}}
\newcommand{\ipl}{\ensuremath{\mathbf{IPL}}}
\newcommand{\sfo}{\textbf{S4}}

\title{Intuitionism with truth-tables: A Decision Procedure for \textbf{IPL} Based on RNmatrices}

\title{A new decision method for Intuitionistic Logic by 3-valued non-deterministic truth-tables (pre-print version)}

\author{
Renato Leme\footnote{\href{mailto:rntreisleme@gmail.com}{rntreisleme@gmail.com}}\hspace{0.2cm}$^{\gamma} $\\ 
Marcelo E. Coniglio\footnote{\href{mailto:coniglio@unicamp.br}{coniglio@unicamp.br}}\hspace{0.2cm}$^{\gamma}$ \\
Bruno Lopes\footnote{\href{mailto:bruno@ic.uff.br}{bruno@ic.uff.br}}\hspace{0.2cm}$^{\delta}$
\\
\\ 
$^{\gamma}${\small Institute of Philosophy and the Human Sciences (IFCH) and}\\ 
{\small Centre for Logic, Epistemology and the History of Science (CLE)}\\
{\small State University of Campinas (UNICAMP), Campinas, SP, Brazil}\\ 
$^{\delta}${\small Institute of Computing (IC), Fluminense Federal University (UFF), Niterói, RJ, Brazil} 
}

\date{}

\begin{document}
\sloppy

\maketitle

\begin{abstract}
   Kurt Gödel proved that it is not possible to characterize Intuitionistic Propositional Logic (\ipl) by means of finite and deterministic truth-tables. After extending the same result with respect to non-deterministic matrices, we provide a semantical characterization of \ipl\ by means of a $3$-valued  non-deterministic matrix with a restricted set of valuations. This structure allows to define an algorithm to delete unsound rows from the non-deterministic truth-tables generated for each formula, which constitutes a new and very simple decision procedure for \ipl. This method can be seen as truth-tables in a broader sense, and a way to overcome G\"odel's limiting result.
\end{abstract}

{\bf MSC class:} 03B20; 03B45; 03B25; 03B35; 68V15

\section{Introduction} \label{sect:Intro}

In 1932, Kurt Gödel established the uncharacterizability of intuitionistic propositional logic (\textbf{IPL}) with respect to finite-valued deterministic matrices~\cite{Godel1932}. He demonstrated that any finite deterministic matrix, strongly enough to validate every intuitionistic tautology, will also validate at least one intuitionistically invalid proposition. Later, in $1940$, James Dugundji proved a similarly significant result by adapting Gödel's argument to show that the same holds true for modal logic \textbf{S5} and all of its subsystems~\cite{dugun:40}.

These limitations led to new interpretations for modal and intuitionistic systems, such as Saul Kripke's relational semantics, where the notion of \emph{possible world} plays a central role. The tradition of using possible worlds to explain modalities dates back to Leibniz and is at the core of philosophical debates: Are possible worlds real or just helpful fiction? John Kearns challenged this tradition by defending that they are neither one nor the other~\cite[p. 86]{kearns_modal_1981}. To support his position, he proposed, in $1981$, a four-valued semantics for the modal systems \textbf{KT}, \textbf{S4}, and \textbf{S5} where the truth-value of a formula depends solely on the truth assignments of its subformulas. In Kearns' approach, the meaning of certain connectives is not entirely determined by their matrices, as some inputs can yield more than one output. This type of semantics was later formalized under the notion of {\em non-deterministic matrices} (Nmatrices, in short)~\cite{avr:lev:01}. To address this indeterminacy, instead of possible worlds, Kearns introduced a method for restricting the set of valuations, which was later referred to more generally as {\em level valuations} (see Section~\ref{sec4}). In this approach, not all valuations in the Nmatrix are permitted. Instead, only a subset that meets specific conditions to ensure the validity of the necessitation rule is allowed; this subset is precisely the set of level valuations. Nmatrices with a restricted set of valuations lead to the notion of {\em Restricted non-deterministic matrices} or  {\em Restricted Nmatrices} (RNmatrices, in short), central in the present paper.

Unlike Kripke semantics, however, Kearns's level valuations do not draw a straightforward decision procedure for those logics. Until recently, it was unclear whether obtaining such a decision procedure from level valuations would be possible. 

\begin{quote}
(\ldots) the interesting problem of determining whether or not a decision procedure can be obtained by means of level valuations remains open.~\cite{coniglio_errata_2016}
\end{quote}

In $2021$, Lukas Gr\"atz solved this problem for some modal systems (see~\cite{gratz_truth_2022}). He built upon Kearns's original RNmatrix $3$-valued RNmatrices for \textbf{KT} and \textbf{S4} with two decidable notions of \emph{partial level valuation}, one for \textbf{KT} and other for \textbf{S4} and show that both are analytic and co-analytic regarding the level valuations of the RNmatrix for the respective modal systems. To obtain the truth-table, one discards (by means of  a simple algorithm) any valuation that does not fit the corresponding definition. Ultimately, the remaining partial valuations suffice to decide the formula's validity. Analyticity and co-analyticity ensure that partial level valuation semantics (which constitute the rows of the truth-tables) coincide with the corresponding RNmatrix semantics, which characterizes each logic.

Constructing non-deterministic (i.e., row-branching) truth-tables in a finite-valued Nmatrix can be seen as a generalization of the standard (deterministic) truth-tables method for deciding validity in a given logic. 
Indeed, non-deterministic matrices were formally introduced by A. Avron and I. Lev in~\cite{avr:lev:01} with the aim of overcoming G\"odel-Dugundji's uncharacterizability results for some logics. The main purpose of Nmatrices was to obtain a decision procedure, analogous to finite-valued truth-tables, for logics that do not admit such a characterization.  It is worth noting that Nmatrices were already considered in the literature before these authors. For instance, N. Rescher already defined a non-deterministic matrix semantics in 1962, under the name of {\em quasi-truth-functional systems} (see~\cite{resch:62}). Afterwards, and among other authors,  this notion was independently proposed by Y. Ivlev for studying non-normal modal logics, starting in the 1970s (see for instance~\cite{ivl:73,ivl:86,ivl:88,ivl:24}), as well as J. Kearns (with restricted valuations).

In spite of its greater expressive power with respect to standard (deterministic) matrices, the generalization of the notion of truth-tables by means of Nmatrices has also limitations:  some logics, such as da Costa's paraconsistent logic $C_1$, cannot be characterized by a single finite Nmatrix, as proved by Avron in~\cite[Theorem~11]{avron:2007}.
The same is true for  \textbf{KT} and \textbf{S4}, as proved by Gr\"atz in~\cite{gratz_truth_2022}. Moreover, in Theorem~\ref{unchar-IPL} below we will show that \ipl\ also cannot be characterized by a single finite Nmatrix, extending so G\"odel's theorem from $1932$. These examples show that G\"odel-Dugundji's results hold, for some logics, in a stronger sense, that is,  w.r.t. Nmatrices. The key to overcoming this limitation is to consider {\em some} valuations in a suitable Nmatrix. This kind of semantical structures was formally introduced in~\cite{coniglio_two_2022}, under the name of {\em restricted Nmatrices}, or RNmatrices,  showing that each da Costa's logics $C_n$ (for $n \geq 1$) can be characterized by a single (decidable) $(n+2)$-valued RNmatrix; in particular, validity in $C_1$ can be decided by a $3$-valued RNmatrix by means of branching truth-tables (plus a simple criterion to delete inadequate rows).  Analogously, Gr\"atz decidable $3$-valued RNmatrices produce a decision procedure for \textbf{KT} and \sfo\ by means of $3$-valued branching truth-tables  (together with a simple method to delete unsound rows).

Finite-valued non-deterministic truth-tables, together with an algorithm to delete inadequate rows, is a decision procedure that can be seen as truth-tables in a broader sense, and is a way to overcome G\"odel-Dugundji's limiting results, even for more complicated logics such as  \textbf{KT}, \sfo, $C_1$ and \ipl. The aim of this paper is, precisely, to obtain a new decision procedure for \ipl\ based on this kind of generalized truth-tables.

Recall that, in $1933$, G\"odel proposed an interpretation of \ipl\ into \sfo\ by means of a translation mapping $t$ between their languages~\cite{Godel1933}. In $1948$, Tarski and McKinsey prove that G\"odel's mapping $t$ constitutes a faithful interpretation, that is: $\alpha$ is a theorem of \ipl\ iff its translation $t(\alpha)$ is a theorem of \sfo~\cite{McKinsey:Tarsk:48}. Since G\"odel translation $t$ is computable, then, by composing it with any algorithm for deciding \sfo, a decision procedure for \ipl\ is immediately obtained, defined by using a translation (from \ipl\ to \sfo). In particular, by composing  G\"odel translation $t$ with  Gr\"atz's algorithm for \sfo, a decision procedure is obtained for \ipl, based on an RNmatrix for \sfo\ (as well as the translation $t$). The definition of a  {\em pure} RNmatrix for \ipl\ (i.e., without depending on G\"odel and Gr\"atz's constructions), which is also decidable by truth-tables, as discussed above, is a challenge that naturally arises.

A pure RNmatrix for \ipl\ with such features is desirable for several reasons. First, it can lead to improvements in efficiency. As will become clear later, the computational cost of producing a truth-table is directly related to the number of subformulas, and the translated formulas are twice as large as the original intuitionistic formulas.  The second reason is that a pure RNmatrix can explain the actual meaning of intuitionistic operators. As it will be argued below, the RNmatrix for \ipl\ proposed here has significant differences when compared to the RNmatrix for \sfo, and those differences are motivated by philosophical considerations on the nature of \ipl\ itself.

In this paper, a pure (and decidable by non-deterministic truth-tables) RNmatrix for \ipl\ will be constructed by abstracting the composition of G\"odel's translation algorithm and Gr\"atz's decision procedure for \sfo. This abstraction will provide a direct way to define the RNmatrix for \ipl. Hence, a $3$-valued Nmatrix, together with suitable notions of level valuations and partial level valuations (different, in nature, to the ones for modal logics) is introduced, and the soundness and completeness of the method are proved without relying on the results of G\"odel~\cite{Godel1933} and Gr\"atz~\cite{gratz_truth_2022}. 

The paper is organized as follows. Section~\ref{sec2} introduces the central notions that will be used, such as the notion of non-deterministic matrices, restricted non-deterministic matrices, and their valuation functions. Section~\ref{sec3} extends G\"odel's uncharacterizability theorem for \ipl\ to the non-deterministic case. Section~\ref{sec4} presents Kearns's Nmatrices, and the notion of level valuation for \sfo\ is detailed. Section~\ref{Sect:Gratz} presents Gr\"atz's RNmatrix for \sfo. Section~\ref{Sect:IPLfirstStep} expands the RNmatrix for \sfo\ with conjunction and disjunction and introduces G\"odel's box-translation. Section~\ref{RNmatrix-IPL} introduces the RNmatrix for \ipl\ and proves its soundness and completeness regarding a new notion of level valuation for \ipl. Section~\ref{tables-IPL} presents the truth-table method based on the RNmatrix, and some examples are discussed. Section~\ref{complexity} briefly considers some aspects of the computational complexity of the method. Finally, Section~\ref{Sect:Related} discusses some related works and possibilities of future research.

\section{Preliminaries: Logics, Nmatrices and RNmatrices}\label{sec2}

Along this paper, a {\em propositional signature} is a denumerable family $\Theta=(\Theta_n)_{n \geq 0}$ of pairwise disjoint sets. Elements in $\Theta_n$ are called {\em $n$-ary connectives}. An {\em algebra over $\Theta$} is a pair $\mathcal{A}=\langle A,\mathcal{O}\rangle$ such that  $A$ is a non-empty set called the {\em universe} or {\em domain} of $\mathcal{A}$ and, $\mathcal{O}$ is a function assigning, to each $n$-ary connective $\#$ of $\Theta$, a function $\mathcal{O}(\#) : A^{n} \rightarrow A$.
The (absolutely) free algebra over $\Theta$ generated by a denumerable set $\mathcal{P}=\{p_0,p_1,\ldots\}$ of propositional variables is called {\em the algebra of formulas over $\Theta$}, and will be denoted by $For(\Theta)$. 

A  {\em logic}  is a pair ${\bf L}=\langle For,\vdash_{\bf L}\rangle$ where $For$ is the set of formulas and $\vdash_{\bf L}$ is the consequence relation. A logic {\bf L} is said to be {\em Tarskian} if it satisfies the following:~(1)~if $\varphi \in \Gamma$ then $\Gamma \vdash_{\bf L} \varphi$; (2)~if $\Gamma \vdash_{\bf L} \varphi$ and $\Gamma \subseteq \Delta$ then $\Delta \vdash_{\bf L} \varphi$; and~(3)~if $\Delta \vdash_{\bf L} \varphi$ and $\Gamma \vdash_{\bf L} \psi$ for every $\psi \in \Delta$ then $\Gamma \vdash_{\bf L} \varphi$. A Tarskian logic {\bf L} is  {\em finitary} if it satisfies: (4)~if $\Gamma \vdash_{\bf L} \varphi$ then there exists  a finite subset $\Gamma_0$ of $\Gamma$ such that $\Gamma_0 \vdash_{\bf L} \varphi$.

A very natural way to characterize (Tarskian) logics is by means of logical matrices, which formalize the idea of truth-tables:

   \begin{definition}[Logical matrices]
        Let $\Theta$ be a signature. A {\em   logical matrix} over $\Theta$ is a triple $\mathcal{M} = \langle \mathcal{V},\mathcal{D},\mathcal{O} \rangle$ where 
        \begin{enumerate}
            \item $\mathcal{V}$ is a non-empty set (of truth-values);
            \item $\mathcal{D} \subseteq \mathcal{V}$ is a non-empty set (of designated values);
            \item $\mathcal{O}$ is a function assigning, to each $\# \in\Theta_n$, a truth function $\tilde {\#} : \mathcal{V}^{n} \rightarrow \mathcal{V}$ (that is, $\mathcal{A}_\mathcal{M}=\langle \mathcal{V},\mathcal{O} \rangle$ is an algebra over $\Theta$).
        \end{enumerate}
\end{definition} 

\begin{definition} \label{conseq-mat}        
    A  valuation over a matrix $\mathcal{M}$ is a homomorphism $v:For(\Theta) \to \mathcal{A}_\mathcal{M}$. That is, for every $\# \in\Theta_n$ and for every $\varphi_1,\ldots,\varphi_n \in For(\Theta)$: $v(\#(\varphi_1, \ldots,\varphi_n))=\mathcal{O}(\#)(v(\varphi_1), \ldots,v(\varphi_n))$. The set of valuations over $\mathcal{M}$  is denoted by $Val(\mathcal{M})$.  We say that $\varphi$ is a semantical consequence of $\Gamma$ w.r.t. $\mathcal{M}$, denoted by $\Gamma \models_{\mathcal{M}} \varphi$ if, for every $v \in Val(\mathcal{M})$: $v(\varphi) \in \mathcal{D}$ whenever $v(\gamma) \in \mathcal{D}$ for every $\gamma \in \Gamma$. We write $\models_{\mathcal{M}} \varphi$ instead of  $\emptyset\models_{\mathcal{M}} \varphi$.
    \end{definition}

As mentioned in the Introduction, the formal notion of non-deterministic matrices (Nmatrices) was introduced by Avron and Lev in~\cite{avr:lev:01} as a way to overcome some G\"odel-Dugundji's uncharacterizability results. Finite-valued Nmatrices which are analytical (i.e., each partial valuation can be extended to a full valuation) provide a decision procedure for some logics which are uncharacterizable by a single finite-valued logic, such as some paraconsistent logics (see, for instance, \cite{avron:2007}).

\begin{definition}[Nmatrices,  \cite{avr:lev:01}] \label{def:Nmatrix}
    Let $\Theta$ be a signature. A {\em non-deterministic matrix} (Nmatrix, for short) over $\Theta$ is a triple $\mathcal{M} = \langle \mathcal{V},\mathcal{D},\mathcal{O} \rangle$ where 
    \begin{enumerate}
        \item $\mathcal{V}$ is a non-empty set (of truth-values);
        \item $\mathcal{D} \subseteq \mathcal{V}$ is a non-empty set (of designated values);
        \item $\mathcal{O}$ is a function assigning, to each $\# \in\Theta_n$, a  (non-deterministic) truth function $\mathcal{O}(\#) : \mathcal{V}^{n} \rightarrow (\wp(\mathcal{V})) \setminus \emptyset$  (that is, $\mathcal{A}_\mathcal{M}=\langle \mathcal{V},\mathcal{O} \rangle$ is a non-deterministic algebra over $\Theta$).
    \end{enumerate}
\end{definition} 

\begin{definition}[Valuations on Nmatrices,  \cite{avr:lev:01}] \label{conseq-Nmat}
    A  valuation over an Nmatrix $\mathcal{M}$ is a function $v:For(\Theta) \to \mathcal{V}$ such that, for every $\# \in\Theta_n$ and for every $\varphi_1,\ldots,\varphi_n \in For(\Theta)$, $v(\#(\varphi_1, \ldots,\varphi_n)) \in \mathcal{O}(\#)(v(\varphi_1), \ldots,v(\varphi_n))$. The set of valuations over $\mathcal{M}$  is denoted by $Val^{nd}(\mathcal{M})$. The notion of semantical consequence w.r.t. an Nmatrix $\mathcal{M}$, as well as the notation, will be analogous to the ones in Definition~\ref{conseq-mat}, but now w.r.t. valuations over the Nmatrix $\mathcal{M}$.
\end{definition}

However, as mentioned before, certain logics are still uncharacterizable by finite-valued Nmatrices, such as  
da Costa's paraconsistent logic $C_1$ (see~\cite{avron:2007}), modal logics \textbf{KT} and \textbf{S4} (see~\cite{gratz_truth_2022}), or \ipl\ (see Theorem~\ref{unchar-IPL} below). In these cases, finite-valued and decidable RNmatrices, which are analytical and co-analytical, can overcome this limitation.

\begin{definition}[RNmatrices, \cite{coniglio_two_2022}] \label{def:RNmatrix}
    Let $\Theta$ be a signature. An  RNmatrix over $\Theta$ is a pair $\mathcal{R} = \langle \mathcal{M}, \mathcal{F} \rangle$ where $\mathcal{M}$ is an Nmatrix over $\Theta$ and $\mathcal{F} \subseteq Val^{nd}(\mathcal{M})$.  The notion of semantical consequence w.r.t. an RNmatrix $\mathcal{R} = \langle\mathcal{M}, \mathcal{F}\rangle$, as well as the notation, will be analogous to the ones in Definition~\ref{conseq-mat}, but now w.r.t. valuations in $\mathcal{F}$.
\end{definition} 

It is worth noting that the logic generated by an RNmatrix is indeed Tarskian.

\begin{table}
    \centering
    \begin{tabular}{|c|c|}
        \hline
          & $\neg^{S4}$ \\ \hline \hline
        0 & $\{ 1, 2 \}$              \\ \hline
        1 & $\{ 0 \}$                 \\ \hline
        2 & $\{ 0 \}$                 \\ \hline
    \end{tabular}
    \begin{tabular}{|c|c|}
        \hline
          & $\Box^{S4}$ \\ \hline\hline
        0 & $\{ 0 \}$                 \\ \hline
        1 & $\{ 0 \}$                 \\ \hline
        2 & $\{ 2 \}$                 \\ \hline
    \end{tabular}
    \begin{tabular}{|c|c|c|c|}
        \hline
        $\rightarrow^{S4}$ & 0           & 1           & 2           \\ \hline\hline
        0                                & $\{ 1,2 \}$ & $\{ 1,2 \}$ & $\{ 1,2 \}$ \\ \hline
        1                                & $\{ 0 \}$   & $\{ 1,2 \}$ & $\{ 1,2 \}$ \\ \hline
        2                                & $\{ 0 \}$   & $\{ 1 \}$   & $\{ 1,2 \}$ \\ \hline
    \end{tabular}
    \begin{tabular}{|c|c|c|c|}
        \hline
        $\lor^{S4}$ & 0           & 1           & 2           \\ \hline\hline
        0                                & $\{ 0 \}$ & $\{ 1,2 \}$ & $\{ 1,2 \}$ \\ \hline
        1                                & $\{ 1,2 \}$   & $\{ 1,2 \}$ & $\{ 1,2 \}$ \\ \hline
        2                                & $\{ 1,2 \}$   & $\{ 1,2 \}$   & $\{ 1,2 \}$ \\ \hline
    \end{tabular}
    \begin{tabular}{|c|c|c|c|}
        \hline
        $\land^{S4}$ & 0           & 1           & 2           \\ \hline\hline
        0                                & $\{ 0 \}$ & $\{ 0 \}$ & $\{ 0 \}$ \\ \hline
        1                                & $\{ 0 \}$   & $\{ 1,2 \}$ & $\{ 1,2 \}$ \\ \hline
        2                                & $\{ 0 \}$   & $\{ 1,2 \}$   & $\{ 1,2 \}$ \\ \hline
    \end{tabular}
    \caption{Tables for $\mathcal{M}_{S4}$.}\label{table:nm_S4*}
\end{table}

\begin{example}[Gr\"atz's Nmatrix for \textbf{S4}]\label{def:gratz_nmatrix_s4}
    Let $\Sigma_b = \{ \{ \neg, \Box \}, \{\rightarrow\} \}$ and $\Sigma = \{ \{ \neg, \Box \}, \{\rightarrow, \lor, \land \} \}$ be, respectively, {\em  the basic signature} and  {\em  the signature}  for modal logic \sfo. Given the set $\mathcal{P}$ of propositional variables, let $For(\Sigma_b)$ and  $For(\Sigma)$ be the corresponding propositional languages. Gr\"atz's Nmatrix over $\Sigma_b$ is the tuple $\mathcal{M}_{S4} = \langle  \mathcal{V}, \mathcal{D}, \mathcal{O} \rangle$, where $\mathcal{V} = \{ 0, 1, 2 \}$, $\mathcal{D} = \{ 1, 2\}$, $\mathcal{O} = \{ \neg^{S4}, \Box^{S4} , \rightarrow^{S4} \}$ and each $c \in \mathcal{O}$ is defined in Table~\ref{table:nm_S4*}. We will return to this Nmatrix in Section~\ref{Sect:Gratz}.
\end{example}

By setting $\varphi \lor \psi := \neg \varphi \to \psi $ and $\varphi \land \psi := \neg (\varphi \to \neg \psi)$, respectively, the multioperators $\lor^{S4}$ and $\land^{S4}$ are defined as in Table~\ref{table:nm_S4*}.

As proven in~\cite{gratz_truth_2022}, no finite Nmatrix can characterize \sfo; in particular, $\mathcal{M}_{S4}$ cannot characterize this modal logic. In Section~\ref{Sect:Gratz} we will analyze an RNmatrix based on $\mathcal{M}_{S4}$ which characterizes \sfo.


\section{G\"odel's uncharacterizability theorem updated}\label{sec3}

As mentioned in the Introduction, Gödel proved in $1932$  that \ipl\ cannot be characterized by a single finite-valued (deterministic) logical matrix. It is a natural question to determine whether \ipl\ could be characterized by means of a single finite-valued non-deterministic matrix. As we shall show in the sequel, it is easy to adapt the proof by G\"odel in~\cite{Godel1932} in order to see that the answer to the question above is negative.

\begin{definition} \label{int-sig}
    Consider the signature $\Omega =(\{ \neg \}, \{ \land, \lor, \rightarrow \} )$ for intuitionistic propositional logic. The set of intuitionistic formulas will be denoted by $For(\Omega)$.
\end{definition}

\begin{definition}  [G\"odel~\cite{Godel1932}]
For $n \geq 1$ let $G_n$ be the formula  expressing that ``there are at most $n$ truth-values'':
$$G_n := \bigvee_{1 \leq i < j \leq n+1} (p_i \to p_j) \land (p_j \to p_i).$$
\end{definition}

\begin{lemma} [G\"odel~\cite{Godel1932}] \label{lem-god-0}
None of the formulas $G_n$ is provable in \textbf{IPL}.
\end{lemma}

\begin{lemma} \label{lem-god-1}
If $\mathcal{M} = \langle \mathcal{V},\mathcal{D},\mathcal{O} \rangle$ is a finite Nmatrix such that $\mathcal{V}$ has exactly $n\geq 1$ elements and $\mathcal{M}$ models \textbf{IPL}, then the formula $G_n$ is valid in  $\mathcal{M}$.
\end{lemma}
\begin{proof} For each connective $\# \in \Omega$ let us denote $\mathcal{O}(\#)$ by $\bar{\#}$.
Since $\mathcal{M}$ models \textbf{IPL} then, for every $a,b\in \mathcal{V}$: (i)~$a \,\bar{\vee}\, b \subseteq \mathcal{D}$ if either $a \in \mathcal{D}$ or $b \in \mathcal{D}$;  (ii)~$a \,\bar{\land}\, b \subseteq \mathcal{D}$ if both $a \in \mathcal{D}$ and $b \in \mathcal{D}$; and  (iii)~$a \,\bar{\to}\, a \subseteq \mathcal{D}$ (since $\varphi \to \varphi$ is a theorem of  \textbf{IPL}). Given that $\mathcal{V}$ has exactly $n\geq 1$ elements, if $v$ is a valuation over $\mathcal{M}$ then, by the pigeonhole principle,  there exists $1 \leq i < j \leq n+1$ such that $v(p_i)=v(p_j)$. Let $a:=v(p_i)$. By consideration (iii) above, $v(p_i \to p_j) \in a \,\bar{\to}\, a \subseteq \mathcal{D}$, hence $v(p_i \to p_j) \in  \mathcal{D}$. Analogously, $v(p_j \to p_i) \in  \mathcal{D}$ and so, by (ii), $v((p_i \to p_j) \land (p_j \to p_i)) \in  \mathcal{D}$. Thus, from the fact that $G_n$ is a disjunction where $(p_i \to p_j) \land (p_j \to p_i)$ is one of its disjuncts, repeated application of (i) implies $v(G_n) \in \mathcal{D}$.
\end{proof}


\begin{theorem} \label{unchar-IPL}
The logic \textbf{IPL} cannot be characterized by a single finite-valued Nmatrix.
\end{theorem}
\begin{proof}
Suppose that $\mathcal{M}$ is a finite Nmatrix with exactly $n\geq 1$ elements such that $\mathcal{M}$ models \textbf{IPL}. Suppose, by absurdum, that $\mathcal{M}$ characterizes \textbf{IPL}. Hence, every formula valid in $\mathcal{M}$ can be proved in \textbf{IPL}. By Lemma~\ref{lem-god-1}, the formula $G_n$ is valid in  $\mathcal{M}$. By hypothesis, $G_n$ must be provable in \textbf{IPL}. But this contradicts Lemma~\ref{lem-god-0}. This shows that no finite Nmatrix $\mathcal{M}$ can characterize \textbf{IPL}.
\end{proof}

\section{Kearns's characterization of \textbf{S4} by RNmatrices}\label{sec4}

Regarding semantics for normal modal logic, Kripke-style semantics are the usual choice. However, there are other attractive alternatives, such as algebraic semantics.
In the 1980s, Kearns  proposed in~\cite{kearns_modal_1981} an RNmatrix-based semantics for modal logics \textbf{KT}, \textbf{S4} and \textbf{S5}. He introduced the level valuations as the set $\mathcal{F}$ of valuations to be used in his RNmatrix. According to him,

\begin{quote}
    [It] is simpler than the standard [Kripke's] account in virtue of having dispensed with possible worlds and their relations. I also think that my account is philosophically preferable to the standard account for having done this. For I do not think there are such things as possible worlds, or even that they constitute a useful fiction. (Kearns, \cite[p. 86]{kearns_modal_1981})
\end{quote}

Figure~\ref{table:Ke} reconstructs Kearns's Nmatrix for \sfo\ over the signature $\Sigma_{Ke}=(\{\Box,\neg\},\{\vee\})$, where the truth-values are: $T$ (necessarily true), $t$ (contingently true), $f$ (contingently false) and $F$ (necessarily false). In this Nmatrix, the set of designated values is given by $\mathcal{D} = \{ T, t \}$.

\begin{figure}[h!]
    \centering
    \includegraphics[scale=0.5]{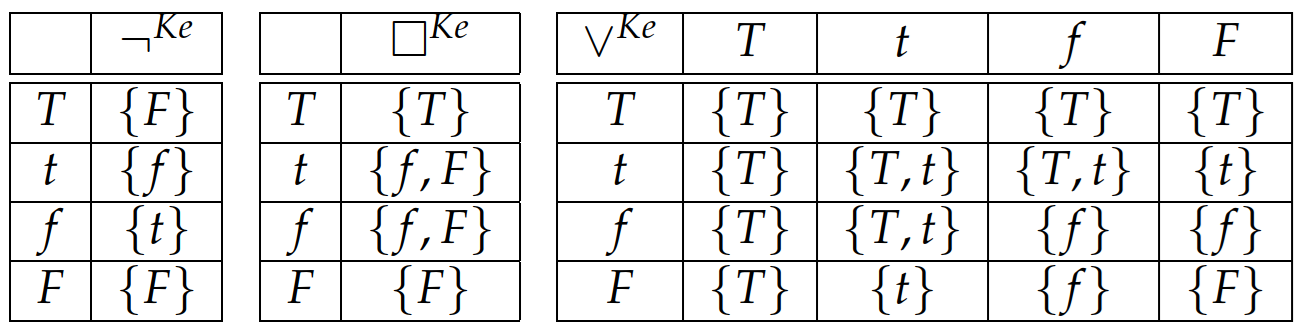}
    \caption{Tables for $\mathcal{M}_{Ke}$.}\label{table:Ke}
\end{figure}


Kearns proved that, in order to fully characterize \sfo\ using $\mathcal{M}_{Ke}$, it is sufficient to introduce the notion of \emph{level valuations}, which are defined as follows:

\begin{definition}\label{def:level_valuations_1}(Kearns's level valuations for \textbf{S4})

    Let $Val^{nd}(\mathcal{M}_{Ke})$ be the set of  valuations over $\mathcal{M}_{Ke}$. We define the set of $n$th-level valuations $\mathcal{L}^{Ke}_k$, where $n \in \omega$, as follows:

    \begin{enumerate}[label=\roman*)]
        \item $\mathcal{L}_0^{Ke} = Val^{nd}(\mathcal{M}_{Ke})$;
        \item $\mathcal{L}^{Ke}_{n + 1} = \{ v \in \mathcal{L}^{Ke}_n \mid \forall \alpha \in For(\Sigma_{Ke}), \mathcal{L}^{Ke}_n(\alpha) \subseteq \{T,t\} \Rightarrow v(\alpha) = T \}$
    \end{enumerate}

    where $\mathcal{L}^{Ke}_n(\alpha) = \{ v(\alpha) \mid v \in \mathcal{L}^{Ke}_n \}$.
    
    The set of level valuations in $\mathcal{M}_{Ke}$ is the intersection of the sets $\mathcal{L}^{Ke}_n$, for $n\geq0$:

    \[
        \mathcal{L}_{Ke} = \bigcap^\infty_{n \geq 0} \mathcal{L}^{Ke}_n.
    \]

\end{definition}

The way in which level valuations work can be appreciated in Figure~\ref{lev-val}.

\begin{figure}[h!]
    \centering
    \includegraphics[scale=0.3]{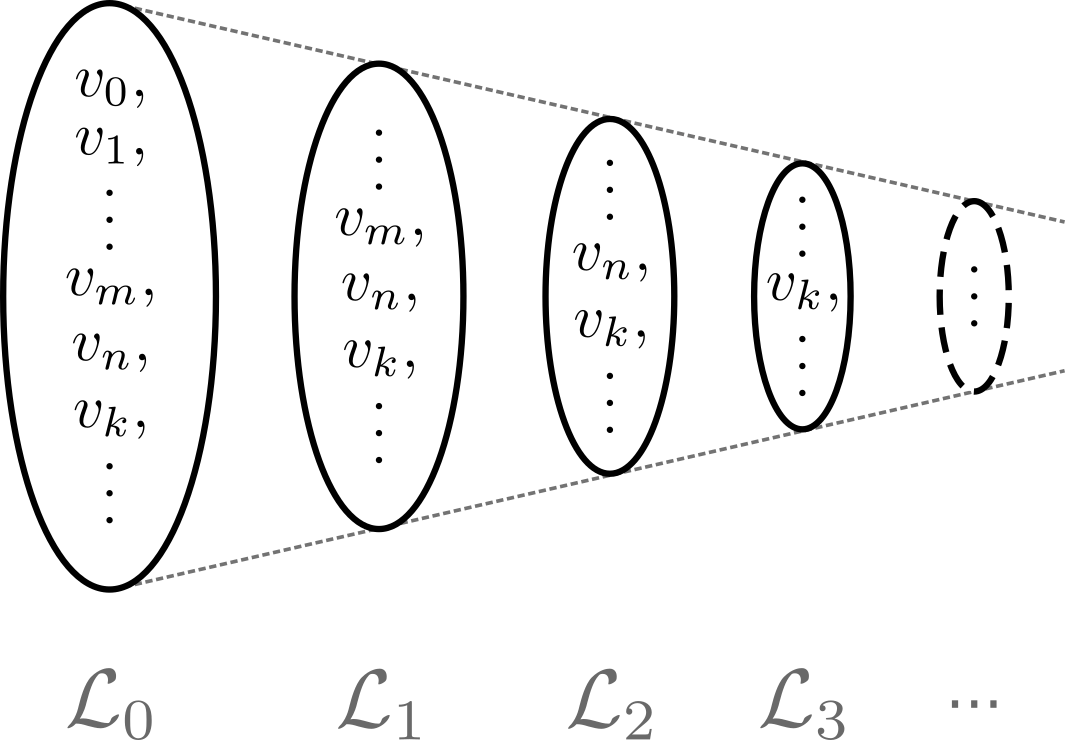}
    \caption{ Assume that $Val^{nd}(\mathcal{M}_{Ke}) = \{ v_0, v_1, \ldots, v_m, v_n, v_k, \ldots \}$. We remove a valuation $v$ in level $\mathcal{L}_{n+1}$ iff there is some  $\alpha$ such that $v(\alpha) = t$ and $w(\alpha) \in \{T,t\}$ for every $w \in \mathcal{L}_n$.}\label{lev-val}
\end{figure}

\begin{definition}[Kearns's RNmatrix for \textbf{S4}]\label{def:kearns_Rnmatrix_s4}
    The RNmatrix for \textbf{S4} defined by Kearns is the pair $\mathcal{R}(\mathcal{M}_{Ke}) = \langle \mathcal{M}_{Ke}, \mathcal{L}_{Ke}\rangle$.
\end{definition}

\begin{theorem} [\cite{kearns_modal_1981}] \label{Theor:sound_compl_S4K} For every $\Gamma$ and $\varphi$ over $\Sigma_{Ke}$: $\Gamma \vdash_{S4} \varphi$ iff $\Gamma  \models_{\mathcal{R}(\mathcal{M}_{Ke})} \varphi$.
\end{theorem}

\section{Gr\"atz's decision procedure for \sfo} \label{Sect:Gratz}

Kearns proved that his Nmatrix for \sfo\ with level valuations is sound and complete regarding the logic \sfo. However, as Section~\ref{sec4} explains, level valuations are defined as the intersection of an infinite family of subsets of valuation functions. How do we reproduce this procedure using finite steps? While it is true that level valuation semantics characterizes \sfo, it is not straightforward to decide the validity of a given formula using Kearns's tables.

In particular, given any theorem $\varphi$, the immediate table for $\Box \varphi$ will always present a partial valuation $v$ such that $v(\Box \varphi)$ is not a designated value, thus invalidating the rule of necessitation. Gr\"atz noticed that the problem lies in the definition of \textit{partial valuation}: such partial valuations are precisely the ones that cannot be extended to full level valuations; thus, they should be ruled out by some appropriate definition. 

\paragraph{Gr\"atz's Nmatrix for \sfo:} Instead of $4$ truth-values, Gr\"atz use a set with three values, $0, 1$ and $2$, where each one has the following intended meaning:

\begin{description}
    \item[$v(\varphi) = 0$] means that $\varphi$ is false;
    \item[$v(\varphi) = 1$] means that $\varphi$ is contingently true, i.e, $\varphi$ is true but possibly false;
    \item[$v(\varphi) = 2$] means that $\varphi$ is necessarily true.
\end{description}

From this interpretation, defining the subset $\mathcal{D} = \{ 2, 1 \}$ as the set of designated values is natural. As primitive operators, it is used negation ($\neg$), implication ($\to$), and box ($\Box$). The Nmatrix  $\mathcal{M}_{S4}$ over  $\Sigma_b = ( \{ \neg, \Box \}, \{\rightarrow\} )$ is defined as  in Example~\ref{def:gratz_nmatrix_s4} (see also Table~\ref{table:nm_S4*}).



To understand the Nmatrix for the modal operator $\Box$, we must recall the axioms that define \sfo. In particular, the axioms $\Box \varphi \to \varphi$ (reflexivity) and $\Box \varphi \to \Box \Box \varphi$ (transitivity). Consider the contrapositive of the reflexivity axiom: $\neg \varphi \to \neg \Box \varphi$. This means that if $\varphi$ is false, then $\Box \varphi$ is also false. The Nmatrix for $\Box$ reflects this behavior when $v(\varphi) = 0$. Next, when $v(\varphi) = 2$ (indicating that $\varphi$ is necessary, or in other words, $\Box \varphi$ is true), we have $v(\Box \varphi) = \{2\}$. This signifies that $\Box \varphi$ is also necessary, which implies that $\Box \Box \varphi$ is true as well. Finally, when $v(\varphi) = 1$, it indicates that $\varphi$ is contingent, meaning it is possibly false, hence $\Box \varphi$ is false.

\begin{definition}\label{def:level_valuations_2}(Gr\"atz's level valuations for \textbf{S4})

    Let $Val^{nd}(\mathcal{M}_{S4})$ be the set of  valuations over $\mathcal{M}_{S4}$. We define the set of $n$th-level valuations $\mathcal{L}^{S4}_k$, where $n \in \omega$, as follows:

    \begin{enumerate}[label=\roman*)]
        \item $\mathcal{L}_0^{S4} = Val^{nd}(\mathcal{M}_{S4})$;
        \item $\mathcal{L}^{S4}_{n + 1} = \{ v \in \mathcal{L}^{S4}_n \mid \forall \alpha \in For(\Sigma_{b}), \mathcal{L}^{S4}_n(\alpha) \subseteq \{1,2\} \Rightarrow v(\alpha) = 2 \}$
    \end{enumerate}

    where $\mathcal{L}^{S4}_n(\alpha) = \{ v(\alpha) \mid v \in \mathcal{L}^{S4}_n \}$.
    
    The set of level valuations in $\mathcal{M}_{S4}$ is the intersection of the sets $\mathcal{L}^{S4}_n$, for $n\geq0$:

    \[
        \mathcal{L}_{S4} = \bigcap^\infty_{n \geq 0} \mathcal{L}^{S4}_n.
    \]
\end{definition}

\begin{definition}[Gr\"atz's RNmatrix for \textbf{S4}]\label{def:gratz_Rnmatrix_s4}
    The RNmatrix for \textbf{S4} defined by Gr\"atz is the pair $\mathcal{R}(\mathcal{M}_{S4}) = \langle \mathcal{M}_{S4}, \mathcal{L}_{S4}\rangle$.
\end{definition}

Let $\vdash_{S4}^b$ be the consequence relation of the standard Hilbert calculus $\mathcal{H}^b_{S4}$ for \textbf{S4} over $\Sigma_b$. As usual, $\Gamma \vdash_{S4}^b \varphi$ iff either $\varphi$ is a theorem of $\mathcal{H}^b_{S4}$ or there exists a non-empty set $\{\beta_1,\ldots,\beta_n\} \subseteq \Gamma$ such that $\beta_1 \to (\beta_2 \to(\ldots \to(\beta_n \to \varphi) \ldots))$ is a theorem of $\mathcal{H}^b_{S4}$.

\begin{theorem} {\em(\cite[Theorems~3.5 and~3.10]{gratz_truth_2022})} \label{Theor:sound_compl_S4}  For every $\Gamma$ and $\varphi$ over $\Sigma_{b}$: $\Gamma \vdash_{S4}^b \varphi$ iff $\Gamma  \models_{\mathcal{R}(\mathcal{M}_{S4})} \varphi$.
\end{theorem}

In order to obtain a decision procedure for \textbf{S4} based on the $3$-valued RNmatrix $\mathcal{R}(\mathcal{M}_{S4})$, Gr\"atz proposes the notion of partial level valuation. The idea is as follows: given a formula $\varphi$ to be checked as a tautology in \textbf{S4}, construct the (finite) table generated by the Nmatrix $\mathcal{M}_{S4}$ for $\varphi$ and all of its subformulas. Then, by means of a decision procedure, it is possible to determine which rows of the table are allowed (since they correspond to level valuations); the non-allowed ones correspond to valuations over $\mathcal{M}_{S4}$ which {\em do not} correspond to level valuations, hence they are deleted. If $\varphi$ only receives a designated value in the allowed rows, then it is declared to be valid in \textbf{S4}, otherwise it is declared to be non-valid. Formally:

\begin{definition}[Partial valuation  in $\mathcal{M}_{S4}$]\label{def:partial_valuation0}
A partial valuation in $\mathcal{M}_{S4}$ is a function $\tilde v_p :\Lambda \rightarrow \mathcal{V}$ such that  $\Lambda \subseteq For(\Sigma_b)$ is  closed under subformulas\footnote{That is: if $\alpha \in \Lambda$ and $\beta$ is  a subformula of $\alpha$ then $\beta \in \Lambda$.} and, for every $\alpha,\beta \in \Lambda$:
 
 \begin{itemize}
 \item[--] if $\# \in \{\neg,\Box\}$ and  $\#\alpha \in \Lambda$ then $\tilde v_p(\#\alpha) \in \#^{S4}(\tilde v_p(\alpha))$;
 \item[--] if  $\alpha \to \beta\in\Lambda$ then  $\tilde v_p(\alpha \to \beta) \in {\to}^{S4}(\tilde v_p(\alpha),\tilde v_p(\beta))$.
 \end{itemize}
\end{definition}

\begin{definition}[Gr\"atz's partial level valuation]\label{def:partial_level_val_S4}
    Let   $\Lambda \subseteq For(\Sigma_b)$ be a  set of formulas closed under subformulas. A partial valuation $v_p : \Lambda \rightarrow \{ 0, 1, 2 \}$ in $\mathcal{M}_{S4}$ is a {\em partial level valuation over $\Lambda$} in $\mathcal{R}(\mathcal{M}_{S4})$ ($v_p \in PLV$) iff

    \begin{enumerate}
        \item $\forall \alpha \in \Lambda$ such that $v_p(\alpha) = 1$, there exists $w_p\in PLV$ such that $w_p(\alpha) = 0$ and, $\forall \beta \in \Lambda$, $w_p(\beta) = 2$ whenever $v_p(\beta) = 2$.
    \end{enumerate}
\end{definition}

\begin{remark} \label{rem:PLVs} The set $PLV$ plays a fundamental role in the proof of soundness and completeness of the truth-tables decision procedure introduced in~\cite{gratz_truth_2022}. Some comments on this notion are in order at this point. 

The existence (and uniqueness) of the set $PLV$ is the crucial fact used in his proof. In this respect, \cite[p. 15]{gratz_truth_2022} presents an algorithm (Algorithm~3.4)  for computing the set  $PLV$ of partial level valuations as a (constructive) truth-table. He argues that this process is well-defined since the set of all partial valuations with domain $\Lambda$ is finite, so is its powerset. Uniqueness is tacitly assumed, given the nature of the algorithm. Moreover, he proves that this notion is analytic in $\mathcal{M}_{S4}$ (i.e., every partial level valuation in $\mathcal{M}_{S4}$ can be extended to a level valuation in $\mathcal{M}_{S4}$) and co-analytic (i.e., for any level valuation $v$ in $\mathcal{M}_{S4}$, the restriction of $v$ to a partial level valuation $v_p$ defined over a finite set $\Lambda$ closed under subformulas is a partial level valuation). This argument, justified by the uniqueness of the set $PLV$, allows us to conclude that the truth-table procedure corresponds exactly to the level valuation semantics. A similar and fully detailed argument will be given for  \ipl\ in Section~\ref{tables-IPL}.
\end{remark}


From the observations above, one can easily prove that the truth-table method for \textbf{S4} is sound and complete, that is:

\begin{theorem}   {\em(\cite[Theorem~3.15]{gratz_truth_2022})} \label{theorem:gratz_an_coan}
Let $\varphi$ be a formula over $\Sigma_b$. Then,    $\models_{\mathcal{R}(\mathcal{M}_{S4})} \varphi$ if, and only if, for every partial valuation $v_p$ in $\mathcal{M}_{S4}$ defined over the set $\Lambda$ of subformulas of $\varphi$: if $v_p \in PLV$, then $v_p(\varphi) \in \mathcal{D}$.
\end{theorem}

\begin{corollary}  \label{Theor:sound_compl_S4-1} For every $\varphi$ over $\Sigma_b$: $\vdash_{S4}^b \varphi$ if, and only if, for every partial valuation $v_p$ in $\mathcal{M}_{S4}$ defined over the set $\Lambda$ of subformulas of $\varphi$, if $v_p \in PLV$ then $v_p(\alpha) \in \mathcal{D}$.
\end{corollary}
\begin{proof}
It is an immediate consequence of Theorems~\ref{theorem:gratz_an_coan} and~\ref{Theor:sound_compl_S4}.
\end{proof}

Therefore, to construct a truth-table, compute every partial valuation and remove any non-$PLV$ in a finite number of steps. This algorithmic process is a sound and complete decision procedure for \sfo. Clearly, the soundness and completeness Theorem~\ref{theorem:gratz_an_coan} (and hence the decision procedure) can be extended to inferences in \sfo\ from a finite set of premises, since $\vdash_{S4}^b$ (and thus $\models_{\mathcal{R}(\mathcal{M}_{S4})}$) satisfies the deduction metatheorem.

\section{Reducing Gr\"atz's RNmatrix and algorithm for \sfo} \label{sect:redu_RNmat_S4}

As suggested by Gr\"atz (\cite{gratz_truth_2022}, p. 17), it is possible to apply the notion of partial level valuation to the tables for the connectives of  \textbf{S4} (in our case, this includes conjunction and disjunction). Then, by eliminating rows $v_p$ which are not supported by any other $w_p$ as prescribed in Definition~\ref{def:partial_level_val_S4}, it is obtained a reduced truth-table for each multioperator of  \textbf{S4} over the full signature $\Sigma$. Along this section we will deal exclusively with signature $\Sigma$ for  \textbf{S4}.

It is easy to see that the reduced Nmatrix for  \textbf{S4} over signature $\Sigma$ (recall Example~\ref{def:gratz_nmatrix_s4} and Table~\ref{table:nm_S4*}) obtained in this way is $\mathcal{M}'_{S4} = \langle  \mathcal{V}, \mathcal{D}, \mathcal{O}' \rangle$, where $\mathcal{O}' = \{ \neg^{S4}, \Box^{S4}, \rightarrow'^{S4}, \lor'^{S4}, \land'^{S4} \}$ such that the tables for $\#'^{S4}$, for $\# \in \{\to, \vee, \land\}$, are displayed in Table~\ref{table:nm_S4*-R}.

\begin{table}
    \centering
    \begin{tabular}{|c|c|}
        \hline
          & $\neg^{S4}$ \\ \hline \hline
        0 & $\{ 1, 2 \}$              \\ \hline
        1 & $\{ 0 \}$                 \\ \hline
        2 & $\{ 0 \}$                 \\ \hline
    \end{tabular}
    \begin{tabular}{|c|c|}
        \hline
          & $\Box^{S4}$ \\ \hline\hline
        0 & $\{ 0 \}$                 \\ \hline
        1 & $\{ 0 \}$                 \\ \hline
        2 & $\{ 2 \}$                 \\ \hline
    \end{tabular}
    \begin{tabular}{|c|c|c|c|}
        \hline
        $\rightarrow'^{S4}$ & 0           & 1           & 2           \\ \hline\hline
        0                                & $\{ 1,2 \}$ & $\{ 1,2 \}$ & $\{ 2 \}$ \\ \hline
        1                                & $\{ 0 \}$   & $\{ 1,2 \}$ & $\{ 2 \}$ \\ \hline
        2                                & $\{ 0 \}$   & $\{ 1 \}$   & $\{ 2 \}$ \\ \hline
    \end{tabular}
    \begin{tabular}{|c|c|c|c|}
        \hline
        $\lor'^{S4}$ & 0           & 1           & 2           \\ \hline\hline
        0                                & $\{ 0 \}$ & $\{ 1,2 \}$ & $\{ 2 \}$ \\ \hline
        1                                & $\{ 1,2 \}$   & $\{ 1,2 \}$ & $\{ 2 \}$ \\ \hline
        2                                & $\{ 2 \}$   & $\{2 \}$   & $\{ 2 \}$ \\ \hline
    \end{tabular}  
    \begin{tabular}{|c|c|c|c|}
        \hline
        $\land'^{S4}$ & 0           & 1           & 2           \\ \hline\hline
        0                                & $\{ 0 \}$ & $\{ 0 \}$ & $\{ 0 \}$ \\ \hline
        1                                & $\{ 0 \}$   & $\{ 1 \}$ & $\{ 1 \}$ \\ \hline
        2                                & $\{ 0 \}$   & $\{ 1 \}$   & $\{ 2 \}$ \\ \hline
    \end{tabular}
    \caption{Tables for $\mathcal{M}'_{S4}$.}\label{table:nm_S4*-R}
\end{table}

Note that the Nmatrices in Table~\ref{table:nm_S4*-R} were obtained by computing truth-tables, and therefore filtered sets of partial valuations. Because of this, some rows were excluded. In particular, if $v(\beta) = 2$ and $v(\alpha \to \beta) = 1$, we cannot find a partial valuation $w$ such that $w(\alpha \to \beta) = 0$ and $w(\beta) = 2$. Therefore, such a partial valuation cannot be a partial level valuation and we can safely remove it from the definition of implication. This was also discussed in~\cite[p. 17]{gratz_truth_2022}.

Of course considering the reduced Nmatrix over signature $\Sigma$ requires adjusting the soundness and completeness proofs of Gr\"atz to the new RNmatrix defined over $\Sigma$. Let $\vdash_{S4}$ be the consequence relation of the standard Hilbert calculus $\mathcal{H}_{S4}$ for \sfo\ over $\Sigma$. The notion $\Gamma \vdash_{S4} \varphi$ of derivations  from premises in  $\mathcal{H}_{S4}$ is defined, as in the case of $\mathcal{H}^b_{S4}$, in terms of theoremhood.

Let  $Val(\mathcal{M}'_{S4})$ is the set of  valuations over $\mathcal{M}'_{S4}$.  Consider the family of level $\mathcal{L}'^{S4}_k$ ($k \geq 0$) as in Definition~\ref{def:level_valuations_1}, but now starting from ${\mathcal{L}'}_0^{S4} = Val(\mathcal{M}'_{S4})$ in item~i), and considering formulas $\alpha \in For(\Sigma)$ in item~ii). Let $\mathcal{L}'_{S4} = \bigcap^\infty_{k \geq 0} \mathcal{L}'^{S4}_k$ and $\mathcal{R}(\mathcal{M}'_{S4}) = \langle \mathcal{M}'_{S4}, \mathcal{L}'_{S4}\rangle$. 

\begin{remark} \label{rem:der:RNmatrices}
Let  $\mathcal{M}$ be an  Nmatrix. By the very definitions it is immediate to see that, if $\mathcal{F} \subseteq \mathcal{F}' \subseteq Val(\mathcal{M})$, then the following holds:  $\Gamma \models_{\langle \mathcal{M},\mathcal{F}'\rangle} \varphi$ implies that  $\Gamma \models_{\langle \mathcal{M},\mathcal{F}\rangle} \varphi$.
\end{remark}

\begin{theorem} [Soundness  of $\mathcal{H}_{S4}$ w.r.t.  $\mathcal{R}(\mathcal{M}'_{S4})$] \label{thm:sound:S4N}
For every $\Gamma$ and $\varphi$: $\Gamma \vdash_{S4} \varphi$ implies that $\Gamma  \models_{\mathcal{R}(\mathcal{M}'_{S4})} \varphi$.
\end{theorem}
\begin{proof} 
It coincides with the proof for $\mathcal{R}(\mathcal{M}_{S4})$ found in~\cite[Theorems~3.4 and~3.5]{gratz_truth_2022}.
The only detail to be checked is that any axiom of  $\mathcal{H}_{S4}$ is still valid in the reduced Nmatrix $\mathcal{M}'_{S4}$. Details are left to the reader.
\end{proof}

\begin{corollary} \label{valid_S4=2}
If $\vdash_{S4} \varphi$ then $v(\varphi)=2$ for every $v \in \mathcal{L}'_{S4}$.
\end{corollary}
\begin{proof}
Assume that $\vdash_{S4} \varphi$, and let $v \in \mathcal{L}'_{S4}$. By Theorem~\ref{thm:sound:S4N}, $v(\varphi) \in \{1,2\}$. Since  $\vdash_{S4} \Box\varphi$ (by necessitation rule) then $v(\Box\varphi) \in \tilde{\Box} v(\varphi)$ such that  $v(\Box\varphi)\in \{1,2\}$, by Theorem~\ref{thm:sound:S4N}. Hence,  $v(\varphi) = 2$.
\end{proof}


The proof of completeness requires the use of a fundamental notion: $\varphi$-saturated sets.

\begin{remark} \label{varphi-sat-sets}
Recall that, given a Tarskian and finitary logic {\bf L} and a formula  $\varphi$, a set $\Delta$ of formulas is said to be {\em  $\varphi$-saturated in {\bf L}}  if $\Delta \nvdash_{\bf L} \varphi$, but $\Delta,\alpha \vdash_{\bf L} \varphi$, whenever $\alpha \notin\Delta$. It is a well-known result that, if $\Gamma \nvdash_{\bf L} \varphi$, there exists a $\varphi$-saturated set $\Delta$ in {\bf L} such that $\Gamma \subseteq \Delta$. If $\Delta$ is $\varphi$-saturated in {\bf L} then it is closed, that is: $\Delta \vdash_{\bf L} \alpha$ iff $\alpha \in \Delta$. In particular, this holds in the logic \sfo, which is generated  by $\mathcal{H}_{S4}$ over $\Sigma$. Since it contains classical logic, any $\varphi$-saturated set $\Delta$ satisfies: $\alpha \vee \beta \in \Delta$ iff $\alpha \in \Delta$ or $\beta \in \Delta$; $\alpha \land \beta \in \Delta$ iff $\alpha \in \Delta$ and $\beta \in \Delta$; and $\alpha \to \beta \in \Delta$ iff either $\alpha \notin \Delta$ or $\beta \in \Delta$. The following result is an adaptation of Lemma~3.7 in~\cite{gratz_truth_2022} to the reduced Nmatrix $\mathcal{M}'_{S4}$:
\end{remark}

\begin{lemma} \label{lemma~3.7-adapted}
Let $\Delta \subseteq For(\Sigma)$ be a $\varphi$-saturated set in the logic generated by $\mathcal{H}_{S4}$ (that is, \sfo). Let $v_\Delta:For(\Sigma) \to \mathcal{V}$ be the function  defined as follows:  \\[1mm]

$v_\Delta(\alpha)= \left \{ \begin{tabular}{rl}
$2$ & if $\Box\alpha \in \Delta$,\\[1mm]
$1$ & if $\Box\alpha \notin \Delta$ but $\alpha \in \Delta$,\\[1mm]
$0$ & if $\Box\alpha \notin \Delta$ and $\alpha \notin \Delta$.\\
\end{tabular}\right.$ \\[2mm]
Then, $v_\Delta$ is a valuation over the Nmatrix $\mathcal{M}'_{S4}$.
\end{lemma}
\begin{proof}
The unary multioperators interpreting $\neg$ and $\Box$ in $\mathcal{M}'_{S4}$ and in the Nmatrix  $\mathcal{M}_{S4}$ for \textbf{S4} considered in~\cite{gratz_truth_2022} coincide, while $\to'^{S4}$ is different to the corresponding multioperator in $\mathcal{M}_{S4}$  (which is precisely $\to^{S4}$). In turn, the multioperators $\lor'^{S4}$ and $\land'^{S4}$ were not considered in~\cite[Lemma~3.7]{gratz_truth_2022}. Thus, it is only required to prove that $v_\Delta(\alpha \# \beta) \in \#'^{S4}(v_\Delta(\alpha),v_\Delta(\beta))$ for every $\alpha,\beta$ and $\# \in \{\lor,\land,\to\}$.\\[1mm]
{\bf Disjunction:} Suppose that $v_\Delta(\alpha \vee \beta)=2$. Then, $\Box(\alpha \vee \beta) \in \Delta$ and so $\alpha \vee \beta \in \Delta$, by axiom T: $\Box \psi \to \psi$ and the fact that $\Delta$ is closed under logical inferences in $\mathcal{H}_{S4}$. This implies that $\alpha \in \Delta$ or $\beta \in \Delta$, hence $v_\Delta(\alpha)  \in \mathcal{D}$ or $v_\Delta(\beta) \in \mathcal{D}$. This implies that $v_\Delta(\alpha \vee \beta)=2 \in  \vee'^{S4}(v_\Delta(\alpha),v_\Delta(\beta))$. Suppose now that $v_\Delta(\alpha \vee \beta)=1$. Then, $\Box(\alpha \vee \beta) \notin \Delta$ but $\alpha \vee \beta \in \Delta$ and so $\alpha \in \Delta$ or $\beta \in \Delta$. Suppose that $\Box\alpha \in \Delta$ or $\Box\beta \in \Delta$. Then, $\Box\alpha \vee \Box\beta \in \Delta$, and so  $\Box(\alpha \vee \beta) \in \Delta$, given that $(\Box\alpha \vee \Box\beta) \to \Box(\alpha \vee \beta)$ is a theorem in  $\mathcal{H}_{S4}$. But this is a contradiction, therefore $\Box\alpha \notin \Delta$ and $\Box\beta \notin \Delta$. This means that $v_\Delta(\alpha)=1$ and $v_\Delta(\beta) \in \{0,1\}$, or vice versa.  Then, $v_\Delta(\alpha \vee \beta)=1 \in  \{1,2\}=\vee'^{S4}(v_\Delta(\alpha),v_\Delta(\beta))$. Finally, if $v_\Delta(\alpha \vee \beta)=0$ then $\alpha \vee \beta \notin \Delta$, hence $\alpha \notin \Delta$ and $\beta \notin \Delta$. This means that $v_\Delta(\alpha)=v_\Delta(\beta)=0$ and so $v_\Delta(\alpha \vee \beta)=0 \in  \{0\}=\vee'^{S4}(v_\Delta(\alpha),v_\Delta(\beta))$.\\[1mm]
{\bf Conjunction:} Suppose that $v_\Delta(\alpha \land \beta)=2$. Then, $\Box(\alpha \land \beta) \in \Delta$ and so  $\Box\alpha \land \Box\beta \in \Delta$, since  $\Box(\alpha \land \beta) \to (\Box\alpha \land \Box\beta)$ is a theorem in  $\mathcal{H}_{S4}$. This implies that $\Box\alpha \in \Delta$ and $\Box\beta \in \Delta$, hence $v_\Delta(\alpha)=v_\Delta(\beta)=2$. Thus, $v_\Delta(\alpha \land \beta)=2 \in \{2\}= \land'^{S4}(v_\Delta(\alpha),v_\Delta(\beta))$. Now, assume  that $v_\Delta(\alpha \land \beta)=1$. Then, $\Box(\alpha \land \beta) \notin \Delta$ but $\alpha \land \beta \in \Delta$ and so $\alpha \in \Delta$ and $\beta \in \Delta$ but either  $\Box\alpha \notin \Delta$ or $\Box\beta \notin \Delta$ (since  $(\Box\alpha \land \Box\beta) \to \Box(\alpha \land \beta)$ is a theorem in  $\mathcal{H}_{S4}$).  That is,  $v_\Delta(\alpha), v_\Delta(\beta) \in \{1,2\}$ and either $v_\Delta(\alpha)=1$ or $v_\Delta(\beta)=1$.  Then, $v_\Delta(\alpha \land \beta)=1 \in  \{1\}=\land'^{S4}(v_\Delta(\alpha),v_\Delta(\beta))$. Finally, if $v_\Delta(\alpha \land \beta)=0$ then $\alpha \land \beta \notin \Delta$, hence $\alpha \notin \Delta$ or $\beta \notin \Delta$. This means that $v_\Delta(\alpha)=0$ or $v_\Delta(\beta)=0$ and so $v_\Delta(\alpha \land \beta)=0 \in  \{0\}=\land'^{S4}(v_\Delta(\alpha),v_\Delta(\beta))$.\\[1mm]
{\bf Implication:} Suppose that $v_\Delta(\alpha \to \beta)=2$. Then, $\Box(\alpha \to \beta) \in \Delta$ and so $\alpha \to \beta \in \Delta$, by axiom T. This implies that either $\alpha \notin \Delta$ or $\beta \in \Delta$. If $\alpha \notin \Delta$ then $v_\Delta(\alpha)=0$. Thus, $v_\Delta(\alpha \to \beta)=2 \in  {\to}'^{S4}(v_\Delta(\alpha),v_\Delta(\beta))$. If  $\beta \in \Delta$, suppose first that $\Box\alpha \in \Delta$. Then, by axiom K it follows that $\Box\beta \in \Delta$. Hence, $v_\Delta(\alpha)=v_\Delta(\beta)=2$ and so $v_\Delta(\alpha \to \beta)=2 \in  \{2\}={\to}'^{S4}(v_\Delta(\alpha),v_\Delta(\beta))$. Otherwise, if $\Box\alpha \notin \Delta$ then  $v_\Delta(\alpha)\neq 2$ and $v_\Delta(\beta)\neq 0$. From this, $v_\Delta(\alpha \to \beta)=2 \in  {\to}'^{S4}(v_\Delta(\alpha),v_\Delta(\beta))$. 
Suppose now that $v_\Delta(\alpha \to \beta)=1$. Then, $\Box(\alpha \to \beta) \notin \Delta$ but $\alpha \to \beta \in \Delta$  and so $\alpha \notin \Delta$ or $\beta \in \Delta$. In addition, $\Box\beta \notin \Delta$, since $\Box\beta \to \Box(\alpha \to \beta)$ is a theorem in  $\mathcal{H}_{S4}$. If $\alpha \notin \Delta$ then  $v_\Delta(\alpha)=0$ and $v_\Delta(\beta) \in \{0,1\}$.  Thus, $v_\Delta(\alpha \to \beta)=1 \in  \{1,2\}={\to}'^{S4}(v_\Delta(\alpha),v_\Delta(\beta))$. Otherwise, if $\beta \in \Delta$ then $v_\Delta(\beta)=1$ and so $v_\Delta(\alpha \to \beta)=1 \in  {\to}'^{S4}(v_\Delta(\alpha),v_\Delta(\beta))$.
Finally, if $v_\Delta(\alpha \to \beta)=0$ then $\alpha \to \beta \notin \Delta$, hence $\alpha \in \Delta$ and $\beta \notin \Delta$. This means that $v_\Delta(\alpha)  \in  \{1,2\}$ and $v_\Delta(\beta)=0$ and so $v_\Delta(\alpha \to \beta)=0 \in  \{0\}={\to}'^{S4}(v_\Delta(\alpha),v_\Delta(\beta))$.

This completes the proof.
\end{proof}

\begin{theorem} [Completeness  of $\mathcal{H}_{S4}$ w.r.t.  $\mathcal{R}(\mathcal{M}'_{S4})$] \label{thm:compl:S4N}
For every $\Gamma$ and $\varphi$: $\Gamma  \models_{\mathcal{R}(\mathcal{M}'_{S4})} \varphi$ implies that  $\Gamma \vdash_{S4} \varphi$.
\end{theorem}
\begin{proof}
It is analogous to the proof for $\mathcal{R}(\mathcal{M}_{S4})$ found in~\cite[Lemma~3.9 and Theorem~3.10]{gratz_truth_2022}, but now by using our Lemma~\ref{lemma~3.7-adapted} instead of~\cite[Lemma~3.7]{gratz_truth_2022}.
\end{proof}

Now, Definitions~\ref{def:partial_valuation0} and~\ref{def:partial_level_val_S4} can be easily adapted to   $\mathcal{M}'_{S4}$, obtaining so the set $PLV'$ of partial level valuations in $\mathcal{R}(\mathcal{M}'_{S4})$ over finite sets of formulas. However, it will be convenient to give a bit more detailed presentation of these sets, by specifying explicitly the domain of each valuation.

\begin{definition} Consider the sets 
$$CS(\Sigma)= \{\Lambda \subseteq For(\Sigma) \mid  \mbox{ $\Lambda$ is non-empty and closed under subformulas}\},$$
 $$FCS(\Sigma)= \{\Lambda \subseteq For(\Sigma) \mid  \mbox{ $\Lambda$ is finite, non-empty and closed under subformulas}\}.$$
\end{definition}

\begin{definition}[Partial valuation  in $\mathcal{M}'_{S4}$]\label{def:partial_valuationN}
 Let $\Lambda \in CS(\Sigma)$.  A partial valuation in $\mathcal{M}'_{S4}$ is a function $\tilde v_p :\Lambda \rightarrow \mathcal{V}$ such that, for every $\alpha,\beta \in \Lambda$:
 
 \begin{itemize}
 \item[--] if $\# \in \{\neg,\Box\}$ and  $\#\alpha \in \Lambda$ then $\tilde v_p(\#\alpha) \in \#^{S4}(\tilde v_p(\alpha))$;
 \item[--] if  $\# \in \{\to,\vee,\land\}$ and $\alpha \# \beta\in \Lambda$ then  $\tilde v_p(\alpha \# \beta) \in \#'^{S4}(\tilde v_p(\alpha),\tilde v_p(\beta))$.
 \end{itemize}
 Let $PV(\Lambda)$ be the set of  partial valuations in $\mathcal{M}'_{S4}$ with domain $\Lambda$.
\end{definition}

\begin{definition}[Partial level valuation over $\Lambda$]\label{def:partial_level_val_S4N}
    Let $\Lambda \in FCS(\Sigma)$. A  partial valuation $\tilde v_p \in PV(\Lambda)$ is a {\em partial level valuation} in $\mathcal{R}(\mathcal{M}'_{S4})$ over $\Lambda$  iff

    \begin{enumerate}
        \item[] $\forall \alpha \in \Lambda$ such that $\tilde v_p(\alpha) = 1$, there exists a partial level valuation $\tilde w_p$ in $\mathcal{R}(\mathcal{M}'_{S4})$ over $\Lambda$ such that $\tilde w_p(\alpha) = 0$ and, $\forall \beta \in \Lambda$, $\tilde w_p(\beta) = 2$ whenever $\tilde v_p(\beta) = 2$.
    \end{enumerate}
The set of partial level valuations in $\mathcal{R}(\mathcal{M}'_{S4})$ over $\Lambda$ will be denoted by $PLV(\Lambda)$.
\end{definition}

\begin{remark} \label{rem:PLVsN}
Given $v,w$ and $\alpha$ let $\mathsf{P}^4_\Lambda(v,w,\alpha)$ iff $w(\alpha) = 0$ and, $\forall \beta \in \Lambda$, $w(\beta) = 2$ whenever $v(\beta) = 2$. If $\Lambda \in FCS(\Sigma)$ then\\[2mm]
$\begin{array}{lll}
PLV(\Lambda)&=& \{\tilde v_p \in PV(\Lambda) \mid \forall \alpha \in \Lambda \big(\tilde v_p(\alpha) = 1 \mbox{ implies that } \mathsf{P}^4_\Lambda(\tilde v_p,\tilde w_p,\alpha)\\[1mm] 
&&\hspace*{7mm}\mbox{ for some $\tilde w_p \in PLV(\Lambda)$}\big)\}.
\end{array}$
\end{remark}

As observed in Remark~\ref{rem:PLVs} for $\mathcal{M}_{S4}$, since $\Lambda$ is finite then $PV(\Lambda)$  is finite, hence the definition above is not cyclic. To be more precise, the following result can be proven:

\begin{proposition} [Existence and uniqueness of partial level valuations in $\mathcal{R}(\mathcal{M}'_{S4})$] \label{algorithm-PLV}
    Let   $\Lambda \in FCS(\Sigma)$. Then, there exists a unique set $PLV(\Lambda) \subseteq PV(\Lambda)$ satisfying the condition of Remark~\ref{rem:PLVsN}. Moreover,  $PLV(\Lambda) \neq \emptyset$.
\end{proposition}
\begin{proof}
The proof can be easily adapted from the corresponding one for Proposition~\ref{algorithm-iPLV} for \ipl\  to be given in Section~\ref{tables-IPL}.
\end{proof}

 The next step is to show  analyticity and co-analyticity w.r.t.  level valuations in $\mathcal{M}'_{S4}$. In order to do this it is considered, adapting~\cite[Definition~3.12]{gratz_truth_2022}, the intermediary notion of partial' level valuation in $\mathcal{R}(\mathcal{M}'_{S4})$ over $\Lambda$.

\begin{definition}[Partial' level valuation over $\Lambda$]\label{def:partial_level_val_S4NN}
    Let $\Lambda \in CS(\Sigma)$. A  partial valuation $\tilde v_p \in PV(\Lambda)$ is a {\em partial' level valuation} in $\mathcal{R}(\mathcal{M}'_{S4})$ over $\Lambda$  iff

    \begin{enumerate}
        \item[] $\forall \alpha \in \Lambda$ such that $\tilde v_p(\alpha) = 1$, there exists a level valuation $w$ in $\mathcal{L}'_{S4}$ such that $w(\alpha) = 0$ and, $\forall \beta \in \Lambda$, $w(\beta) = 2$ whenever $\tilde v_p(\beta) = 2$.
    \end{enumerate}
The set of partial' level valuations in $\mathcal{R}(\mathcal{M}'_{S4})$ over $\Lambda$ will be denoted by $PLV'(\Lambda)$.
\end{definition}

\begin{remark} \label{rem:PLVsNN}
Consider the predicate $\mathsf{P}^4_\Lambda(v,w,\alpha)$ introduced in Remark~\ref{rem:PLVsN}.  Clearly,\\[2mm]
$\begin{array}{lll}
PLV'(\Lambda)&=& \{\tilde v_p \in PV(\Lambda) \mid \forall \alpha \in \Lambda \big(\tilde v_p(\alpha) = 1 \mbox{ implies that } \mathsf{P}^4_\Lambda(\tilde v_p,w,\alpha)\\[1mm] 
&&\hspace*{7mm}\mbox{ for some $w \in \mathcal{L}'_{S4}$}\big)\}.
\end{array}$
\end{remark}

\begin{lemma} [Co-analyticity lemma for $\mathcal{R}(\mathcal{M}'_{S4})$] \label{co-analyticityS4} Let $v \in \mathcal{L}'_{S4}$ and  $\Lambda \in CS(\Sigma)$. Then, the restriction $\tilde v_p:=v_{|\Lambda}$ of $v$ to the domain $\Lambda$ belongs to $PLV'(\Lambda)$.
\end{lemma}
\begin{proof} It is an adaptation of the proof of~\cite[Lemma~3.13]{gratz_truth_2022}.
Suppose, by contradiction, that $\tilde v_p \notin PLV'(\Lambda)$. 
Let $\Gamma:= \{\beta \in \Lambda \mid \tilde v_p(\beta) =2\}$. Then, there exists $\alpha \in \Lambda$ such that $\tilde v_p(\alpha) = 1$ and, for every $w \in \mathcal{L}'_{S4}$, $w(\alpha)=0$ implies that $w(\beta) \neq 2$ for some $\beta \in \Gamma$. Hence, for every $w \in \mathcal{L}'_{S4}$, $w(\alpha) \not\in \mathcal{D}$ implies that $w(\Box \beta) \notin \mathcal{D}$ for some $\beta \in \Gamma$. By contraposition, this means that $\{\Box \beta \mid \beta \in \Gamma \} \models_{\mathcal{R}(\mathcal{M}'_{S4})} \alpha$.  By Theorem~\ref{thm:compl:S4N},  $\{\Box \beta \mid \beta \in \Gamma \} \vdash_{S4} \alpha$. Observe that $\nvdash_{S4} \alpha$: otherwise, $v(\alpha)=2$, by Corollary~ \ref{valid_S4=2}. Hence, there exist $\beta_1,\ldots, \beta_n \in \Gamma$ such that $\vdash_{S4} \Box\beta_1 \to (\Box\beta_2 \to(\ldots \to(\Box\beta_n \to \alpha) \ldots))$. By necessitation rule,  $\vdash_{S4} \Box(\Box\beta_1 \to (\Box\beta_2 \to(\ldots \to(\Box\beta_n \to \alpha) \ldots)))$. By Theorem~\ref{thm:sound:S4N}, $\models_{\mathcal{R}(\mathcal{M}'_{S4})} \Box(\Box\beta_1 \to (\Box\beta_2 \to(\ldots \to(\Box\beta_n \to \alpha) \ldots)))$. But $v(\Box\beta_i) =2$ for $1 \leq i \leq n$ and $v(\alpha)=1$, hence $v(\Box\beta_1 \to (\Box\beta_2 \to(\ldots \to(\Box\beta_n \to \alpha) \ldots)))=1$. This means that $v(\Box(\Box\beta_1 \to (\Box\beta_2 \to(\ldots \to(\Box\beta_n \to \alpha) \ldots))))=0$, a contradiction. From this we infer that  $\tilde v_p \in PLV'(\Lambda)$. 
\end{proof}

\begin{definition}
The {\em complexity} $\co_1(\alpha)$ of a formula $\alpha \in For(\Sigma)$ is defined as follows:  $\co_1(p)=0$ if $p$ is  a propositional variable; $\co_1(\neg \alpha)=\co_1(\Box \alpha)= \co_1(\alpha)+1$; and $\co_1(\alpha \# \beta)=\co_1(\alpha)+\co_1(\beta)+1$, for $\# \in \{\to,\vee,\land\}$. Let $For(\Sigma)_n=\{\alpha \in For(\Sigma) \mid \co_1(\alpha) \leq n\}$ for every $n \geq 0$.
\end{definition}

\begin{remark} \label{enum:forS4}
Observe that it is possible to define an enumeration  $\alpha_1,\alpha_2,\ldots \alpha_m, \ldots$ (for $m \in \omega^2$)  of $For(\Sigma)$ such that $\co_1(\alpha_i) \leq \co_1(\alpha_j)$ if $i \leq j$ and, for every $i$ such that $\co_1(\alpha_i) >0$, if $\beta$ is a strict subformula of $\alpha_i$ then $\beta=\alpha_j$ for some $j < i$. Such an enumeration of $For(\Sigma)$ can be defined as follows: every formula have and index in $\omega^2$, which is a denumerable ordinal, such that all the formulas with complexity $0$ (which form a denumerable set) are placed first, with indexes in $I_0:=\omega$; after this, all the formulas with complexity $1$ (which form a denumerable set) are placed with an index in $I_1:=\omega\cdot 2\setminus \omega=\{\omega,\omega + 1, \ldots\}$; in general, the formulas with complexity $n$ (which form a denumerable set) have an index in $I_{n}:=\omega\cdot(n+1)\setminus\omega\cdot n=\{\omega\cdot n, \omega\cdot n+1, \ldots\}$.
\end{remark}

Based on the enumeration of $For(\Sigma)$ given in the previous remark, consider the following sets:

\begin{definition}  \label{sets:enum:forS4}
Let $\alpha_1,\alpha_2,\ldots \alpha_m, \ldots$ (for $m \in \omega^2$) be an enumeration of $For(\Sigma)$ as in Remark~\ref{enum:forS4}. For every $n,m \in \omega$ let  ${\Lambda}_n:=\Lambda \cup For(\Sigma)_n$;   ${\Lambda}_n^0:={\Lambda}_n$; and ${\Lambda}_n^{m+1}={\Lambda}_n^m \cup \{\alpha_{\omega\cdot (n+1) +m}\}$.
\end{definition}

\begin{remark} \label{enum2:forS4} By definition, ${\Lambda}_n^{m+1}={\Lambda}_n \cup \{\alpha_{\omega\cdot (n+1)}, \alpha_{\omega\cdot (n+1) +1}, \ldots, \alpha_{\omega\cdot (n+1) +m}\}$, and $\alpha_{\omega\cdot (n+1) +m}\in \Lambda$ if and only if ${\Lambda}_n^{m+1}={\Lambda}_n^m$. Hence, ${\Lambda}_n^{m+1}$ is obtained from ${\Lambda}_n$ by adding the first $m+1$ formulas (of the given enumeration) with complexity $n+1$, and so ${\Lambda}_n^{m}$ adds at most $m$ formulas to ${\Lambda}_n$, for $m \in \omega$.
It is worth noting that ${\Lambda}_n^m \in CS(\Sigma)$ for every $n,m \in \omega$; in particular, ${\Lambda}_n \in CS(\Sigma)$ for every $n \in \omega$. Clearly ${\Lambda}_n^m \subseteq {\Lambda}_n^{m+1}$, ${\Lambda}_{n+1} = \bigcup_{m \in \omega} {\Lambda}_n^m$, and $For(\Sigma)= \bigcup_{n \in \omega} {\Lambda}_n$.
\end{remark}

\begin{lemma} [Analyticity lemma for $\mathcal{R}(\mathcal{M}'_{S4})$] \label{analyticityS4} Let $\Lambda \in CS(\Sigma)$ and  $\tilde v_p \in PLV'(\Lambda)$. Then, there exists $v \in \mathcal{L}'_{S4}$ such that the restriction $v_{|\Lambda}$ of $v$ to the domain $\Lambda$ coincides with $\tilde v_p$.
\end{lemma}
\begin{proof} It is an adaptation of the proof of~\cite[Lemma~3.14]{gratz_truth_2022}. In this case, more modifications are required because of the conjunction and disjunction connectives, and by the fact of considering a fixed set $\Lambda$.

Suppose first that $\Lambda=For(\Sigma)$.  By induction on $n \in \omega$, it will be shown that $\tilde v_p \in \mathcal{L}'^{S4}_n$ for every $n$. By Definition~\ref{def:partial_valuationN} of partial valuation, $\tilde v_p \in Val(\mathcal{M}'_{S4})= {\mathcal{L}'}_0^{S4}$. Suppose that $\tilde v_p \in \mathcal{L}'^{S4}_n$ for a given $n \geq 0$, and let $\alpha$ such that $\mathcal{L}'^{S4}_n(\alpha) \subseteq \mathcal{D}$. In particular, $\tilde v_p(\alpha) \in  \mathcal{D}$. Suppose that $\tilde v_p(\alpha)=1$. By Definition~\ref{def:partial_level_val_S4NN} of partial' level valuation,  there exists a level valuation $w$ in $\mathcal{L}'_{S4}$ such that $w(\alpha) = 0$. But this is  a contradiction, since $w \in \mathcal{L}'^{S4}_n$. Then $\tilde v_p(\alpha)=2$ and so $\tilde v_p \in \mathcal{L}'^{S4}_{n+1}$. From this, $\tilde v_p \in \mathcal{L}'_{S4}$ and the result clearly holds by taking $v:=\tilde v_p$.

Suppose now that $\Lambda \neq \emptyset $ is a proper subset of $For(\Sigma)$.
Consider  an enumeration of $For(\Sigma)$ as in Remark~\ref{enum:forS4}.
The valuation $v$ will be defined by induction on the complexity $n$ of the formulas. 
More precisely, with terminology as in Definition~\ref{sets:enum:forS4} and from the observations in Remark~\ref{enum2:forS4}, an extension $v_n^m$ of $\tilde v_p$ (i.e.,  $\tilde v_p \subseteq v_n^m$) will be defined for every $(n,w) \in \omega\times \omega$ such that $v_n^m \in PLV'({\Lambda}_n^m)$ and $v_i^j \subseteq v_n^{m}$ if $(i,j) \leq (n,m)$, where: $(i,j) \leq (n,m)$ iff $i \leq n$ or $i=n$ and $j \leq m$ (and where $\omega \times \omega$ denotes, as usual, the Cartesian product of $\omega$ with itself). \\[1mm]
{\bf Base} $n=m=0$: Observe that ${\Lambda}_0={\Lambda}_0^0=\Lambda \cup \mathcal{P}$.
Define $v_0^0(\alpha)= \tilde v_p(\alpha)$ if $\alpha \in \Lambda$, and $v_0^0(p)=0$ if $p$ is a propositional variable which does not belong to $\Lambda$. Clearly, $\tilde v_p \subseteq v_0^0$ and $v_0^0 \in PLV'({\Lambda}_0)=PLV'({\Lambda}_0^0)$. Let $v_0:=v_0^0$.\\[1mm]
{\bf Inductive step}: Assume that $v_i^j(\alpha)$ was defined for every $\alpha \in {\Lambda}_i^j$ such that  $v_i^j(\alpha)= \tilde v_p(\alpha)$ if $\alpha \in \Lambda$, $v_i^j \in PLV'({\Lambda}_i^j)$, $v_i^j \subseteq v_k^r$ if $(i,j) \leq (k,r) \leq (n,m)$, for given $n,m \geq 0$ (Induction Hypothesis, IH). Now it will be shown how to extend $v_n^m$ to a function $v_n^{m+1}$ with domain ${\Lambda}_n^{m+1}={\Lambda}_n^m \cup \{\alpha_{\omega\cdot (n+1) +m}\}$. That is, it will be shown how to define a value for $\alpha_{\omega\cdot (n+1) +m}$ while preserving the values assigned by $v_n^m$, in such a manner that $v_n^{m+1}\in PLV'({\Lambda}_n^{m+1})$. To start with, define $v_n^{m+1}(\beta)=v_n^m(\beta)$ for every $\beta \in {\Lambda}_n^m$.

Let $\alpha=\alpha_{\omega\cdot (n+1) +m}$. If $\alpha \in \Lambda$ then ${\Lambda}_n^{m+1}={\Lambda}_n^m$. Then  $v_n^{m+1}:=v_n^m$, and the required hypothesis are obviously satisfied by $v_n^{m+1}$. Suppose now that  $\alpha \not\in \Lambda$. Observe that $v_n^{m+1}(\beta)=v_n^{m}(\beta)$ was already defined, for every strict subformula $\beta$ of $\alpha$. 
Since $v_n^m \in PLV'({\Lambda}_n^m)$ then, for every $\delta \in {\Lambda}_n^m$ such that $v_n^m(\delta)=1$ there exists $w   \in  \mathcal{L}'_{S4}$ such that $\mathsf{P}^4_{{\Lambda}_n^m}(v_n^m,w,\delta)$. Such a $w$ will be denoted by $w^\delta$ (observe that it is possible to have more than one $w^\delta$ for each $\delta$, and it is possible to have $w^\delta=w^\gamma$ for $\delta \neq\gamma$). Recall that ${\Lambda}_n^m$ is closed under subformulas.
There are three cases to analyze:\\[1mm]
(1) $\alpha=\#\beta$ such that $\#^{S4} v_n^{m}(\beta)=\{0\}$ (for some $\# \in \{\neg,\Box\}$), or $\alpha=\beta\#\gamma$ such that $\#'^{S4} (v_n^{m}(\beta),v_n^{m}(\gamma))=\{0\}$ (for some $\# \in \{\land,\vee,\to\}$). Then, define $v_n^{m+1}(\alpha)=0$. Observe that $v_n^{m+1} \in PLV'({\Lambda}_n^{m+1})$.\\[1mm]
(2) $\alpha=\#\beta$ such that $\#^{S4} v_n^{m}(\beta) \in \mathcal{D}$ (for some $\# \in \{\neg,\Box\}$), or $\alpha=\beta\#\gamma$ such that $\#'^{S4} (v_n^{m}(\beta),v_n^{m}(\gamma)) \in \mathcal{D}$ (for some $\# \in \{\vee,\to\}$). There are three subcases to analyze:\\[1mm]
(2.1) There is $\delta \in{\Lambda}_n^m$ with $v_n^{m}(\delta)=1$ and there exists some $w^\delta$ such that $w^\delta(\alpha)=0$. Define $v_n^{m+1}(\alpha)=1$. Hence $v_n^{m+1} \in PLV'(\Lambda_n^{m+1})$ such that $w^\alpha=w^\delta$.\\[1mm]
(2.2) There is $\delta \in \Lambda_n^m$ with $v_n^m(\delta)=1$ and there exists some $w^\delta$ such that $w^\delta(\alpha)=1$. By Lemma~\ref{co-analyticityS4}, $v'=w^\delta_{|\Lambda_n^m} \in PLV'(\Lambda_n^m)$, hence there exists $w'' \in  \mathcal{L}'_{S4}$ such that $\mathsf{P}^4_{\Lambda_n^m}(v',w'',\alpha)$.  Define $v_n^{m+1}(\alpha)=1$. Hence, $v_n^{m+1} \in PLV'(\Lambda_n^{m+1})$ such that $w^\alpha=w''$.\\[1mm]
(2.3) For every $\delta \in \Lambda_n^m$ such that $v_n^m(\delta)=1$, $w^\delta(\alpha)=2$ for every $w^\delta$. Define $v_n^{m+1}(\alpha)=2$. Clearly,  $v_n^{m+1} \in PLV'(\Lambda_n^{m+1})$.\\[1mm]
(3) $\alpha=\beta\#\gamma$ such that $\#'^{S4} (v_n^{m}(\beta),v_n^{m}(\gamma)) =\{1\}$ (for some $\# \in \{\land,\to\}$). There are two subcases to analyze:\\[1mm]
(3.1) $\alpha=\beta\to\gamma$. Then $v_n^{m}(\beta)=2$ and $v_n^{m}(\gamma)=1$. Define $v_n^{m+1}(\alpha)=1$. Then  $v_n^{m+1} \in PLV'(\Lambda_n^{m+1})$ such that $w^\alpha=w^\gamma$. Indeed, since $w^\gamma(\beta)=2$ and $w^\gamma(\gamma)=0$ then $w^\gamma(\alpha)=0$. \\[1mm]
(3.2) $\alpha=\beta\land\gamma$. Then $v_n^{m}(\beta)=1$ and $v_n^{m}(\gamma) \in \mathcal{D}$ or vice versa. Then, there exists $w \in  \mathcal{L}'_{S4}$ (where $w=w^\beta$ or  $w=w^\gamma$)  such that $w(\beta)=0$ or $w(\gamma)=0$ and so $w(\alpha)=0$. Define $v_n^{m+1}(\alpha)=1$. Then  $v_n^{m+1} \in PLV'(\Lambda_n^{m+1})$ such that $w^\alpha=w^\beta$ or $w^\alpha=w^\gamma$. \\[1mm]
(4) $\alpha=\Box\beta$ such that $\Box^{S4} v(\beta) =\{2\}$ or $\alpha=\beta\#\gamma$ such that $\#'^{S4} (v_n^{m}(\beta),v_n^{m}(\gamma))  =\{2\}$ (for some $\# \in \{\vee,\land, \to\}$). Define $v_n^{m+1}(\alpha)=2$. It is clear that $v_n^{m+1} \in PLV'(\Lambda_n^{m+1})$.

We have shown in cases (1)-(4) how to extend the domain of $v_n^{m}$ to the additional formula $\alpha_{\omega\cdot (n+1) +m}$, showing that the resulting function $v_n^{m+1}$ is in $PLV'(\Lambda_n^{m+1})$, and $\tilde v_p \subseteq v_i^j \subseteq v_k^r$, for every  $(i,j) \leq (k,r) \leq (n,m+1)$. That is, we show how to increase the superscript $m$. In order to increase the subscript $n$, let $v_{n+1}=v_{n+1}^0 := \bigcup_{m \in \omega} v_n^m$. By the procedure described above, it is immediate to see that $v_{n+1}^0$ is in $PLV'(\bigcup_{m \in \omega} {\Lambda}_n^m) = PLV'({\Lambda}_{n+1})= PLV'({\Lambda}_{n+1}^0)$, and $\tilde v_p \subseteq v_i^j \subseteq v_k^r$, for every  $(i,j) \leq (k,r) \leq (n+1,0)$. By repeating the process, it is possible to define $v_n^m$ for every $(n,m) \in \omega\times \omega$ with the required properties. Finally, let $v:=\bigcup_{n\in \omega} v_n$. From the manner in which the construction was made, it follows that  $v$ is in $PLV'(\bigcup_{n \in \omega}{\Lambda}_{n})= PLV'(For(\Sigma))$, and $\tilde v_p \subseteq v$. By the first part of the proof, we conclude that $v \in  \mathcal{L}'_{S4}$.
\end{proof}

\begin{proposition} \label{prop:analiticityS4}
Let $\Lambda \in FCS(\Sigma)$. Then $PLV(\Lambda)=PLV'(\Lambda)$.
\end{proposition}
\begin{proof}
Let  $\tilde v_p \in PLV'(\Lambda)$. Then,  $\tilde v_p \in PV(\Lambda)$. Suppose that $\alpha \in \Lambda$ such that $\tilde v_p(\alpha)=1$.   Then, there exists  $w \in \mathcal{L}'_{S4}$ such that $\mathsf{P}^4_\Lambda(\tilde v_p,w,\alpha)$. Let $\tilde w_p=w_{|\Lambda}$. By Lemma~\ref{co-analyticityS4}, $\tilde w_p \in PLV'(\Lambda)$ such that $\mathsf{P}^4_\Lambda(\tilde v_p,\tilde w_p,\alpha)$. Now, suppose that there exists $\tilde w'_p\in PLV'(\Lambda)$ such that $\mathsf{P}^4_\Lambda(\tilde v_p,\tilde w'_p,\alpha)$. By Lemma~\ref{analyticityS4}, there exists $v' \in \mathcal{L}'_{S4}$ such that the restriction of $v'$ to $\Lambda$ coincides with $\tilde w'_p $. Hence,  $\mathsf{P}^4_\Lambda(\tilde v_p,v',\alpha)$. This shows that \\[2mm]
$\begin{array}{lll}
PLV'(\Lambda)&=& \{\tilde v_p \in PV(\Lambda) \mid \forall \alpha \in \Lambda \big(\tilde v_p(\alpha) = 1 \mbox{ implies that } \mathsf{P}^4_\Lambda(\tilde v_p,\tilde w_p,\alpha)\\[1mm] 
&&\hspace*{7mm}\mbox{ for some $\tilde w_p \in PLV'(\Lambda)$}\big)\}.
\end{array}$
From this, it follows that  $PLV(\Lambda)=PLV'(\Lambda)$, by Proposition~\ref{algorithm-PLV}.
\end{proof}

%
%

\begin{theorem} \label{Theor:sound_compl_S4-2} Let $\varphi \in For(\Sigma)$, and let $\Lambda$ be the set of subformulas of $\varphi$. Then: $\vdash_{S4} \varphi$ if, and only if, for every partial valuation $\tilde v_p \in PLV(\Lambda)$ in $\mathcal{M}'_{S4}$,  $\tilde v_p(\varphi)=2$.
\end{theorem}
\begin{proof} Let $\Lambda$ be the set of subformulas of $\varphi$. Then, $\Lambda \in FCS(\Sigma)$.
Suppose that $\vdash_{S4} \varphi$. By Theorem~\ref{thm:sound:S4N}, $\models_{\mathcal{R}(\mathcal{M}'_{S4})} \varphi$. 
Let $\tilde v_p \in PLV(\Lambda)$. By Proposition~\ref{prop:analiticityS4}, $\tilde v_p \in PLV'(\Lambda)$. By Lemma~\ref{analyticityS4},  there exists $v \in \mathcal{L}'_{S4}$ such that the restriction $v_{|\Lambda}$ of $v$ to $\Lambda$ coincides with $\tilde v_p$. Then, $\tilde v_p(\varphi)=v(\varphi)=2$, by Corollary~\ref{valid_S4=2}.

Conversely, suppose that  $\tilde v_p(\varphi)=2$ for every  $\tilde v_p \in PLV(\Lambda)$. By Proposition~\ref{prop:analiticityS4},  $\tilde v_p(\varphi)=2$ for every  $\tilde v_p \in PLV'(\Lambda)$. Let $v \in \mathcal{L}'_{S4}$. By Lemma~\ref{co-analyticityS4}, the restriction $\tilde v_p:=v_{|\Lambda}$ of $v$ to $\Lambda$ belongs to $PLV'(\Lambda)$. Then, $v(\varphi)=\tilde v_p(\varphi)=2$ and then $\models_{\mathcal{R}(\mathcal{M}'_{S4})} \varphi$. By Theorem~\ref{thm:compl:S4N}, $\vdash_{S4} \varphi$.
\end{proof}


This means that Gr\"atz's algorithm can be adapted to the reduced (and expanded with disjunction and conjunction)  RNmatrix $\mathcal{R}(\mathcal{M}'_{S4})$.

\section{First steps towards an RNmatrix  for \ipl}\label{Sect:IPLfirstStep}

As mentioned in Section~\ref{sect:Intro}, Kearns's level valuations fail to provide an effective decision procedure for \sfo, while Gr\"atz method produces an algorithm to decide that logic. Clearly, the composition of  G\"odel's translation from \ipl\ to \sfo\ with Gr\"atz's decision procedure for \sfo\ naturally produces a decision procedure for \ipl. 

Let $\Sigma= (\{\neg,\Box\},\{\land,\vee,\to\})$ be the full signature of \sfo, and let $\Omega = (\{ \neg \}, \{ \land, \lor, \rightarrow \} )$ be the signature for intuitionistic propositional logic \ipl. 

\begin{definition} \label{complex-fbf-IPL}
The {\em complexity} $\co(\alpha)$ of a formula $\alpha \in For(\Omega)$ is defined as follows:  $\co(p)=0$ if $p$ is  a propositional variable; $\co(\neg \alpha)= \co(\alpha)+1$; and $\co(\alpha \# \beta)=\co(\alpha)+\co(\beta)+1$, for $\# \in \{\to,\vee,\land\}$.
\end{definition}

Recall the G\"odel-Tarski-McKinsey Box translation $t:For(\Omega) \rightarrow For(\Sigma)$ defined as follows:

    \begin{enumerate}[label=\roman*)]
        \item $t(p) = \Box p$, if $p$ is a propositional variable;
        \item $t(\neg \alpha) = \Box \neg t(\alpha)$;
        \item $t(\alpha \# \beta) = \Box (t(\alpha) \# t(\beta))$, where $\# \in \{ \rightarrow, \lor, \land \}$.
     \end{enumerate}

It is widely known that this translation is faithful (see~\cite{McKinsey:Tarsk:48}; for a more recent proof see~\cite{mints2012godel}), that is:

\begin{theorem}[\cite{McKinsey:Tarsk:48}] \label{teor-McKinsey:Tarski} Let $\varphi$ be a formula in $For(\Omega)$. Then: $\varphi$ is valid in \ipl\ if and only if $t(\varphi)$ is valid in \sfo.
\end{theorem}

As seen in Section~\ref{sect:redu_RNmat_S4}, Gr\"atz's results and algorithm can be extended to \sfo\ defined over $\Sigma$, now  based on  $\mathcal{M}'_{S4}$.
Then, by Theorem~\ref{teor-McKinsey:Tarski} and by Theorem~\ref{Theor:sound_compl_S4-2}, a formula $\alpha$ in $For(\Omega)$  is valid in \ipl\ iff the reduced truth-table  obtained by taking the partial valuations in $\mathcal{M}'_{S4}$ for the set of subformulas of the formula $t(\alpha)$ in $For(\Sigma)$ shows that $t(\alpha)$ is a valid formula in \sfo. This produces obviously a decision procedure for \ipl, based on an RNmatrix for \sfo\ and a translation between \ipl\ and \sfo. It is a challenging question how to define a direct decision procedure for \ipl\ based on a finite-valued RNmatrix for \ipl, with level valuations (and with an associate truth-tables decision procedure) specifically designed for this logic. Moreover, the soundness and completeness of the RNmatrix (and the decision procedure obtained from it) should not rely on the corresponding soundness and completeness results of G\"odel's~\cite{Godel1933} and  Gr\"atz's constructions. A solution to this question will be given in the next two sections.



\section{An RNmatrix  for \ipl} \label{RNmatrix-IPL}

This section describes a new semantics for \ipl\ based on a 3-valued RNmatrix with a very intuitive interpretation. We obtain this semantics by abstracting the composed procedure for \ipl\ mentioned at the end of the previous section but defined in a direct way. 
That is, we describe a `pure' RNmatrix for \ipl\ (i.e., one that does not depend on any other construction), which turns out to be a suitable semantics for \ipl. Moreover, Section~\ref{tables-IPL} presents an algorithm for deciding validity in \ipl\ by removing spurious rows from the truth-tables generated by the procedure. Hence, the soundness and completeness of the semantics and the associated decision procedure are proven independently of G\"odel's 1933 results (proved in~\cite{McKinsey:Tarsk:48}) and those in~\cite{gratz_truth_2022}.


\subsection{Inducing the Nmatrix $\mathcal{M}_{IPL}$ for \ipl}

The truth-tables for intuitionistic logic can be defined based on the  Box translation $t$ described above, combined with a semi-translation $t_0$ to be defined now.

\begin{definition}\label{def:h_function}
    The {\em semi-translation function} $t_0 : For(\Omega) \rightarrow For(\Sigma)$ is defined as follows:
    \begin{enumerate}[label=\roman*)]
        \item $t_0(p) = p$, if $p$ is a propositional variable;
        \item $t_0(\neg \alpha) = \neg t(\alpha)$;
        \item $t_0(\alpha \# \beta) = t(\alpha) \# t(\beta)$, where $\# \in \{ \rightarrow, \lor, \land \}$.
    \end{enumerate}
\end{definition}

\begin{lemma}\label{lemma:f_to_h}
    For every $\alpha \in For(\Omega)$, $t(\alpha) = \Box(t_0(\alpha))$.
\end{lemma}
\begin{proof}
Straightforward, by induction on the complexity of $\alpha$.
\end{proof}

\begin{remark}\label{remark:box_s4_deterministic}
Observe that the interpretation of $\Box$ in $\mathcal{M}'_{S4}$ is deterministic  (see Table~\ref{table:nm_S4*-R}). Let us denote by $\Box^4(a)$ the only element of ${\Box'}^{S4}(a)$, for $a \in \mathcal{V}$. Hence, if $v$ is a valuation in $\mathcal{M}'_{S4}$ then $v(\Box \alpha) = \Box^{4}(v(\alpha))$, for every $\alpha \in For(\Sigma)$.  By Lemma~\ref{lemma:f_to_h}, $v(t(\alpha)) = \Box^{4}(v(t_0(\alpha)))$, for every $\alpha \in For(\Omega)$.
\end{remark}

To check the validity in \ipl\ of $\alpha \in For(\Omega)$, one must construct a truth-table in $\mathcal{M}'_{S4}$ for $t(\alpha) \in For(\Sigma)$. Assume that the set of subformulas of $\alpha$ is arranged as a sequence $\alpha_1,\ldots,\alpha_n=\alpha$ in non-decreasing order of complexity (the complexity of formulas in $For(\Omega)$  is defined as usual, see Definition~\ref{complex-fbf-IPL} above).  


But then, the columns of the  truth-table for  $t(\alpha)$ in $\mathcal{M}'_{S4}$ can be organized in pairs of formulas of the form $\langle t_0(\alpha_1);t(\alpha_1)\rangle, \ldots , \langle t_0(\alpha_n);t(\alpha_n)\rangle= \langle t_0(\alpha);t(\alpha)\rangle$.  Indeed, the sequence of formulas $ t_0(\alpha_1),t(\alpha_1),t_0(\alpha_2), \ldots , t_0(\alpha_n),t(\alpha_n)=t(\alpha)$ in $For(\Sigma)$ induced by the sequence of pairs of formulas is the sequence of formulas that can be considered to construct a truth-table for $t(\alpha)$.\footnote{Rigorously speaking, if $\alpha$ depends on the propositional variables $p_1,\ldots,p_k$ then usually it is considered $\alpha_i=p_i$ (for $1 \leq i \leq k$), and then $t(\alpha_i)=\Box p_i$ is placed at a latter stage, say $\alpha_{i+k}$. Hence, the first $2k$ columns of a `standard' table for $t(\alpha)$ starts with formulas $p_1,\ldots,p_k,\Box p_1, \ldots \Box p_k$ instead of $p_1,\Box p_1,p_2,\Box p_2,\ldots p_k,\Box p_k$. After this, the pairs produce the final columns of a table for $t(\alpha)$.} For instance, in order to check $\alpha=p \vee \neg p$ it is considered the sequence of pairs $\langle p;\Box p\rangle$, $\langle \neg \Box p; \Box \neg \Box p\rangle$, $\langle \Box p \lor \Box \neg \Box p; \Box (\Box p \lor \Box \neg \Box p)\rangle$ based on the sequence $p,\neg p, p \vee \neg p$ for $\alpha$. Observe that $t(\alpha)=\Box (\Box p \lor \Box \neg \Box p)$.

This suggests considering for \ipl\ an Nmatrix (derived from $\mathcal{M}'_{S4}$) with domain  $\left\{ \langle c; \Box^{4}(c) \rangle \mid c \in \mathcal{V} \right\} = \{ \langle 2; 2 \rangle, \langle 1; 0 \rangle, \langle 0; 0 \rangle\}$ and multioperations defined as follows:

 \begin{enumerate}
        \item $\neg^{IPL} (\langle a; b \rangle) = \left\{ \langle c; \Box^{4}(c) \rangle \mid c \in {\neg'}^{S4}(b) \right\}$
        \item $\#^{IPL} (\langle a; b \rangle, \langle c; d \rangle) = \left\{ \langle e; \Box^{4}(e) \rangle\mid e \in {\#'}^{S4}(b, d)  \right\} $
    \end{enumerate}
    where $\# \in \{ \rightarrow, \lor, \land \}$.

Let us now consider a label for the truth-values above as follows:  $\bT := \langle 2; 2 \rangle$, $\bU := \langle 1; 0 \rangle$ and $\bF := \langle 0; 0 \rangle$. The observations above lead us to the following:

\begin{definition}[Intuitionistic Nmatrix  $\mathcal{M}_{IPL}$] \label{def:M_IPL} The intuitionistic Nmatrix over signature $\Omega$ is  given by $\mathcal{M}_{IPL} = \langle  V_{IPL}, \mathcal{D}_{IPL}, \mathcal{O}_{IPL} \rangle$, where $V_{IPL}  = \{ \bT, \bU,\bF \}$, $\mathcal{D}_{IPL} = \{ \bT \}$, and  $\mathcal{O}_{IPL}(\#)=\#^{IPL}$ for each connective $\#$ in $\Omega$  is defined as in Figure~\ref{table:nm_IPL}.
\end{definition}

\begin{figure}[h!]
    \centering
    \includegraphics[scale=0.5]{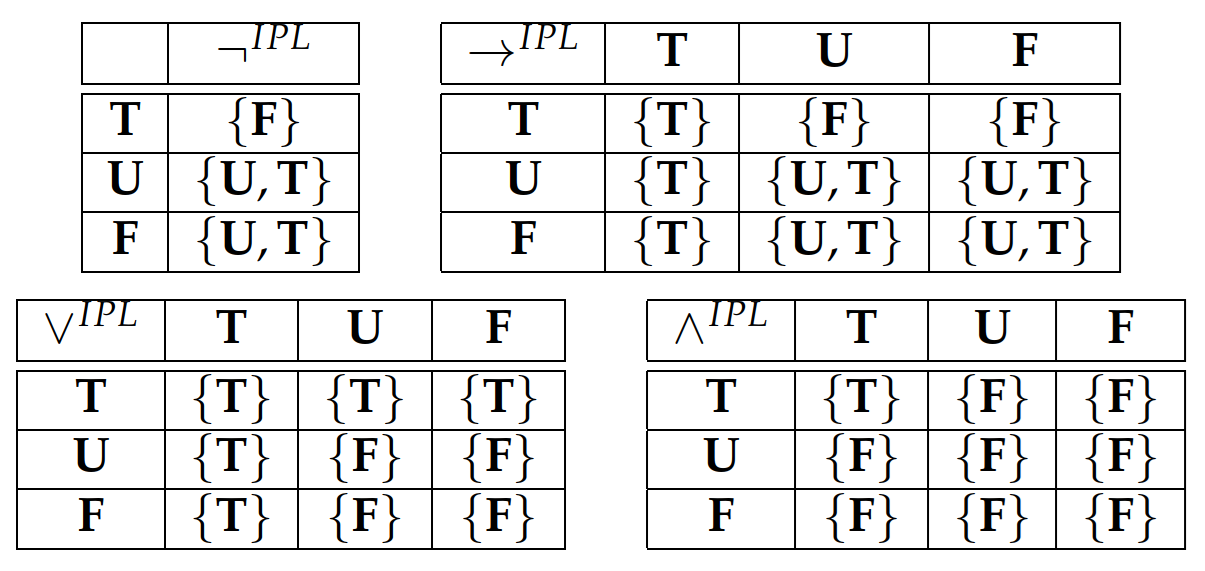}
    \caption{Tables for $\mathcal{M}_{IPL}$.}\label{table:nm_IPL}
\end{figure}

Observe that, in a given valuation over  $\mathcal{M}_{IPL}$, a formula is false when receives a non-designated value, $\bU$ or $\bF$. By looking at the tables of the Nmatrix, apparently both values have the same behavior. But it is not the case.

The first thing to note in Figure~\ref{table:nm_IPL} is that the truth-value $\bU$, which intuitively means that the epistemic status of $\alpha$ is {\em unknown}, only appears in the definitions of $\neg^{IPL}$ and $\to^{IPL}$. As it is going to be more detailed in the completeness proof, the intended meaning is the following (see also Remark~\ref{ren:canval} below):

\begin{description}
    \item[$v(\neg \alpha) = \bU$] only happens when both $\neg \alpha$ and  $\alpha$ are false; that is, when there is no proof of (or knowledge or information about) neither $\alpha$ nor $\neg\alpha$, hence excluded middle fails for $\alpha$ (intuitionistic, non-classical scenario);
    \item[$v(\neg \alpha) = \bF$] means that $\neg \alpha$ is false, but  $\alpha$ is true; that is, although $\neg \alpha$ is unknown, $\alpha$ is known, therefore excluded middle holds for $\alpha$ (classical scenario);
    \item[$v(\alpha \to \beta) = \bU$] only happens when both  $\alpha$ and $\alpha \to\beta$ are false; that is, there is  no proof of (or knowledge or information about) neither $\alpha$ nor $\alpha \to\beta$, hence Peirce's law $\alpha \vee (\alpha \to\beta)$ fails for $\alpha$ and $\alpha \to\beta$ (intuitionistic, non-classical scenario);
    \item[$v(\alpha \to \beta) = \bF$] means that $\alpha \to \beta$ is false, but  $\alpha$ is true; that is, although $\alpha \to \beta$ is unknown, $\alpha$ is known, therefore Peirce's law holds for $\alpha$ and $\alpha \to\beta$ (classical scenario).
\end{description}

In turn, the behavior of conjunction and disjunction is `classical' (in terms of designated/non-designated values), and that is to be expected: the inference rules for conjunction and disjunction are identical in classical and intuitionistic logic. Because of this, $\bU$  never appears as a value assigned to a conjunction or a disjunction.

The third point to note is that, if we disregard the truth-value $\bU$, the remaining Nmatrix corresponds exactly to classical logic. In fact, it is straightforward to see how classical logic can be fully recovered from the intuitionistic Nmatrix. This recovery can be achieved by refining $\mathcal{M}_{IPL}$ (in the sense of~\cite{avron2019rexpansions} and~\cite{caleiro2021axioms}) to the truth-values ${\bT, \bF}$. The resulting is an Nmatrix with zero degree of non-determinism, equivalent to the standard (deterministic) matrix for classical logic.


\subsection{Level valuations for $\mathcal{M}_{IPL}$}

The following step is the definition of a suitable set of valuations (level valuations) over the Nmatrix for \ipl\ just defined. The construction of the sets of level valuations given by  Kearns and Gr\"atz is based on the idea of {\em super designated truth-values}, that is: designated values $a$ such that any element of $\tilde{\Box}(a)$ is still designated, where $\tilde{\Box}$ is the multioperator associated to $\Box$. Thus, if a formula gets a designated value for every $n$th-level valuation, it must receive a super-designated value in every $(n+1)$th-level valuation. This guarantees the validity of the {\em Necessitation} inference rule, level by level. In this way, if a formula is provable in \sfo\ with $n$ applications of the necessitation rule, then it is valid in $\mathcal{L}^{S4}_n$. From this, soundness of syntactical proofs in \sfo\ w.r.t. the RNmatrix semantics follows easily. {\em Necessitation} is usually called a {\em global} inference rule: in order to be sound, the hypothesis of the rule must be provable, not merely assumed. However, \ipl\ has no global inference rules in its usual Hilbert-style presentation (being {\em Modus Ponens} the only inference rule, which is {\em local} in opposition to global).

Taking into account that we are abstracting the composition of the Box translation with Gr\"atz's truth-tables algorithm for \sfo, the set of $n$th-level valuations will be constructed based on partial level valuations (Definition~\ref{def:partial_level_val_S4}) instead of the Kearns-like Definition~\ref{def:level_valuations_2} of level valuations.  Thus, the requirement that every valuation (row) $v_p$ such $v_p(\alpha)=1$ must be supported by another valuation (row) $w_p$ with $w_p(\alpha)=0$ while preserving the values $2$ can be easily adapted to the present framework, by changing $1$, $0$ and $2$ by $\bU$, $\bF$ and $\bT$, respectively (see Definition~\ref{def:level_valuations_IPL} below). 

However, as observed above, there is no counterpart to {\em Necessitation} rule in the usual Hilbert-style calculus for \ipl, therefore the level valuations seem, in principle, unrelated to syntactical proofs in \ipl. The key is to consider a standard sequent calculus presentation for \ipl, and so it can be established, in a similar way to the case of \sfo, a relationship between syntactical proofs in \ipl\ and $nth$-level valuations: if a sequent is proved in \ipl\  with a derivation of height $n$ then it is valid w.r.t. $nth$-level valuations for \ipl\ (see Proposition~\ref{seq-rules-sound}). Using this, the soundness of the sequent calculus for \ipl\ w.r.t. its RNmatrix semantics can be easily obtained.

In addition, intuitionistic level valuations have a very interesting interpretation. Indeed, a suitable and very natural $3$-valued characteristic function associated to $\varphi$-saturated sets in \ipl\ will be introduced in  Definition~\ref{def:val-Delta-IPL}. Hence, in Proposition~\ref{prop:val-Delta-IPL-level} it will be proven that these characteristic functions correspond to level valuations.

The relationship between level valuations and $3$-valued characteristic functions of $\varphi$-saturated sets also has an interesting consequence in terms of complexity: at level $0$ (and hence at any level) the valuations can only assign the classical truth-values $\bT$ and $\bF$ to propositional variables. Indeed, atomic formulas contain neither negation nor implication connectives, so it makes no sense to assign them the value $\bU$, which, as discussed above, is exclusively related to the failure of the law of excluded middle and Peirce's law. This feature will reduce the complexity of the algorithm, as we will be discuss in Section~\ref{complexity}.


\begin{definition}\label{def:0-level_valuations_IPL}
Let  $Val^{nd}(\mathcal{M}_{IPL})$ be the set of  valuations over $\mathcal{M}_{IPL}$, and let $Val^{nd}_+(\mathcal{M}_{IPL})$ be the set of valuations $v$ in $Val^{nd}(\mathcal{M}_{IPL})$ such that $v(p) \in\{\bT,\bF\}$ for every propositional variable $p$.
\end{definition}

\begin{definition}\label{def:level_valuations_IPL} (Level valuations for \ipl)
The set of $n$th-level valuations $\mathcal{L}^{IPL}_n$ for \ipl, where $n \in \omega$, is defined as follows: 

    \begin{enumerate}[label=\roman*)]
        \item $\mathcal{L}_0^{IPL} = Val^{nd}_+(\mathcal{M}_{IPL})$;
        \item $\mathcal{L}^{IPL}_{n + 1} = \{ v \in \mathcal{L}^{IPL}_n \mid \forall \alpha \in For(\Omega)$, if $v(\alpha)=\bU$ then there exists $w \in \mathcal{L}^{IPL}_n$ such that $w(\alpha)=\bF$ and,  for every $\beta \in  For(\Omega)$: $v(\beta)=\bT$ implies $w(\beta)=\bT\}$.
    \end{enumerate}
    The set of level valuations in $\mathcal{M}_{IPL}$ is defined as the intersection of the sets of $n$th-level valuations, for $n\geq0$:

    \[
        \mathcal{L}_{IPL} = \bigcap^\infty_{n \geq 0} \mathcal{L}^{IPL}_n
    \]

\end{definition}

\begin{definition}[RNmatrix for \textbf{IPL}]\label{def:Rnmatrix_IPL}
    Let $\mathcal{M}_{IPL}$ be the Nmatrix defined in Definition~\ref{def:M_IPL}. The RNmatrix for \textbf{IPL} is the pair $\mathcal{R}(\mathcal{M}_{IPL}) = \langle \mathcal{M}_{IPL}, \mathcal{L}_{IPL}\rangle$.
\end{definition}

\subsection{Soundness and Completeness of \ipl\ w.r.t. $\mathcal{R}(\mathcal{M}_{IPL})$}

From now on, {\bf LI} will stand for a standard sequent calculus for \textbf{IPL} over $\Omega$. In fact, the sequent calculus {\bf LI} presented here is obtained from Gentzen's system {\bf LJ} (see, for instance, \cite[Chapter~1]{Takeuti}) with minor modifications. Concretely, we consider sequents formed by pairs of finite multisets of formulas instead of pairs of  sequences (which allows to eliminate the {\em exchange} rule), and by adding contexts to the axiom (which allows to eliminate the {\em left weakening} rule). Certain rules can also be simplified in this setting. Clearly, both calculi are equivalent. 

By a {\em sequent} we mean an expression $\Gamma \Rightarrow \Delta$ such that $\Gamma$ and $\Delta$ are (non simultaneously empty) finite multisets, and $\Delta$ contains at most one formula. The calculus {\bf LI} is given by the following axiom and inference rules:\\[4mm]
{\bf Axiom}
$$(Ax) \, \, \alpha,\Gamma\Rightarrow \alpha$$

\

\noindent
{\bf Structural Rules}

$$ (contr\Rightarrow) \, \displaystyle\frac{\alpha,\alpha,\Gamma\Rightarrow\Delta}{\alpha,\Gamma\Rightarrow\Delta} \hspace{2cm} (\Rightarrow w) \, \displaystyle\frac{\Gamma\Rightarrow}{\Gamma\Rightarrow \alpha} $$

$$(cut) \, \displaystyle\frac{\Gamma\Rightarrow\alpha \hspace{.5cm} \alpha, \Gamma\Rightarrow\Delta}{\Gamma\Rightarrow\Delta}$$

\

\noindent
{\bf Logical Rules}

$$\mbox{ ($\vee\Rightarrow$) }\displaystyle\frac{\alpha, \Gamma\Rightarrow\Delta \hspace{0,5cm}  \beta, \Gamma\Rightarrow\Delta}{\alpha\vee\beta
, \Gamma\Rightarrow\Delta} \hspace{1.1cm} \mbox{ ($\Rightarrow\vee_1$) } \, \, \displaystyle\frac{\Gamma\Rightarrow\alpha }{\Gamma\Rightarrow \alpha\vee\beta } \hspace{1.1cm} \mbox{ ($\Rightarrow\vee_2$) } \, \, \displaystyle\frac{\Gamma\Rightarrow\beta }{\Gamma\Rightarrow \alpha\vee\beta }$$

$$\mbox{ ($\wedge\Rightarrow$) }\displaystyle\frac{\alpha, \beta,\Gamma\Rightarrow\Delta}{\alpha\wedge\beta
, \Gamma\Rightarrow\Delta} \hspace{2cm} \mbox{ ($\Rightarrow\wedge$) } \, \, \displaystyle\frac{\Gamma\Rightarrow\alpha  \hspace{0,5cm} \Gamma\Rightarrow\beta }{\Gamma\Rightarrow \alpha\wedge\beta}$$

$$\mbox{ ($\to\Rightarrow$) } \, \,\displaystyle\frac{\Gamma\Rightarrow\alpha \hspace{0,5cm} \beta, \Gamma\Rightarrow \Delta}{\alpha \to\beta,\Gamma\Rightarrow \Delta}  \hspace{2cm} \mbox{ ($\Rightarrow\to$) } \, \,\displaystyle\frac{\alpha,  \Gamma\Rightarrow \beta }{\Gamma\Rightarrow \alpha \to\beta} $$

$$\mbox{($\neg\Rightarrow$)} \, \, \displaystyle\frac{ \Gamma \Rightarrow \alpha }{\neg\alpha, \Gamma\Rightarrow} \hspace{1.5cm} \mbox{($\Rightarrow\neg$)} \, \, \displaystyle\frac{\alpha, \Gamma \Rightarrow}{ \Gamma\Rightarrow \neg\alpha}$$

The consequence relation of  {\bf LI} will be denoted by $\vdash_{IPL}$. Hence, for any  set $\Gamma \cup \{\alpha\} \subseteq For(\Omega)$, $\Gamma \vdash_{IPL} \alpha$ iff there exists a finite set $\Gamma_0 \subseteq \Gamma$ such that the sequent $\Gamma_0 \Rightarrow \alpha$ is derivable in {\bf LI}. The logic generated by {\bf LI}, with consequence relation $\vdash_{IPL}$, is precisely \ipl.

\begin{remark} \label{MTD}
Being a presentation of \ipl, the deduction metatheorem (DM) holds for the consequence relation $\vdash_{IPL}$ associated to {\bf LI}:
$$(DM) \ \ \Gamma,\alpha \vdash_{IPL} \beta \ \ \mbox{ iff } \ \  \Gamma \vdash_{IPL}\alpha \to \beta$$
for every set of formulas $\Gamma \cup \{\alpha, \beta\}$. The validity of (DM) can be checked directly by using ($Ax$) and the rules for implication. Indeed, the `only if' part follows by ($\Rightarrow\to$). For the `if' part, it can be observed that $\Gamma_0,\alpha, \alpha \to \beta \Rightarrow\beta$ is provable in {\bf LI} for every finite multiset $\Gamma_0$ and formulas $\alpha$ and $\beta$: this is a consequence of applying ($\to\Rightarrow$) to the axioms $\Gamma_0,\alpha \Rightarrow\alpha$ and $\Gamma_0,\alpha, \beta \Rightarrow\beta$. From this, the result follows from the hypothesis  $\Gamma_0 \Rightarrow \alpha \to \beta$ (for a finite set $\Gamma_0 \subseteq \Gamma$) and ($cut$).
\end{remark}

\begin{definition} Let $\Pi$ be a derivation of a sequent $\Gamma \Rightarrow\Delta$ in {\bf LI} (that is, $\Pi$ is a tree  of sequents with root $\Gamma \Rightarrow\Delta$ and axioms as leaves). The {\em depth} of $\Pi$ is its height (as a tree). If  a sequent $\Gamma \Rightarrow\Delta$ admits a derivation in {\bf LI} of depth $k$ (for $k \geq 0$), we write $\Gamma \vdash_{IPL}^k \Delta$.
\end{definition}

The proof of soundness of  {\bf LI} w.r.t. the RNmatrix $\mathcal{R}(\mathcal{M}_{IPL})$ requires to establish, in a  precise way, soundness of  derivations of depth $k$ w.r.t. valuations of level $k$.

\begin{definition} Let $k \geq 0$ and $v \in \mathcal{L}_k^{IPL}$. We say that $v$ {\em satisfies a formula} $\alpha$, written as $v \models\alpha$, if $v(\alpha)=\bT$. If $\Gamma$ is a non-empty (multi)set of formulas, we say that $v$ {\em satisfies} $\Gamma$, denoted by $v\models\Gamma$, if $v\models \gamma$ for every $\gamma \in \Gamma$. We say that $v$ satisfies a sequent $\Gamma\Rightarrow\Delta$, denoted by $v\models \Gamma\Rightarrow\Delta$, if the following holds:\\[1mm]
(1) $\Delta=\emptyset$. In this case, $v \not\models\Gamma$, that is, $v(\beta)\in\{\bU,\bF\}$ for some $\beta \in \Gamma$; or\\[1mm]
(2) $\Delta=\{\gamma\}$, for some formula $\gamma$. In this case, $v \models\gamma$ whenever $v\models\Gamma$.

A sequent $\Gamma\Rightarrow\Delta$ is said to be {\em valid} in $\mathcal{L}_k^{IPL}$, denoted by $\Gamma\models_k\Delta$, if it is satisfied by any $v \in\mathcal{L}_k^{IPL}$. 
\end{definition}

\begin{proposition} \label{seq-rules-sound}
For every $k \geq 0$ and every sequent $\Gamma \Rightarrow\Delta$: if $\Gamma \vdash_{IPL}^k \Delta$, then $\Gamma\models_k\Delta$.
\end{proposition}
\begin{proof}
By induction on $k$.\\[1mm]
{\bf Base} $k=0$: Suppose that $\Gamma \vdash_{IPL}^0 \Delta$. Then  $\Gamma \Rightarrow\Delta$ is an instance of the axiom ($Ax$), that is:  $\Gamma \Rightarrow\Delta = \alpha,\Gamma' \Rightarrow \alpha$ for some multiset $\Gamma'$. Let $v \in \mathcal{L}_0^{IPL}$ such that $v \models \Gamma' \cup \{\alpha\}$. In particular, $v \models \alpha$, hence $v \models \Gamma \Rightarrow\Delta$. From this,  $\Gamma \models_0 \Delta$.\\[1mm]
{\bf Inductive step}: Suppose that, for any sequent  $\Gamma \Rightarrow\Delta$, if $\Gamma \vdash_{IPL}^j \Delta$ then  $\Gamma\models_j\Delta$, for every $j \leq k$, for a given $k \geq 0$ (Induction Hypothesis, IH). Let $\Gamma \Rightarrow\Delta$ be a sequent such that $\Gamma \vdash_{IPL}^{k+1} \Delta$. If $\Gamma \Rightarrow \Delta$ is an instance of the axiom, it can be proven as above that  $\Gamma\models_{k+1}\Delta$. Suppose now that $\Gamma \Rightarrow\Delta$ was obtained from other sequents by applying a rule of {\bf LI} in a derivation of depth $k+1$. We have the following cases:\\[1mm]
($contr\Rightarrow$): Suppose that  $\Gamma \Rightarrow\Delta= \alpha,\Gamma' \Rightarrow\Delta$ is obtained from $\alpha,\alpha,\Gamma'\Rightarrow\Delta$ by applying rule ($contr\Rightarrow$). Then $\alpha,\alpha,\Gamma' \vdash_{IPL}^k \Delta$ and so $\alpha,\alpha,\Gamma' \models_k \Delta$, by (IH). But the latter clearly implies that $\alpha,\Gamma' \models_k \Delta$ and so  $\alpha,\Gamma' \models_{k+1} \Delta$,  given that $\mathcal{L}_{k+1}^{IPL} \subseteq \mathcal{L}_k^{IPL}$.\\[1mm]
($\Rightarrow w$): Suppose that  $\Gamma \Rightarrow\Delta= \Gamma \Rightarrow\alpha$ is obtained from $\Gamma\Rightarrow$ by ($\Rightarrow w$). Then  $\Gamma \vdash_{IPL}^k $ and so, by (IH), $\Gamma\models_k$.  This means that, for every $v \in \mathcal{L}_k^{IPL}$, $v\not\models \Gamma$ and so $v\models\Gamma\Rightarrow \alpha$. From this it follows that $\Gamma\models_{k+1}\Delta$,  given that $\mathcal{L}_{k+1}^{IPL} \subseteq \mathcal{L}_k^{IPL}$.\\[1mm]
($cut$): Assume that $\Gamma \Rightarrow\Delta$ follows from $\Gamma\Rightarrow\alpha$ and $\alpha, \Gamma\Rightarrow\Delta$ by ($cut$) rule. Then  $\Gamma\vdash_{IPL}^k\alpha$ and $\alpha, \Gamma\vdash_{IPL}^k\Delta$ and so, by (IH),  $\Gamma \models_k\alpha$ and $\alpha, \Gamma \models_k\Delta$.
Suppose first that $\Delta=\emptyset$.  Let $v \in \mathcal{L}_{k+1}^{IPL}$ such that $v \models\Gamma$; in particular, $v \in \mathcal{L}_{k}^{IPL}$ and so $v \models\Gamma\Rightarrow\alpha$. From this it follows that  $v\models\alpha$. But this is impossible, given that $v \models \Gamma \cup\{\alpha\} \Rightarrow$, by (IH). From this, $v \not\models\Gamma$ for every $v\in \mathcal{L}_{k+1}^{IPL}$. That is, $v\models\Gamma\Rightarrow$, for every $v\in \mathcal{L}_{k+1}^{IPL}$. Now, suppose that $\Delta=\{\gamma\}$ for some formula $\gamma$, and let $v\in \mathcal{L}_{k+1}^{IPL}$ such that $v \models \Gamma$. Once again, since $v \models\Gamma\Rightarrow\alpha$ we infer that  $v\models\alpha$. But then $v \models \Gamma \cup \{\alpha\}$ and so $v\models\gamma$, given that $v \models\Gamma \cup\{\alpha\}\Rightarrow \gamma$. This means that $v \models \Gamma\Rightarrow\Delta$, for every $v\in \mathcal{L}_{k+1}^{IPL}$. That is, $\Gamma\models_{k+1}\Delta$.\\[1mm]
($\vee\Rightarrow$): Assume that  $\Gamma \Rightarrow\Delta=\alpha \vee\beta,\Gamma' \Rightarrow\Delta$ follows from $\alpha, \Gamma'\Rightarrow\Delta$ and $\beta, \Gamma'\Rightarrow\Delta$ by ($\vee\Rightarrow$). By (IH),  $\alpha, \Gamma' \models_k \Delta$ and $\beta, \Gamma' \models_k \Delta$. Suppose first that $\Delta=\emptyset$. By hypothesis, if $v \in \mathcal{L}_{k+1}^{IPL}$ is such that $v\models\Gamma'$, then $v\not\models\alpha$ and $v\not\models\beta$, and so $v\not\models\alpha \vee\beta$. This means that $v \models \Gamma' \cup \{\alpha\vee\beta\}\Rightarrow$, for every $v\in \mathcal{L}_{k+1}^{IPL}$. Now, if $\Delta=\{\gamma\}$ for some formula $\gamma$, let $v\in \mathcal{L}_{k+1}^{IPL}$. By hypothesis, if  $v\models\Gamma'$ and $v\models\alpha$, then $v\models\gamma$; and if  $v\models\Gamma'$ and $v\models\beta$, then $v\models\gamma$. From this, if  $v\models\Gamma'$ and $v\models\alpha\vee\beta$ then, by definition of $\mathcal{M}_{IPL}$, either $v\models\alpha$ or $v\models\beta$  and so, by hypothesis,  $v\models\gamma$. That is, $v \models \Gamma\Rightarrow\Delta$ for every $v\in \mathcal{L}_{k+1}^{IPL}$ and so  $\Gamma\models_{k+1}\Delta$.\\[1mm]
($\Rightarrow\vee_i$): It is enough to observe that, if either $v\models \alpha$ or $v\models \beta$ then $v\models \alpha\vee\beta$, by definition of $\mathcal{M}_{IPL}$.\\[1mm]
($\wedge\Rightarrow$): Suppose that $\Gamma \Rightarrow\Delta=\alpha \land \beta,\Gamma' \Rightarrow\Delta$ follows from $\alpha, \beta,\Gamma'\Rightarrow\Delta$ by ($\wedge\Rightarrow$). Assume first that $\Delta=\emptyset$. Then, if $v \in \mathcal{L}_{k+1}^{IPL}$ such that $v \models\Gamma'$  it follows that $v \not\models\{\alpha,\beta\}$, since $\mathcal{L}_{k+1}^{IPL} \subseteq \mathcal{L}_k^{IPL}$ and $\alpha, \beta,\Gamma'\models_{k}\Delta$ by (IH). From this, either  $v \not\models\alpha$ or  $v \not\models\beta$. By definition of $\mathcal{M}_{IPL}$,  $v \not\models\alpha \land\beta$. That is, $v \not\models\Gamma' \cup \{\alpha\land\beta\}$ for every $v\in \mathcal{L}_{k+1}^{IPL}$, i.e., $\alpha\wedge\beta, \Gamma'\models_{k+1}$. In turn, if  $\Delta=\{\gamma\}$ for some formula $\gamma$, let $v \in \mathcal{L}_{k+1}^{IPL}$ such that $v \models\Gamma' \cup \{\alpha\land\beta\}$. By definition of $\mathcal{M}_{IPL}$, $v \models\Gamma' \cup \{\alpha,\beta\}$ and so $v\models\gamma$, by (IH) and the fact that  $v \in \mathcal{L}_k^{IPL}$. That is, $v\models \Gamma' \cup \{\alpha \land\beta\}\Rightarrow\Delta$ for every $v\in \mathcal{L}_{k+1}^{IPL}$.\\[1mm]
($\Rightarrow\wedge$): It is enough to observe that, if both $v\models \alpha$ and $v\models \beta$ then $v\models \alpha\wedge\beta$, by definition of $\mathcal{M}_{IPL}$.\\[1mm]
($\to\Rightarrow$) Suppose that $\Gamma \Rightarrow\Delta= \alpha \to \beta,\Gamma'\Rightarrow \Delta$ is obtained from $\Gamma' \Rightarrow \alpha$ and $\beta, \Gamma' \Rightarrow\Delta$ by ($\to\Rightarrow$). By (IH),  $\Gamma'\models_{k} \alpha$ and $\beta, \Gamma' \models_{k}\Delta$. Assume first that $\Delta=\emptyset$. Let $v \in \mathcal{L}_{k+1}^{IPL}$ such that $v \models\Gamma' \cup\{\alpha \to \beta\}$. Then, $v  \in \mathcal{L}_{k}^{IPL}$ such that $v \models\Gamma'$ and so $v\models \alpha$, by hypothesis. Since $v \models\alpha \to \beta$ then  $v \models \beta$, by definition of  $\mathcal{M}_{IPL}$. That is, $v  \in \mathcal{L}_{k}^{IPL}$ is such that $v \models\Gamma'\cup \{\beta\}$, which contradicts the fact that $\beta, \Gamma' \models_{k}$. From this,  $v \not\models\Gamma' \cup\{\alpha \to \beta\}$ for every $v \in \mathcal{L}_{k+1}^{IPL}$, which means that $\alpha \to\beta,\Gamma'\models_{k+1}$. Now, suppose that $\Delta=\{\gamma\}$ for some formula $\gamma$.  Let $v \in \mathcal{L}_{k+1}^{IPL}$ such that $v \models\Gamma' \cup\{\alpha \to \beta\}$. By reasoning as above,   $v  \in \mathcal{L}_{k}^{IPL}$ is such that $v \models\Gamma'\cup \{\beta\}$. But then $v \models \gamma$, by hypothesis. This implies that $v \models\Gamma' \cup\{\alpha \to \beta\} \Rightarrow \gamma$ for every $v \in \mathcal{L}_{k+1}^{IPL}$. \\[1mm]
($\Rightarrow\to$): Assume that $\Gamma \Rightarrow\Delta= \Gamma\Rightarrow \alpha \to \beta$ is obtained from $\alpha,  \Gamma \Rightarrow \beta$ by ($\Rightarrow\to$). By (IH),  $\alpha,  \Gamma\models_{k} \beta$. Let $v \in \mathcal{L}_{k+1}^{IPL}$ such that $v \models \Gamma$. Suppose that $v(\alpha \to \beta)=\bF$. By definition of $\mathcal{M}_{IPL}$, $v(\alpha)=\bT$ and $v(\beta) \in \{\bU,\bF\}$. That is, $v \models\Gamma \cup\{\alpha\}$ but $v \not\models\beta$, which contradicts the fact that $\alpha,  \Gamma\models_{k} \beta$, given that $v \in \mathcal{L}_{k}^{IPL}$. Hence, $v(\alpha \to \beta)\neq\bF$. Suppose now that $v(\alpha \to \beta)=\bU$. By Definition~\ref{def:level_valuations_IPL} of level valuations, there exists $w \in \mathcal{L}_{k}^{IPL}$ such that  $w(\alpha \to \beta)=\bF$ and $w \models \Gamma$. By reasoning as above, $w(\alpha)=\bT$ and $w(\beta) \in \{\bU,\bF\}$. But then $w \models \Gamma \cup\{\alpha\}$ and $w \not\models\beta$, which contradicts the fact that  $\alpha,  \Gamma\models_{k} \beta$. From this, it follows that  $v(\alpha \to \beta)=\bT$. That is, $v\models \Gamma\Rightarrow \alpha\to\beta$ for every $v \in \mathcal{L}_{k+1}^{IPL}$.\\[1mm]
($\neg\Rightarrow$): Assume that $\Gamma \Rightarrow\Delta= \neg \alpha,\Gamma'\Rightarrow$ is obtained from $\Gamma' \Rightarrow \alpha$ by ($\neg\Rightarrow$). By (IH),  $\Gamma'\models_{k} \alpha$. Let $v \in \mathcal{L}_{k+1}^{IPL}$ such that $v \models \Gamma' \cup\{\neg\alpha\}$. Then, $v \in \mathcal{L}_{k}^{IPL}$ such that $v \models \Gamma'$ and so $v \models \alpha$, by hypothesis, which contradicts the fact that $v \models \neg\alpha$.  From this,  $v \not\models\Gamma' \cup\{\neg\alpha\}$ for every $v \in \mathcal{L}_{k+1}^{IPL}$, which means that $\neg\alpha,\Gamma'\models_{k+1}$. \\[1mm]
($\Rightarrow\neg$): Assume that $\Gamma \Rightarrow\Delta= \Gamma\Rightarrow \neg\alpha$ is obtained from $\alpha,  \Gamma \Rightarrow$ by ($\Rightarrow\neg$). By (IH),  $\alpha,  \Gamma\models_{k} $. Let $v \in \mathcal{L}_{k+1}^{IPL}$ such that $v \models \Gamma$. Suppose that $v(\neg\alpha)=\bF$. By definition of $\mathcal{M}_{IPL}$, $v(\alpha)=\bT$, hence $v \models\Gamma \cup\{\alpha\}$, which contradicts the fact that $\alpha,  \Gamma\models_{k}$, since $v \in \mathcal{L}_{k}$. Now, suppose that $v(\neg\alpha)=\bU$. By Definition~\ref{def:level_valuations_IPL} of level valuations, there exists $w \in \mathcal{L}_{k}^{IPL}$ such that  $w(\neg\alpha)=\bF$ and $w \models \Gamma$. As proven above, $w(\alpha)=\bT$ and so $w \models \Gamma \cup\{\alpha\}$, contradicting the fact that  $\alpha,  \Gamma\models_{k}$. From this, it follows that  $v(\neg\alpha)=\bT$. That is, $v\models \Gamma\Rightarrow \neg\alpha$ for every $v \in \mathcal{L}_{k+1}^{IPL}$.

This concludes the proof.
\end{proof}


\begin{theorem} [Soundness  of {\bf LI}  w.r.t.  $\mathcal{R}(\mathcal{M}_{IPL})$] \label{thm:sound:IPL}
For every set $\Gamma \cup \{\varphi\}$ of formulas over $\Omega$: $\Gamma \vdash_{IPL} \varphi$ implies that $\Gamma  \models_{\mathcal{R}(\mathcal{M}_{IPL})} \varphi$.
\end{theorem}
\begin{proof}
Suppose that  $\Gamma \vdash_{IPL} \varphi$. By definition, there exists a finite set $\Gamma_0 \subseteq \Gamma$ such that the sequent $\Gamma_0 \Rightarrow \varphi$ is derivable in {\bf LI}. But then, there exists $k\geq 0$ such that $\Gamma_0 \vdash_{IPL}^k \varphi$. By Proposition~\ref{seq-rules-sound}, $\Gamma_0\models_k \varphi$. This implies that $\Gamma  \models_{\mathcal{R}(\mathcal{M}_{IPL})} \varphi$, given that $\mathcal{L}_{IPL} \subseteq \mathcal{L}_{k}^{IPL}$.
\end{proof}

We will concentrate now on the  completeness proof of {\bf LI} w.r.t. $\mathcal{R}(\mathcal{M}_{IPL})$.

\begin{remark} \label{rem:sat-IPL}
Recall from Remark~\ref{varphi-sat-sets} that, in any Tarskian and finitary logic {\bf L}, if $\Gamma \nvdash_{\bf L} \varphi$ then there exists a $\varphi$-saturated set $\Delta$ in {\bf L} such that $\Gamma \subseteq \Delta$. This property holds, in particular, in the logic generated by {\bf LI} (that is, \textbf{IPL}). Moreover, it is easy to prove that, if $\Delta$ is a $\varphi$-saturated set in \textbf{IPL} then: $\alpha \in \Delta$ implies that $\neg\alpha \notin \Delta$; $\alpha \vee \beta \in \Delta$ iff $\alpha \in \Delta$ or $\beta \in \Delta$; $\alpha \land \beta \in \Delta$ iff $\alpha \in \Delta$ and $\beta \in \Delta$;  $\alpha \to \beta \in \Delta$ and $\alpha \in \Delta$ implies that $\beta \in \Delta$; and $\beta \in \Delta$ implies that $\alpha \to \beta \in \Delta$.
\end{remark}

\begin{definition} [Valuation associated to a $\varphi$-saturated set in \textbf{IPL}] \label{def:val-Delta-IPL}
Let $\Delta \subseteq For(\Omega)$ be a $\varphi$-saturated set in the logic generated by {\bf LI} (that is, \textbf{IPL}). The {\em valuation associated to $\Delta$ in \textbf{IPL}} is the function  $v_\Delta:For(\Omega) \to V_{IPL}$ defined inductively as follows (here, $p$ is a propositional variable and $\alpha,\beta \in For(\Omega)$):  \\[1mm]

$v_\Delta(p)= \left \{ \begin{tabular}{rl}
$\bT$ & if $p \in \Delta$,\\[1mm]
$\bF$ & if $p \notin \Delta$;\\
\end{tabular}\right.$ \\[4mm]

\

$v_\Delta(\neg\alpha)= \left \{ \begin{tabular}{rl}
$\bT$ & if $\neg\alpha \in \Delta$,\\[1mm]
$\bU$ & if $\neg\alpha \notin \Delta$ and $\alpha \notin \Delta$,\\[1mm]
$\bF$ & if $\neg\alpha \notin \Delta$ and $\alpha \in \Delta$;\\
\end{tabular}\right.$ \\[4mm]

\

$v_\Delta(\alpha\to\beta)= \left \{ \begin{tabular}{rl}
$\bT$ & if $\alpha\to\beta \in \Delta$,\\[1mm]
$\bU$ & if $\alpha\to\beta \notin \Delta$ and $\alpha \notin \Delta$,\\[1mm]
$\bF$ & if $\alpha\to\beta \notin \Delta$ and $\alpha \in \Delta$;\\
\end{tabular}\right.$ \\[4mm]

\

$v_\Delta(\alpha\land\beta)= \left \{ \begin{tabular}{rl}
$\bT$ & if $\alpha \in \Delta$ and $\beta \in \Delta$,\\[1mm]
$\bF$ & if $\alpha \notin \Delta$ or $\beta \notin \Delta$;\\
\end{tabular}\right.$ \\[4mm]

\

$v_\Delta(\alpha\vee \beta)= \left \{ \begin{tabular}{rl}
$\bT$ & if $\alpha \in \Delta$ or $\beta \in \Delta$,\\[1mm]
$\bF$ & if $\alpha \notin \Delta$ and $\beta \notin \Delta$.\\
\end{tabular}\right.$ \\[2mm]
\end{definition}

\begin{remark} \label{ren:canval}
The observations about the meaning of $\bU$ made after Definition~\ref{def:M_IPL} can be recast now.
Observe that $v_\Delta(\neg\alpha)=\bU$ whenever $\alpha$ is unknown in $\Delta$, that is: $\neg\alpha \notin \Delta$ and $\alpha \notin \Delta$. Hence, the failure of excluded middle $\alpha \vee \neg\alpha$  (a principle valid in classical logic but not in  \textbf{IPL}) occurs only for formulas $\alpha$ with such undetermined status in $\Delta$. Analogously, $v_\Delta(\alpha\to\beta)=\bU$ whenever $\alpha$ and $\alpha\to\beta$ are jointly unknown in $\Delta$, that is: neither $\alpha$ nor $\alpha\to\beta$ belong to $\Delta$. Hence, the failure of Peirce's law $\alpha \vee (\alpha \to\beta)$ (a principle valid in classical logic but not in  \textbf{IPL}) occurs only for formulas $\alpha$ and $\alpha\to\beta$ with such joint undetermined status in $\Delta$. The function $v_\Delta$ can be seen as a 3-valued characteristic map associated to $\Delta$, in which the third value $\bU$ reflects the undetermined status of a sentence that does not belong to $\Delta$. Only sentences of the form $\neg\alpha$ or $\alpha \to \beta$ can be unknown or undetermined (in this sense) w.r.t.  a $\varphi$-saturated set in \ipl.
\end{remark}

\begin{proposition} \label{prop:val-Delta-IPL}
Let $\Delta \subseteq For(\Omega)$ be a $\varphi$-saturated set in the logic generated by {\bf LI}. Then, the valuation $v_\Delta$ associated to $\Delta$ in \textbf{IPL} is a valuation in the Nmatrix $\mathcal{M}_{IPL}$ which belongs to $Val^{nd}_+(\mathcal{M}_{IPL})$ (recall Definition~\ref{def:0-level_valuations_IPL}) such that, for every $\alpha$: $v_\Delta(\alpha)=\bT$ iff $\alpha \in \Delta$.
\end{proposition}
\begin{proof} Observe that, by definition and  by Remark~\ref{rem:sat-IPL}, $v_\Delta(\alpha)=\bT$ iff $\alpha \in \Delta$, for every formula $\alpha$. Moreover, $v_\Delta(p) \in\{\bT,\bF\}$ for every propositional variable $p$.\\[1mm]
{\bf Negation:} Let us prove first that $v_\Delta(\neg\alpha) \in \neg^{IPL}(v_\Delta(\alpha))$ for every $\alpha$. If $v_\Delta(\neg\alpha)=\bT$ then $\neg\alpha \in \Delta$, by definition. Hence, $\alpha \notin \Delta$ and so $v_\Delta(\alpha) \in \{\bU,\bF\}$. From this, $v_\Delta(\neg\alpha) = \bT \in \{\bU,\bT\}=\neg^{IPL}(v_\Delta(\alpha))$. Now, if $v_\Delta(\neg\alpha)=\bU$ then $\alpha \notin \Delta$, by definition. Hence $v_\Delta(\alpha) \in \{\bU,\bF\}$ and so $v_\Delta(\neg\alpha) = \bU \in \{\bU,\bT\}=\neg^{IPL}(v_\Delta(\alpha))$. Finally,   if $v_\Delta(\neg\alpha)=\bF$ then $\alpha \in \Delta$, by definition. Hence $v_\Delta(\alpha)=\bT$ and so $v_\Delta(\neg\alpha) = \bF \in \{\bF\}=\neg^{IPL}(v_\Delta(\alpha))$.\\[1mm]
{\bf Implication:} There are 3 cases:\\
(I1) Suppose that $v_\Delta(\alpha \to \beta)=\bT$. Then $\alpha \to \beta \in \Delta$, by definition.\\
-- If $\alpha \notin \Delta$ then $v_\Delta(\alpha) \in \{\bU,\bF\}$. Hence, $v_\Delta(\alpha \to \beta)=\bT \in {\to}^{IPL}(v_\Delta(\alpha),v_\Delta(\beta))$.\\
-- If $\alpha \in \Delta$ then $\beta \in \Delta$, by (DM).  Hence  $v_\Delta(\alpha)=v_\Delta(\beta)=\bT$ and so $v_\Delta(\alpha \to \beta)=\bT \in \{\bT\}= {\to}^{IPL}(v_\Delta(\alpha),v_\Delta(\beta))$.\\
(I2)  Suppose that $v_\Delta(\alpha \to \beta)=\bU$. Then $\alpha \to \beta \notin \Delta$ and $\alpha \notin\Delta$, by definition. Moreover, $\beta \notin\Delta$, by Remark~\ref{rem:sat-IPL}. Then $v_\Delta(\alpha) \in \{\bU,\bF\}$ and  $v_\Delta(\beta) \in \{\bU,\bF\}$. Hence, $v_\Delta(\alpha \to \beta)=\bU \in \{\bU,\bT\}={\to}^{IPL}(v_\Delta(\alpha),v_\Delta(\beta))$.\\
(I3)  Suppose that $v_\Delta(\alpha \to \beta)=\bF$. Then $\alpha \to \beta \notin \Delta$ and $\alpha \in\Delta$, by definition. Moreover, $\beta \notin\Delta$, by Remark~\ref{rem:sat-IPL}. Then $v_\Delta(\alpha) =\bT$ and  $v_\Delta(\beta) \in \{\bU,\bF\}$. Hence, $v_\Delta(\alpha \to \beta)=\bF \in \{\bF\}={\to}^{IPL}(v_\Delta(\alpha),v_\Delta(\beta))$.\\[1mm]
{\bf Conjunction:} There are 2 cases:\\
(C1) Suppose that $v_\Delta(\alpha \land \beta)=\bT$. Then  $\alpha \in \Delta$ and  $\beta \in\Delta$, by definition. From this, $v_\Delta(\alpha)=v_\Delta(\beta)=\bT$. Hence, $v_\Delta(\alpha \land \beta)=\bT \in \{\bT\}= \land^{IPL}(v_\Delta(\alpha),v_\Delta(\beta))$.\\
(C2)  Suppose that $v_\Delta(\alpha \land \beta)=\bF$. Then, either  $\alpha \notin \Delta$ or  $\beta \notin\Delta$, by definition. From this, $v_\Delta(\alpha)  \in \{\bU,\bF\}$ or $v_\Delta(\beta)  \in \{\bU,\bF\}$. Hence, $v_\Delta(\alpha \land \beta)=\bF \in \{\bF\}= \land^{IPL}(v_\Delta(\alpha),v_\Delta(\beta))$.\\[1mm]
{\bf Disjunction:} There are 2 cases:\\
(D1)  Suppose that $v_\Delta(\alpha \vee \beta)=\bT$. Then, either  $\alpha \in \Delta$ or  $\beta \in\Delta$, by definition. From this, $v_\Delta(\alpha)=\bT$ or $v_\Delta(\beta)=\bT$. Hence, $v_\Delta(\alpha \vee \beta)=\bT \in \{\bT\}= \vee^{IPL}(v_\Delta(\alpha),v_\Delta(\beta))$.\\
(D2) Suppose that $v_\Delta(\alpha \vee \beta)=\bF$. Then  $\alpha \notin \Delta$ and  $\beta \notin\Delta$, by definition. From this, $v_\Delta(\alpha)  \in \{\bU,\bF\}$ and $v_\Delta(\beta)  \in \{\bU,\bF\}$. Hence, $v_\Delta(\alpha \land \beta)=\bF \in \{\bF\}= \vee^{IPL}(v_\Delta(\alpha),v_\Delta(\beta))$.

This concludes the proof.
\end{proof}

\begin{proposition} \label{prop:val-Delta-IPL-level}
Let $\Delta \subseteq For(\Omega)$ be a $\varphi$-saturated set in the logic generated by {\bf LI}. Then, the valuation $v_\Delta$ associated to $\Delta$ in \textbf{IPL} is a  level valuation in $\mathcal{M}_{IPL}$ such that, for every $\alpha$: $v_\Delta(\alpha)=\bT$ iff $\alpha \in \Delta$.
\end{proposition}

\begin{proof}
 By Proposition~\ref{prop:val-Delta-IPL}, $v_\Delta(\alpha)=\bT$ iff $\alpha \in \Delta$ for every $\alpha$. It will be shown, by induction on $n$, that $v_\Delta \in \mathcal{L}^{IPL}_n$ for every $n \geq 0$.\\[1mm]
{\bf Base} $n=0$: $v_\Delta \in \mathcal{L}^{IPL}_0 = Val^{nd}_+(\mathcal{M}_{IPL})$, by Proposition~\ref{prop:val-Delta-IPL}.\\[1mm]
{\bf Inductive step}: Assume that, for every $\varphi'$ and for every $\varphi'$-saturated set $\Delta'$, $v_{\Delta'} \in \mathcal{L}^{IPL}_k$ for every $k \leq n$, for a given $n \geq 0$ (Induction Hypothesis, IH). Let us prove that  $v_\Delta \in \mathcal{L}^{IPL}_{n+1}$. By IH, $v_\Delta \in  \mathcal{L}^{IPL}_n$. Let $\alpha \in For(\Omega)$ such that $v_\Delta(\alpha)=\bU$. By Definition~\ref{def:val-Delta-IPL}, $\alpha \neq p$ for every propositional variable $p$; and $\alpha \neq \beta \# \gamma$ for $\#\in \{\land,\vee\}$ and every formulas $\beta,\gamma$. Then, there are two cases to analyze:\\[1mm]
(1) $\alpha =\neg\beta$ for some $\beta$. By Definition~\ref{def:val-Delta-IPL}, $\neg\beta \not\in \Delta$ and  $\beta \not\in \Delta$. Suppose that $\Delta,\beta \vdash_{IPL} \neg\beta$. By the deduction metatheorem (DM),  $\Delta \vdash_{IPL} \beta \to \neg \beta$, and so the sequent  $\Delta_0 \Rightarrow \beta \to \neg \beta$ is provable in {\bf LI}, for some finite set $\Delta_0 \subseteq \Delta$. But, for every finite multiset $\Gamma$, the sequent $\Gamma,\beta \to \neg\beta  \Rightarrow \neg\beta$ is provable in {\bf LI}. Indeed, from axiom  $\Gamma,\beta \Rightarrow \beta$ it follows   $\Gamma,\beta, \neg \beta \Rightarrow$, by ($\neg\Rightarrow$). From this, and by using axiom $\Gamma,\beta \Rightarrow \beta$ once again, by ($\to\Rightarrow$) we infer $\Gamma,\beta, \beta \to \neg\beta  \Rightarrow$. Finally, $\Gamma,\beta \to \neg\beta  \Rightarrow \neg\beta$ follows by ($\Rightarrow\neg$). In particular, $\Delta_0,\beta \to \neg\beta  \Rightarrow \neg\beta$ is provable in {\bf LI} and so, by ($cut$) with $\Delta_0 \Rightarrow \beta \to \neg \beta$, it follows that  $\Delta_0 \Rightarrow \neg \beta$ is provable in {\bf LI}. That is, $\Delta \vdash_{IPL} \neg\beta$, hence $\neg\beta \in \Delta$, a contradiction. From this, $\Delta,\beta \nvdash_{IPL} \neg\beta$. Then, there exists a $\neg\beta$-saturated set $\Delta'$ in \textbf{IPL} such that $\Delta \cup \{\beta\} \subseteq \Delta'$. By IH, $v_{\Delta'} \in  \mathcal{L}^{IPL}_n$ such that, by Definition~\ref{def:val-Delta-IPL}, $v_{\Delta'}(\neg\beta)=\bF$. Moreover, for every $\delta \in  For(\Omega)$, if $v_\Delta(\delta)=\bT$ then  $v_{\Delta'}(\delta)=\bT$, since $\Delta \subseteq \Delta'$.\\[1mm]
(2) $\alpha =\beta \to \gamma$ for some $\beta$ and $\gamma$. By Definition~\ref{def:val-Delta-IPL}, $\beta \to \gamma \not\in \Delta$ and  $\beta \not\in \Delta$. Suppose that $\Delta,\beta \vdash_{IPL} \beta \to \gamma$. By (DM),  $\Delta,\beta \vdash_{IPL} \gamma$ and so, by (DM) again,  $\Delta \vdash_{IPL} \beta \to \gamma$, a contradiction. From this, $\Delta,\beta \nvdash_{IPL} \beta \to \gamma$. Then, there exists a $(\beta \to \gamma)$-saturated set $\Delta''$ in \textbf{IPL} such that $\Delta \cup \{\beta\} \subseteq \Delta''$. By IH, $v_{\Delta''} \in  \mathcal{L}^{IPL}_n$ such that, by Definition~\ref{def:val-Delta-IPL}, $v_{\Delta''}(\beta \to \gamma)=\bF$. Moreover, for every $\delta \in  For(\Omega)$, if $v_\Delta(\delta)=\bT$ then  $v_{\Delta''}(\delta)=\bT$, since $\Delta \subseteq \Delta''$.

From the previous analysis we conclude that  $v_\Delta \in \mathcal{L}^{IPL}_{n+1}$. This shows that $v_\Delta \in \mathcal{L}^{IPL}_{n}$ for every $n \geq 0$ and so  $v_\Delta \in \mathcal{L}_{IPL}$.
\end{proof}

\begin{theorem}  [Completeness  of {\bf LI} w.r.t.  $\mathcal{R}(\mathcal{M}_{IPL})$] \label{thm:compl:IPL}
For every $\Gamma$ and $\varphi$: $\Gamma  \models_{\mathcal{R}(\mathcal{M}_{IPL})} \varphi$ implies that  $\Gamma \vdash_{IPL} \varphi$.
\end{theorem}
\begin{proof}
Suppose that $\Gamma \nvdash_{IPL} \varphi$. Then, there exists a $\varphi$-saturated set $\Delta$ in {\bf LI} such that $\Gamma \subseteq \Delta$. By Proposition~\ref{prop:val-Delta-IPL-level}, $v_\Delta$ is a  level valuation in $\mathcal{M}_{IPL}$ such that $v_\Delta(\alpha)=\bT$ for every $\alpha \in \Gamma$, but $v_\Delta(\varphi)\neq \bT$. This shows that  $\Gamma  \not\models_{ \mathcal{R}(\mathcal{M}_{IPL})} \varphi$.
\end{proof}

\begin{remark} \label{rem:level-val}
Observe that the semantical characterization of \ipl\ w.r.t. level valuation semantics is quite interesting from a theoretical point of view. However, it has little interest from a practical point of view, given that level valuations are infinitary objects, and it is far from obvious how to check validity of a formula in \ipl\ by using $\mathcal{R}(\mathcal{M}_{IPL})$. The next step is the definition of an (algorithmic) way to construct {\em intuitionistic truth-tables}, based precisely on level valuations.
\end{remark}

From soundness and completeness of \ipl\ w.r.t.  $\mathcal{R}(\mathcal{M}_{IPL})$, several important consequences can be obtained. Some of these consequence will be extremely useful to prove the soundness and completeness of the (non-deterministic) truth-tables decision method for \ipl\ to be defined in the next section.

\begin{corollary} [Finitariness for $\mathcal{R}(\mathcal{M}_{IPL})$] \label{compact_RN_IPL}  Let $\Gamma \cup\{\varphi\} \subseteq For(\Omega)$. If $\Gamma \models_{\mathcal{R}(\mathcal{M}_{IPL})} \varphi$  then $\Gamma_0 \models_{\mathcal{R}(\mathcal{M}_{IPL})} \varphi$ for some finite subset $\Gamma_0$ of $\Gamma$. 
\end{corollary}
\begin{proof} It follows from  Theorems~\ref{thm:compl:IPL},  \ref{thm:sound:IPL}  and the fact that  {\bf LI} is finitary.
\end{proof}

\begin{corollary} [Deduction metatheorem for $\mathcal{R}(\mathcal{M}_{IPL})$] \label{MTD_RN_IPL}  Let $\Gamma \cup\{\alpha,\beta\} \subseteq For(\Omega)$. Then, $\Gamma, \alpha \models_{\mathcal{R}(\mathcal{M}_{IPL})} \beta$  if, and only if, $\Gamma \models_{\mathcal{R}(\mathcal{M}_{IPL})} \alpha \to \beta$. 
\end{corollary}
\begin{proof} It follows from  Theorems~\ref{thm:compl:IPL},  \ref{thm:sound:IPL}  and the fact that  ${\bf LI}$ satisfies the deduction metatheorem (recall Remark~\ref{MTD}).
\end{proof}

\begin{proposition} \label{closed-T} Let $v \in \mathcal{L}_{IPL}$ and let $\Gamma=\{\alpha \in For(\Omega)  \mid v(\alpha)=\bT\}$. Then, $\Gamma$ is a closed theory in \ipl, i.e.: if $\Gamma \vdash_{IPL} \alpha$ then $\alpha \in \Gamma$.
\end{proposition}
\begin{proof}
Let $\Gamma=\{\alpha \in For(\Omega)  \mid v(\alpha)=\bT\}$ and suppose that $\Gamma \vdash_{IPL} \alpha$. By soundness of \ipl\ w.r.t. level valuation semantics  (Theorem~\ref{thm:sound:IPL}), $\Gamma  \models_{\mathcal{R}(\mathcal{M}_{IPL})} \alpha$. Given that $v \in \mathcal{L}_{IPL}$ such that $v\models\Gamma$, it follows that $v \models\alpha$. That is, $v(\alpha)=\bT$ and so $\alpha \in \Gamma$. 
\end{proof}

The following result will be crucial for proving soundness and completeness of intuitionistic truth-tables:

\begin{proposition} \label{prop-level-val}
 Let $v \in \mathcal{L}_{IPL}$, and $\alpha \in For(\Omega)$ such that $v(\alpha)=\bU$. Then, there exists $w \in  \mathcal{L}_{IPL}$ such that $w(\alpha)=\bF$ and $w(\gamma)=\bT$ for every formula $\gamma$ such that $v(\gamma)=\bT$.
\end{proposition}
\begin{proof}
$v \in \mathcal{L}_{IPL}$, and $\alpha \in For(\Omega)$ such that $v(\alpha)=\bU$. Since, in particular, $v \in Val^{nd}_+(\mathcal{M}_{IPL})$ (recall Definitions~\ref{def:0-level_valuations_IPL} and~\ref{def:level_valuations_IPL}), $v(\alpha)=\bU$ implies that $\alpha$ cannot be a propositional variable.
By definition of the conjunction and disjunction tables in $\mathcal{M}_{IPL}$ (Figure~\ref{table:nm_IPL}), and since $v(\alpha)=\bU$, it follows that $\alpha$ cannot be either a conjunction or a disjunction.
Let $\Gamma=\{\gamma \in For(\Omega)  \mid v(\gamma)=\bT\}$. Then, $\alpha \not\in \Gamma$.  By the observations above, there are only two cases to consider:\\[1mm]
(1) $\alpha=\neg\beta$ for some $\beta$. As noted above, $\alpha \not\in\Gamma$, i.e.,  $\neg\beta \not\in \Gamma$. Hence $\Gamma \nvdash_{IPL}\neg\beta$, given that $\Gamma$ is closed by Proposition~\ref{closed-T}. If $\beta \in \Gamma$ then $v(\beta)=\bT$ and so, by definition of $\mathcal{M}_{IPL}$, $v(\alpha)=v(\neg\beta)=\bF$, a contradiction. Then, $\beta \not\in \Gamma$. Suppose that $\Gamma,\beta \vdash_{IPL}\neg\beta$. As in the proof of Proposition~\ref{prop:val-Delta-IPL-level} we infer that  $\Gamma \vdash_{IPL}\neg\beta$. But this lead us to a contradiction since, as pointed out before,  $\Gamma \nvdash_{IPL}\neg\beta$. From this, $\Gamma,\beta \nvdash_{IPL}\neg\beta$. Thus, there exists a $\neg\beta$-saturated set $\Delta$ in \ipl\ such that $\Gamma \cup\{\beta\} \subseteq \Delta$. By  Proposition~\ref{prop:val-Delta-IPL-level}, the function $v_\Delta$ is a level valuation such that $v_\Delta(\alpha)=\bF$ (since $v_\Delta(\beta)=\bT$) and $v_\Delta(\gamma)=\bT$ for every formula $\gamma$ such that $v(\gamma)=\bT$, as required.\\[1mm]
(2) $\alpha=\beta \to \delta$ for some $\beta$ and $\delta$. Since $\Gamma,\delta \vdash_{IPL} \beta \to \delta$ and  $\Gamma \nvdash_{IPL} \beta \to \delta$ then $\Gamma \nvdash_{IPL} \delta$.  Suppose that $\Gamma,\beta \vdash_{IPL}\beta \to \delta$. By (DM), $\Gamma,\beta \vdash_{IPL} \delta$ and so, by (DM) again, $\Gamma \vdash_{IPL}\beta \to \delta$, a contradiction. From this, $\Gamma,\beta \nvdash_{IPL}\beta \to \delta$. Thus, there exists a $(\beta\to\delta)$-saturated set $\Delta$ in \ipl\ such that $\Gamma \cup\{\beta\} \subseteq \Delta$.  By reasoning as above, we conclude that $\delta \not\in \Delta$.  By  Proposition~\ref{prop:val-Delta-IPL-level}, $v_\Delta\in  \mathcal{L}_{IPL}$ is such that $v_\Delta(\alpha)=\bF$ (since $v_\Delta(\beta)=\bT$ and $v(\delta)\neq \bT$) and $v_\Delta(\gamma)=\bT$ for every formula $\gamma$ such that $v(\gamma)=\bT$, as required.\\[1mm]
\end{proof}




\section{A truth-tables procedure  for \ipl} \label{tables-IPL}

After obtaining a suitable RNmatrix for \ipl, the next step  is the definition of a decision procedure from this. As in the case of \sfo, this procedure is given  by means of branching truth-tables with an algorithm for deleting inadequate rows.

\subsection{Intuitionistic partial (level) valuations and intuitionistic truth-tables} \label{sect:Int_truth_tables}

\begin{definition} Consider the sets 
$$CS(\Omega)= \{\Lambda \subseteq For(\Omega) \mid  \mbox{ $\Lambda$ is non-empty and closed under subformulas}\},$$
 $$FCS(\Omega)= \{\Lambda \subseteq For(\Omega) \mid  \mbox{ $\Lambda$ is finite, non-empty and closed under subformulas}\}.$$
\end{definition}

\begin{definition}[Partial valuation  in $\mathcal{M}_{IPL}$]\label{def:partial_valuationIPL}
 Let $\Lambda \in CS(\Omega)$.  A {\em partial valuation in $\mathcal{M}_{IPL}$} is a function $\tilde v_p :\Lambda \rightarrow V_{IPL}$ such that, for every $\alpha,\beta \in \Lambda$:
 
 \begin{itemize}
 \item[--] if $\alpha \in \mathcal{P} \cap \Lambda$ then $\tilde v_p(\alpha) \in \{\bT,\bF\}$;
 \item[--] if $\neg\alpha \in \Lambda$ then $\tilde v_p(\neg \alpha) \in \neg^{IPL}(\tilde v_p(\alpha))$;
 \item[--] if  $\# \in \{\to,\vee,\land\}$ and $\alpha \# \beta\in \Lambda$ then  $\tilde v_p(\alpha \# \beta) \in \#^{IPL}(\tilde v_p(\alpha),\tilde v_p(\beta))$.
 \end{itemize}
 Let $iPV(\Lambda)$ be the set of  partial valuations in $\mathcal{M}_{IPL}$ with domain $\Lambda$.
\end{definition}

\begin{definition}[Intuitionistic partial level valuation over $\Lambda$]\label{def:partial_level_val_IPL}
    Let $\Lambda \in FCS(\Omega)$. An {\em intuitionistic partial level valuation} in $\mathcal{M}_{IPL}$ over $\Lambda$ is a  partial valuation $\tilde v_p \in iPV(\Lambda)$  satisfying the following:

    \begin{enumerate}
        \item[] $\forall \alpha \in \Lambda$ such that $\tilde v_p(\alpha) = \bU$, there exists an intuitionistic partial level valuation $\tilde w_p$ in $\mathcal{M}_{IPL}$ over $\Lambda$ such that $\tilde w_p(\alpha) = \bF$ and, $\forall \beta \in \Lambda$, $\tilde w_p(\beta) = \bT$ whenever $\tilde v_p(\beta) = \bT$.
    \end{enumerate}
The set of intuitionistic partial level valuations in $\mathcal{M}_{IPL}$ over $\Lambda$ will be denoted by $iPLV(\Lambda)$.
\end{definition}

\begin{remark} \label{rem:PLVsIPL}
Given $v,w$ and $\alpha$ let $\mathsf{P}_\Lambda(v,w,\alpha)$ iff $w(\alpha) = \bF$ and, $\forall \beta \in \Lambda$, $w(\beta) = \bT$ whenever $v(\beta) = \bT$. If $\Lambda \in FCS(\Omega)$ then\\[2mm]
$\begin{array}{lll}
iPLV(\Lambda)&=& \{\tilde v_p \in iPV(\Lambda) \mid \forall \alpha \in \Lambda \big(\tilde v_p(\alpha) = \bU \mbox{ implies that } \mathsf{P}_\Lambda(\tilde v_p,\tilde w_p,\alpha)\\[1mm] 
&&\hspace*{7mm}\mbox{ for some $\tilde w_p \in iPLV(\Lambda)$}\big)\}.
\end{array}$
\end{remark}

\begin{remark} \label{non-circularity}
As in the case of {\bf S4} discussed in Remark~\ref{rem:PLVs}, the impredicative definition of $iPLV(\Lambda)$ does not lead to circularity. Indeed, by using an algorithm (see Algorithm~1 below) similar to the Algorithm~3.4 given in~\cite[p. 15]{gratz_truth_2022} for $\mathcal{M}_{S4}$, Definition~\ref{def:partial_level_val_IPL} can be transformed into a predicative one (see Definition~\ref{def:partial'_level_val_IPL} and Proposition~\ref{prop:analiticityIPL}). This is the consequence of the fact that  $iPLV(\Lambda)$ is obtained as the 
output of Algorithm~1, being the unique subset of $iPV(\Lambda)$ satisfying the condition of
Remark~\ref{rem:PLVsIPL}, as it will be shown in Proposition~\ref{algorithm-iPLV}.
The set  $iPLV(\Lambda)$ is precisely a (intuitionistic) truth-table, in the wider sense proposed here, based on  $\mathcal{R}(\mathcal{M}_{IPL})$. It is obtained by means of a mechanical procedure, where its elements correspond to the rows of the table, and its columns correspond to the values assigned in each row to the elements of $\Lambda$.
\end{remark}



\begin{notation}\label{remark:requirements} 
Given $\Lambda \in FCS(\Omega)$ with $n\geq 1$ elements, consider an arrangement $\alpha_1, \ldots,\alpha_n$ of this set such that the complexity of $\alpha_i$ is less than or equal to the complexity of  $\alpha_j$ if $i < j$ (so the atomic formulas occurring in $\Lambda$ appear as the first elements of the sequence, in arbitrary order). In this way, any subformula of  $\alpha_j$ appears as  $\alpha_i$ for some unique $i\leq j$. Suppose that the set $iPV(\Lambda)$ has $m \geq 1$ elements. Consider now an arrangement $v_p^1, \ldots, v_p^m$ of its elements, ordered by lexicographical order, by stipulating that $\bF < \bU < \bT$. In this way, $iPV(\Lambda)$ can be represented as a 2-dimensional array (i.e., an array of arrays, or a matrix) $V$ such that each row $V[i]$, for $1 \leq i \leq m$, corresponds to the $ith$ element $v_p^i$ of $iPV(\Lambda)$. In turn, $V[i][j]$ is the truth-value assigned to the formula $\alpha_j \in \Lambda$ by the partial valuation $V[i]= v_p^i$. That is, $V[i][j] = v_p^i(\alpha_j)$. This notation will be used in Algorithm~1, to be described in the sequel, also for subsets of  $iPV(\Lambda)$.
\end{notation}

\

\noindent {\bf Algorithm~1 (in Figure~\ref{fig:algorithm1}) for obtaining $iPLV(\Lambda)$.} \\[1mm]
{\em Let $\Lambda \in FCS(\Omega)$, and consider the notation stipulated in Notation~\ref{remark:requirements}. Hence,  $\Lambda =\{\alpha_1, \ldots,\alpha_n\}$ and  $iPV(\Lambda) = \{v_p^1, \ldots, v_p^m\}$ (which can also be seen as sequences). 
Let $V:=iPLV(\Lambda)$. For each element $v_p^i$ of $V$, if $v_p^i(\alpha_j)=\bU$, determine if there is some $v_p^{i'}$ such that $\mathsf{P}_\Lambda(v_p^i,v_p^{i'},\alpha_j)$. This is represented by the function \texttt{isSupported}$(V,i,j)$ in Figure~\ref{fig:algorithm1}.
If, for some $j$, \texttt{isSupported}$(V,i,j)$ is false (meaning no valuation in $V$ supports $v_p^i$ w.r.t. $\alpha_j$) then $v_p^i$ is removed from the candidate set for 
$iPLV(\Lambda)$ and will not be considered in subsequent algorithm iterations.
After checking all the elements of $V$ using the elements of $V$ itself, let $IPV$ be the set of remaining valuations. If no valuation was removed, i.e., if $IPV=V$, the algorithm stops at this point and returns $IPV$ as being the set $iPLV(\Lambda)$ of partial level valuations. The set $IPV$ is the result of a {\em refinement} function applied to $V$, represented by \texttt{refine}$(V)$ in Figure~\ref{fig:algorithm1}. In case $IPV\neq V$, it means that $IPV$ is a proper subset of $V$ (i.e., some valuations were effectively removed in this round of the process). The set $IPV$ needs to be checked again with the function \texttt{refine}, to guarantee that every row is still supported.  Then, the set $V$ is updated to be equal to $IPV$  (and so $V$ is, again, the universe of partial valuations to be checked), it is renumbered as a sequence $V[1], \ldots, V[m']$ (where $m'<m$ is the cardinal or $IPV$) and then we repeat the process with the updated $V$ until no changes occur. The set $iPLV(\Lambda)$ is then defined to be the fixpoint $V=IPV$  (see Figure~\ref{fig:algorithm1}).}

\

\begin{figure}[h!]
    \centering
    \includegraphics[scale=0.30]{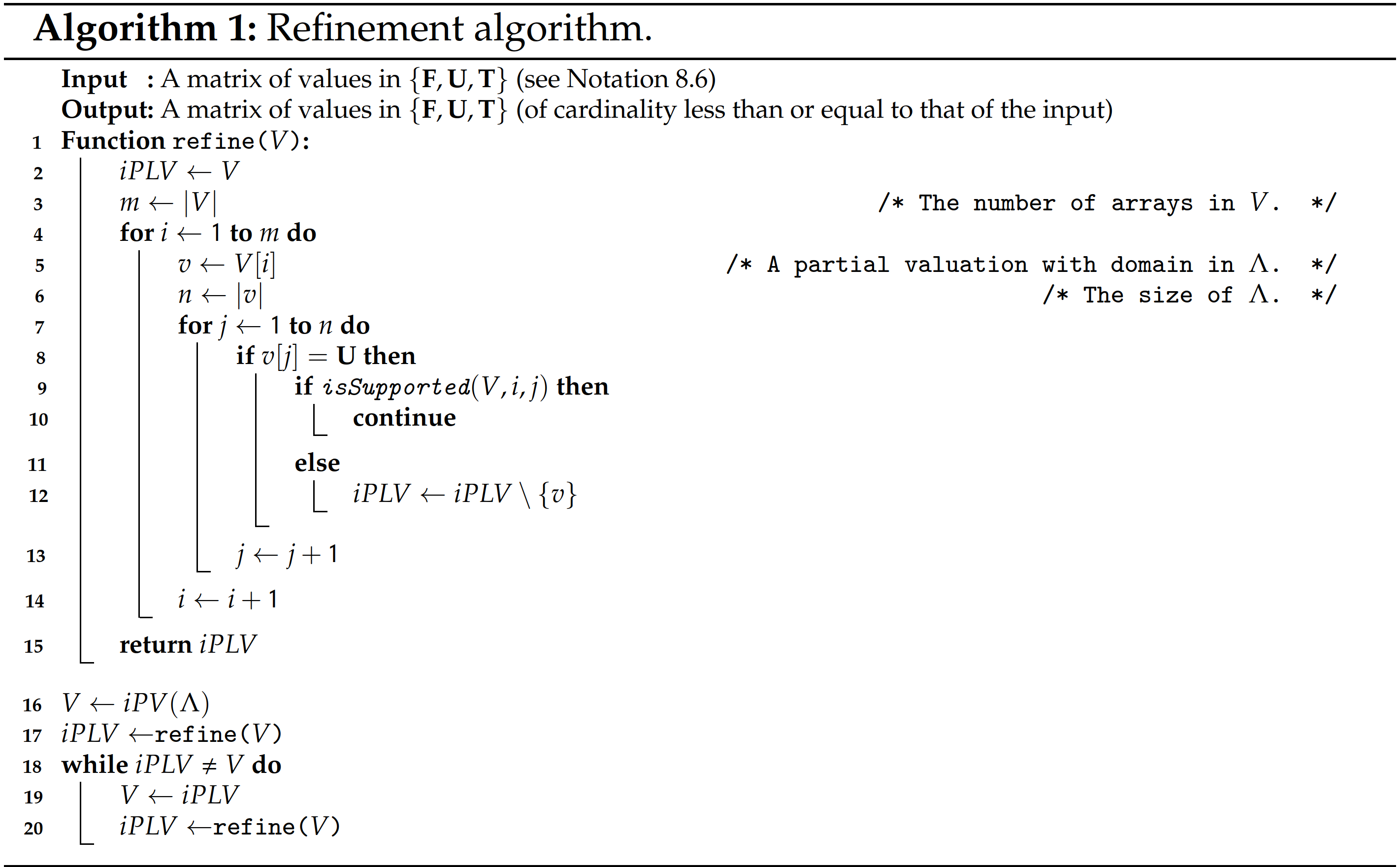}
    \caption{Algorithm for obtaining the set $iPLV(\Lambda)$.}\label{fig:algorithm1}
\end{figure}

Next result shows that, thanks to Algorithm~1, the set $iPLV(\Lambda)$ is well-defined, it is unique, and it is non-empty:

\begin{proposition} [Existence and uniqueness of intuitionistic partial level valuations] \label{algorithm-iPLV}
    Let $\Lambda \in FCS(\Omega)$. Then, there exists a unique set $iPLV(\Lambda) \subseteq iPV(\Lambda)$ satisfying the condition of Remark~\ref{rem:PLVsIPL}:\\
    
$\begin{array}{lll}
iPLV(\Lambda)&=& \{\tilde v_p \in iPV(\Lambda) \mid \forall \alpha \in \Lambda \big(\tilde v_p(\alpha) = \bU \mbox{ implies that } \mathsf{P}_\Lambda(\tilde v_p,\tilde w_p,\alpha)\\[1mm] 
&&\hspace*{7mm}\mbox{ for some $\tilde w_p \in iPLV(\Lambda)$}\big)\}.
\end{array}$\\
Moreover, the set $iPLV(\Lambda)$ is non-empty, having at least $2^k$ elements, where $k$ is the number of different propositional variables occurring in $\Lambda$.
\end{proposition}
\begin{proof}
Given $\Lambda \in FCS(\Omega)$, define the sequences $(\alpha_1, \ldots,\alpha_n)$ and $(v_p^1, \ldots, v_p^m)$ representing $\Lambda$ and $iPV(\Lambda)$ as in Notation~\ref{remark:requirements}, and then execute Algorithm~1 as described. Notice that $iPV(\Lambda)$ is a finite set,  and the algorithm operates on a subset $V \subseteq iPV(\Lambda)$ in each round. If the algorithm does not stop in a given round, it deletes some elements from $V$ and restarts the process with a new subset that is strictly smaller than the previous one.
This shows that the process will eventually terminate. There are two possible termination conditions: either every element of $V$ is eliminated, or a set $IPV = V$ is produced. In the second case, every element of $IPV$ is supported (as no removal is possible at this point); hence, $iPLV(\Lambda):=IPV$ satisfies the condition of Remark~\ref{rem:PLVsIPL}.
In turn, the first possibility (arriving to a set $IPV=\emptyset$) is impossible by the following reason: every partial valuation $v :\Lambda \rightarrow V_{IPL}$ such that $v(\alpha) \in \{\bT,\bF\}$ for every $\alpha \in \Lambda$ (i.e., every classical partial valuation) is never removed from $iPV(\Lambda)$ by means of Algorithm~1. Indeed, any classical valuation $v$ never satisfies that $v(\alpha_j)=\bU$ for some $j$, hence it will never be removed. Since there are exactly $2^k$ classical valuations in $iPV(\Lambda)$ (where $k$ is the number of different propositional variables occurring in $\Lambda$), the set $IPV$ produced by Algorithm~1 is non-empty, having at least $2^k$ elements. We have proven, therefore, the existence of a (non-empty) subset $IPV$ of $iPV(\Lambda)$ satisfying the condition of Remark~\ref{rem:PLVsIPL}.

To prove the uniqueness of such a set of partial valuations, observe that the order of the valuations $V[i]$ in the array $V$ is irrelevant for Algorithm~1. The algorithm's execution does not depend on this initial ordering. In each round, every valuation $V[i]$ is tested against every formula $\alpha_j$ in $\Lambda$. If $V[i][j]=\bU$, the function \texttt{isSupported}$(V,i,j)$ checks all other valuations $V[i']$ in $V$ to determine whether $\mathsf{P}_\Lambda(V[i],V[i'],\alpha_j)$ holds. If this condition fails for every $i'$, $V[i]$ is removed in the next step.
This procedure is clearly independent of the lexicographical order used in the formal presentation of Algorithm~1: that order was chosen solely to provide a precise description of the algorithm's steps.
In other words, regardless of the initial arrangement of $iPV(\Lambda)$, the first round will eliminate the same valuations and produce a unique set $V_1$ for the second stage. The same argument applies: any arrangement of $V_1$ yields the same subset $V_2$ after the second round. By induction, the algorithm's output \textemdash\ the fixpoint $V_s = IPV$ reached in the final step $s$ \textemdash\ must be unique.

This completes the proof.
\end{proof}

\begin{remark}
Observe that it is possible to have a set $iPLV(\Lambda)$ depending on $k$ different propositional variables $p_1,\ldots,p_k$ having {\em exactly} $2^k$ elements. It is enough considering a formula $\alpha(p_1,\ldots,p_k)$ constructed by using exclusively conjunction and disjunction. Let $\Lambda$ be the set of subformulas of $\alpha$. Clearly, the  truth-table  $iPLV(\Lambda)$ for $\alpha$  only contains classical values $\bT$ and $\bF$, coinciding so with $iPV(\Lambda)$. It has, therefore, exactly $2^k$ elements. Moreover, by expanding this table with a column for $\alpha \to \alpha$, i.e., by considering $iPV(\Lambda \cup\{\alpha \to\alpha\})$, the last column only contains values $\bU$ or $\bT$. But then the rows assigning $\bU$ to $\alpha \to \alpha$ are not supported by any other row, hence they are removed in the first round of Algorithm~1, arriving so to a set   $iPLV(\Lambda \cup\{\alpha \to\alpha\})$ with $2^k$ elements.
\end{remark}

%
%
%
%
%
%
%

\begin{definition} \label{def:intui:truth-tables}
An {\em intuitionistic truth-table in $\mathcal{M}_{IPL}$} is an exhaustive list of every and only  intuitionistic partial level valuations for a given finite domain closed under subformulas.
\end{definition}

\begin{remark}
    Each of these partial valuations corresponds to a row of the table.
\end{remark}

\paragraph{Example} In what follows, we present an example of a truth-table produced by Algorithm~1 which keep track of the dependencies of each partial valuation\footnote{This example was automatically generated by a Rocq module for RNmatrices written by R. Leme, see \ \url{https://github.com/renatoleme/Forest} }. The first table (Table~\ref{T1}) is the full table for $\neg\neg(p \lor \neg p)$, a valid intuitionistic formula, and was obtained by applying the Nmatrix for \ipl\ in the usual way.

\begin{figure}[h]
\centering
\begin{subfigure}[b]{.45\linewidth}
\includegraphics[width=\linewidth]{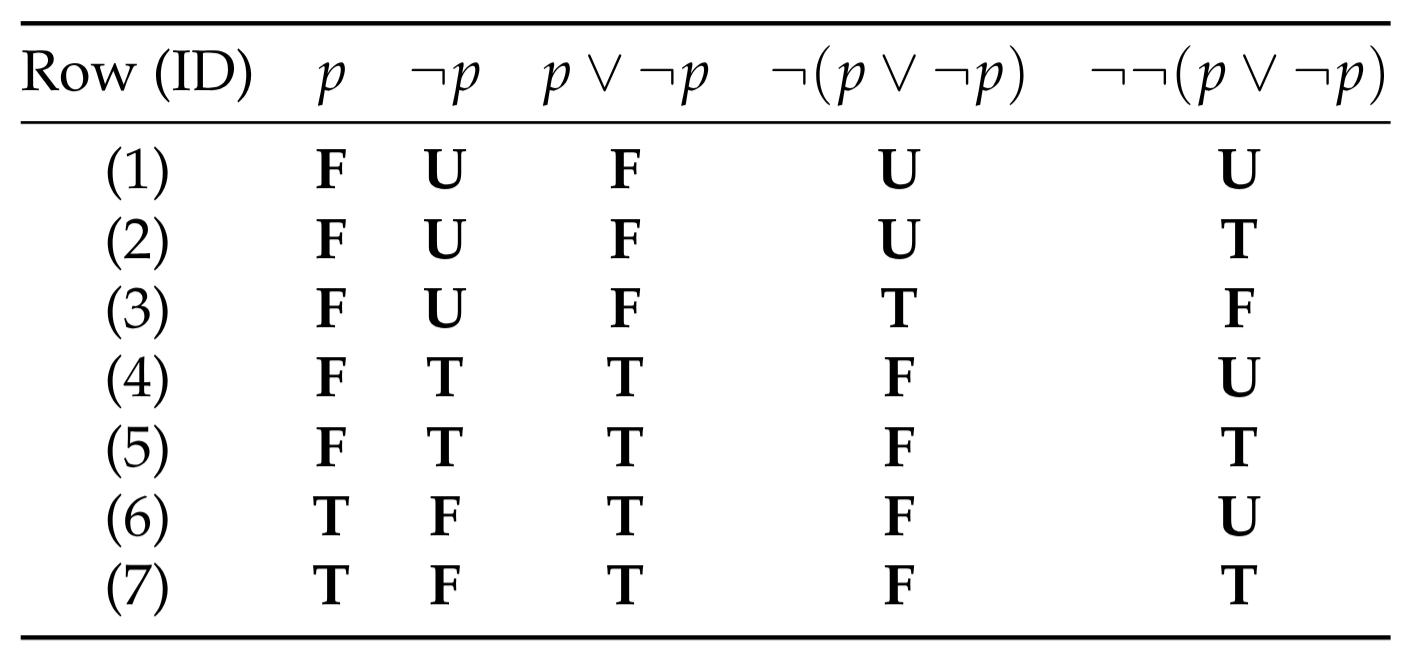}
\caption{Initial table.}\label{T1}
\end{subfigure}
\begin{subfigure}[b]{.45\linewidth}
\includegraphics[width=\linewidth]{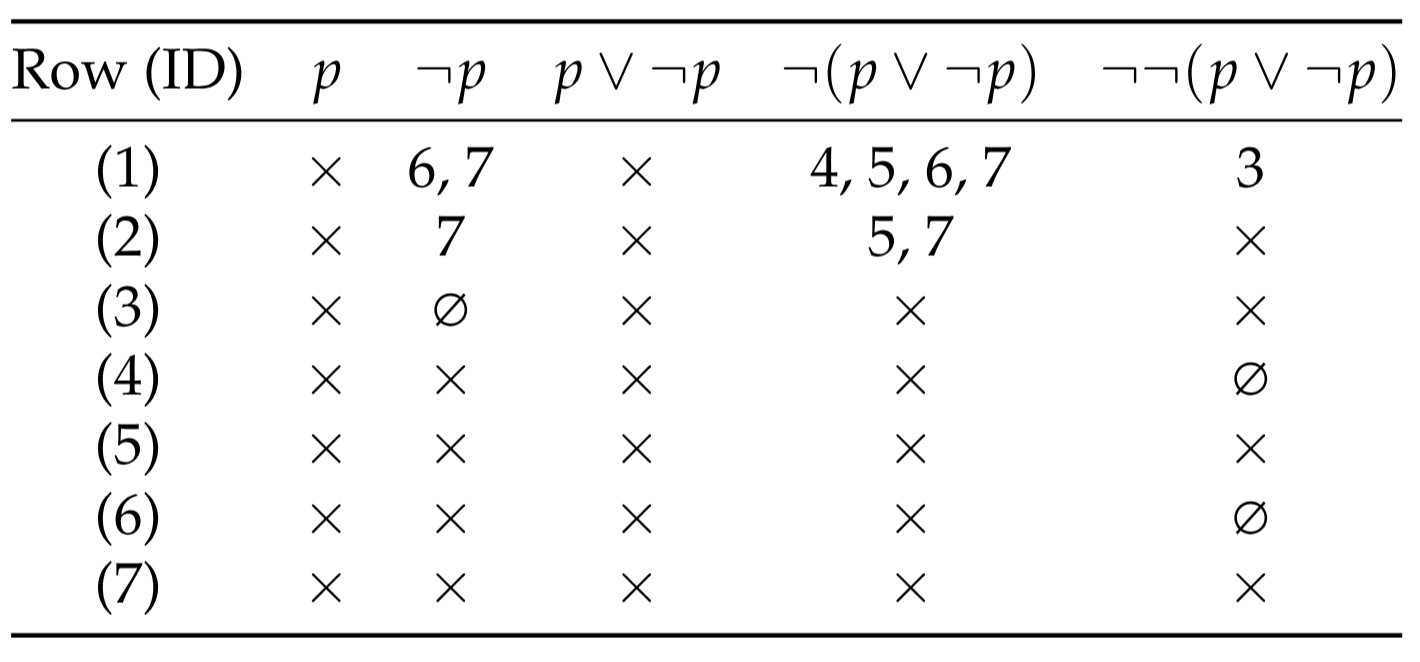}
\caption{First cycle of dependencies.}\label{T2}
\end{subfigure}
\begin{subfigure}[b]{.45\linewidth}
\includegraphics[width=\linewidth]{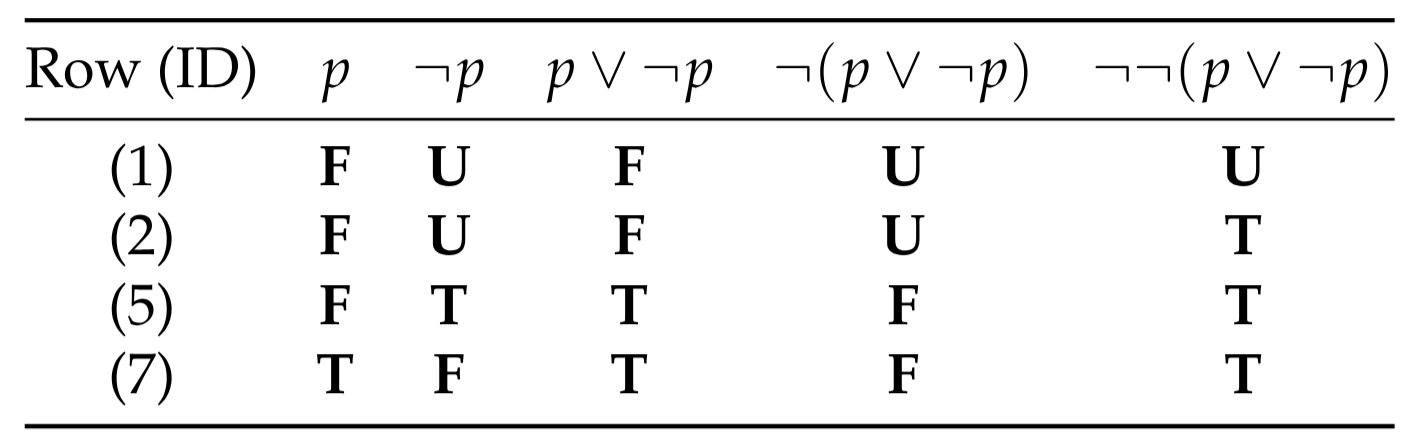}
\caption{Intermediate table.}\label{T3}
\end{subfigure}
\begin{subfigure}[b]{.45\linewidth}
\includegraphics[width=\linewidth]{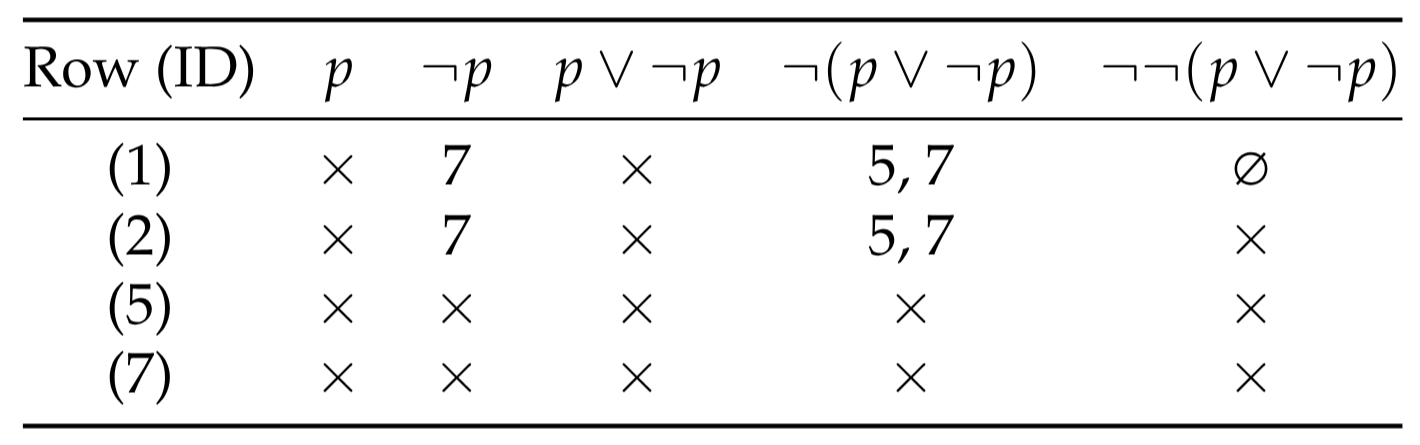}
\caption{Second cycle of dependencies.}\label{T4}
\end{subfigure}
\begin{subfigure}[b]{.45\linewidth}
\includegraphics[width=\linewidth]{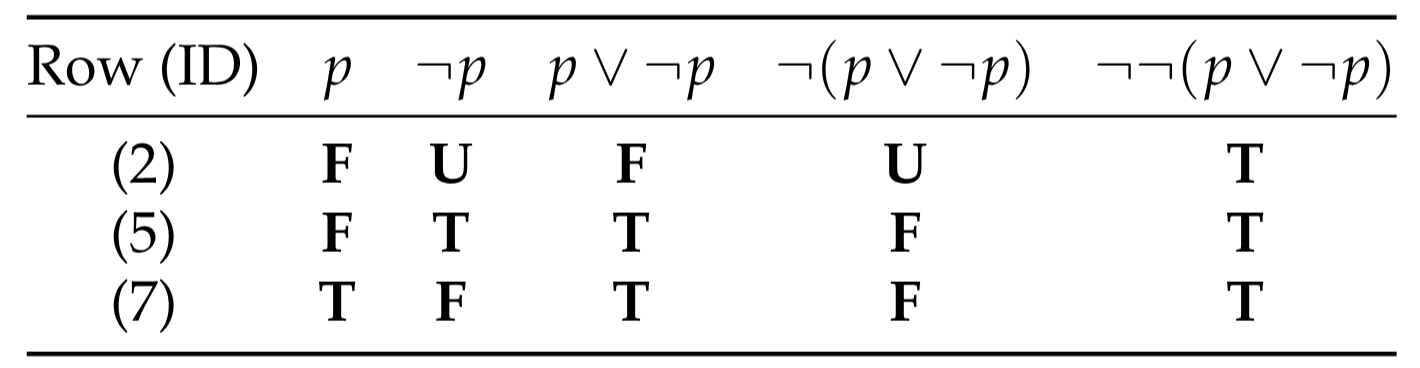}
\caption{Final table.}\label{T5}
\end{subfigure}
\begin{subfigure}[b]{.45\linewidth}
\includegraphics[width=\linewidth]{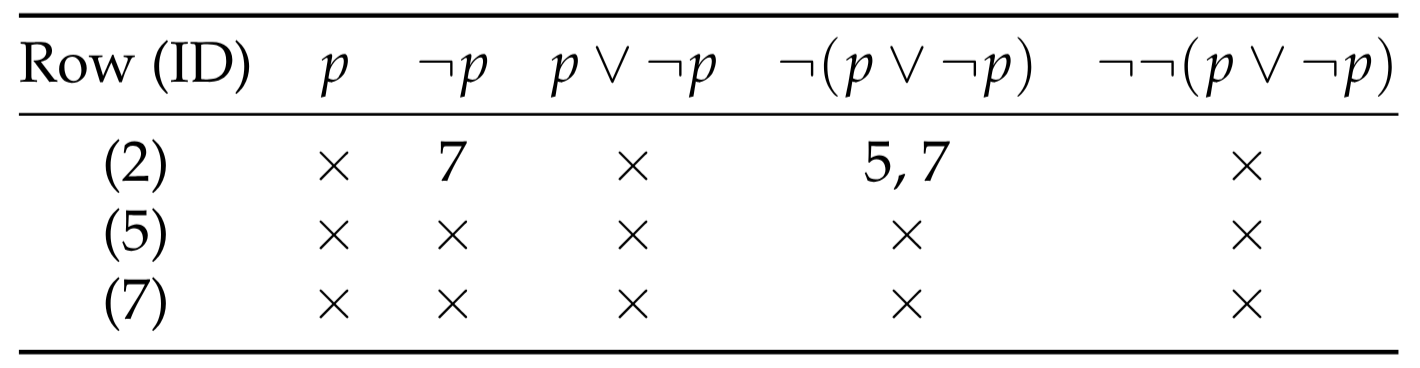}
\caption{Final dependency table.}\label{T6}
\end{subfigure}
\caption{Example of executions of refinement algorithm. Note that each partial valuation is identified by its ID, which is an index $n \in \mathbb{N}$.  The dependency tables are obtained in two steps: first, find every $v$ such that there is some $\alpha$ such that $v(\alpha) = \bU$. Then, for each of these partial valuations, find every valuation $w$ such that $w(\alpha) = \bF$ and, for every $\beta$, $w(\beta) = \bT$ whenever $v(\beta) = \bT$. }\label{example}
\end{figure}

In Figure~\ref{example}, the tables on the left correspond to an execution of Algorithm 1 (see Figure~\ref{fig:algorithm1}), while the tables on the right represent the dependency tables. Whenever $v(\alpha)=\bU$ occurs in a table on the left, the corresponding table on the right contains a set of supports (i.e., partial valuations $w$ such that $w(\alpha) = \bF$ and $w(\beta) = \bT$ whenever $v(\beta) = \bT$) in the correspondent cell. If no such support exists for that case, then the corresponding partial valuation is removed in the subsequent execution of Algorithm 1, which can be observed in the transition between successive tables on the left.

By removing rows $3$, $4$, and $6$ from Table~\ref{T1}, we obtain Table~\ref{T2}. In this new table, the row $1$, which was previously supported by the row $3$, no longer has that support available. Running the algorithm again with the updated table, we find that the row $1$ must now also be removed. This results in Table~\ref{T5}. At this point, Table~\ref{T6} contains no occurrences of $\emptyset$, which means that every instance of $\bU$ is now adequately supported. Hence, there are no more rows to remove and the algorithm is terminated. Finally, by inspecting the resulting Table~\ref{T5}, we can confirm that the formula is indeed valid.

Algorithm~1 is fundamentally the same as the one presented by Gr\"atz (\cite{gratz_truth_2022}, Algorithm $3.4$), but is applied to the new RNmatrix for \ipl. As a result, it remains within the same complexity class using exponential space. However, it has a lower upper bound because the atomic valuation considers only two truth-values: $\bF$ and $\bT$. Section~\ref{complexity} provides a further discussion on this topic.

\subsection{Soundness and completeness of the intuitionistic truth-tables}

Now, it will be shown that the generalized notion of truth-tables for \textbf{IPL} proposed in Subsection~\ref{sect:Int_truth_tables} is sound and complete. This will be established by first proving analyticity and co-analyticity with respect to level valuations in $\mathcal{M}_{IPL}$. To this end, and adapting~\cite[Definition~3.12]{gratz_truth_2022}, we introduce the intermediary notion of intuitionistic partial level valuation in $\mathcal{R}(\mathcal{M}_{IPL})$ over $\Lambda$.

\begin{definition}[Intuitionistic partial' level valuation over $\Lambda$]\label{def:partial'_level_val_IPL}
    Let $\Lambda \in CS(\Omega)$ and $v_p \in iPV(\Lambda)$. Then, $\tilde v_p$ is an {\em intuitionistic partial' level valuation} in $\mathcal{R}(\mathcal{M}_{IPL})$ over $\Lambda$  iff

    \begin{enumerate}
        \item[] $\forall \alpha \in \Lambda$ such that $\tilde v_p(\alpha) = \bU$, there exists a \emph{level valuation} $w$ in $\mathcal{L}_{IPL}$ such that $w(\alpha) = \bF$ and, $\forall \beta \in \Lambda$, $w(\beta) = \bT$ whenever $\tilde v_p(\beta) = \bT$.
    \end{enumerate}
The set of intuitionistic partial' level valuations in $\mathcal{R}(\mathcal{M}_{IPL})$ over $\Lambda$ will be denoted by $iPLV'(\Lambda)$.
\end{definition}

Definition~\ref{def:partial'_level_val_IPL} above is similar to Definition~\ref{def:partial_level_val_IPL}, except for two important aspects:  
(1) the domain $\Lambda$ of the partial valuation $\tilde v_p$ no longer needs to be finite, and  
(2) whenever there is a formula $\alpha \in \Lambda$ such that $\tilde v_p(\alpha) = \bU$, the support $w$ must be a level valuation.  

This difference is central for proving both co-analyticity and analyticity. Aspect (1) is crucial for analyticity (Lemma~\ref{analyticityIPL}), since we aim to prove that any domain closed under subformulas (possibly infinite) can be extended to a domain over the entire set of formulas. Aspect (2), in turn, is essential for co-analyticity (Lemma~\ref{co-analyticityIPL}), as it relies on the previous result obtained in Proposition~\ref{prop-level-val}, which establishes the existence of a certain level valuation under appropriate circumstances.

The notion of $iPLV'$ is well-suited for proof-theoretic purposes, but it is not sufficient for obtaining a decision procedure, since it depends on the set of level valuations. For this reason, we later prove that, for finite domains (which correspond to truth-tables), Definition~\ref{def:partial'_level_val_IPL} and Definition~\ref{def:partial_level_val_IPL} are equivalent (see Proposition~\ref{prop:analiticityIPL}). This result, a direct consequence of Proposition~\ref{algorithm-iPLV}, allows us to replace the notion of $iPLV'$ with its finitary counterpart $iPLV$, showing that, in order to decide the validity of a formula (or argument), it is sufficient to restrict the analysis to the finite universe of its subformulas (see Theorem~\ref{Theor:sound_compl_IPL0}).



\begin{remark} \label{rem:PLVsNN1}
Recall the predicate $\mathsf{P}_\Lambda(v,w,\alpha)$ introduced in Remark~\ref{rem:PLVsIPL}.  Clearly,\\[2mm]
$\begin{array}{lll}
iPLV'(\Lambda)&=& \{\tilde v_p \in iPV(\Lambda) \mid \forall \alpha \in \Lambda \big(\tilde v_p(\alpha) = \bU \mbox{ implies that } \mathsf{P}_\Lambda(\tilde v_p,w,\alpha)\\[1mm] 
&&\hspace*{7mm}\mbox{ for some $w \in \mathcal{L}_{IPL}$}\big)\}.
\end{array}$
\end{remark}

\begin{lemma} [Co-analyticity lemma for $\mathcal{R}(\mathcal{M}_{IPL})$] \label{co-analyticityIPL} Let $v \in \mathcal{L}_{IPL}$ and  $\Lambda \in CS(\Omega)$. Then, the restriction $\tilde v_p:=v_{|\Lambda}$ of $v$ to the domain $\Lambda$ belongs to $iPLV'(\Lambda)$.
\end{lemma}
\begin{proof}
Let $v \in \mathcal{L}_{IPL}$ and $\tilde v_p:=v_{|\Lambda}$. Clearly, $\tilde v_p \in iPV(\Lambda)$. Let $\alpha \in \Lambda$ such that $\tilde v_p(\alpha) = \bU$. Hence, $v(\alpha)=\bU$. By Proposition~\ref{prop-level-val}, there exists $w \in \mathcal{L}_{IPL}$ such that $w(\alpha)=\bF$ and $w(\beta)=\bT$ for every $\beta$ such that $v(\beta)=\bT$. In particular, $\mathsf{P}_\Lambda(\tilde v_p,w,\alpha)$ holds. From this, $\tilde v_p \in iPLV'(\Lambda)$ by Remark~\ref{rem:PLVsNN1}.
\end{proof}



A key concern regarding truth-tables is whether a given partial valuation can always be extended to a full valuation, i.e., whether a domain closed under subformulas can be extended to a complete domain. This property, called \emph{analyticity}, is central for proving the completeness of the method (see Theorem~\ref{Theor:sound_compl_IPL0}). To establish analyticity, we start with a domain over a set closed under subformulas and conservatively add new formulas to this domain. At the limit, we obtain the function given by the union of all such extensions, which corresponds to a function defined over the set of all formulas. However, a formula can be evaluated only if all of its subformulas are already evaluated. To ensure that this condition always holds, we consider a suitable ordering on the set of formulas. The technique will be entirely analogous to the case of \textbf{S4}, recall Remark~\ref{enum:forS4}.

\begin{definition}
Given $n \geq 0$, the set of formulas over $\Omega$ with complexity less or equal than $n$ will be denoted by $For(\Omega)_n$. That is:  $For(\Omega)_n = \{\alpha \in For(\Omega) \mid \co(\alpha) \leq n\}$.
\end{definition}

\begin{remark} \label{enum:forIPL} Given the ordinal $\omega$, let $\omega^2:=\omega\cdot \omega$ be its ordinal multiplication.
Observe that it is possible to define an enumeration  $\alpha_1,\alpha_2,\ldots \alpha_m, \ldots$ (for $m \in \omega^2$)  of $For(\Omega)$ such that $\co(\alpha_i) \leq \co(\alpha_j)$ if $i \leq j$ and, for every $i$ such that $\co(\alpha_i) >0$, if $\beta$ is a strict subformula of $\alpha_i$ then $\beta=\alpha_j$ for some $j < i$. Such an enumeration of $For(\Omega)$ can be defined as follows: every formula have and index in $\omega^2$, which is a denumerable ordinal, such that all the formulas with complexity $0$ (which form a denumerable set) are placed first, with indexes in $I_0:=\omega$; after this, all the formulas with complexity $1$ (which form a denumerable set) are placed with an index in $I_1:=\omega\cdot 2\setminus \omega=\{\omega,\omega + 1, \ldots\}$; in general, the formulas with complexity $n$ (which form a denumerable set) have an index in $I_{n}:=\omega\cdot(n+1)\setminus\omega\cdot n=\{\omega\cdot n, \omega\cdot n+1, \ldots\}$.
\end{remark}

Based on the enumeration of $For(\Omega)$ given in the previous remark, consider the following sets:

\begin{definition}  \label{sets:enum:forIPL}
Let $\alpha_1,\alpha_2,\ldots \alpha_m, \ldots$ (for $m \in \omega^2$) be an enumeration of $For(\Omega)$ as in Remark~\ref{enum:forIPL}. For every $n,m \in \omega$ let  $\bar{\Lambda}_n:=\Lambda \cup For(\Omega)_n$;   $\bar{\Lambda}_n^0:=\bar{\Lambda}_n$; and $\bar{\Lambda}_n^{m+1}=\bar{\Lambda}_n^m \cup \{\alpha_{\omega\cdot (n+1) +m}\}$.
\end{definition}

\begin{remark} \label{enum2:forIPL} By definition, $\bar{\Lambda}_n^{m+1}=\bar{\Lambda}_n \cup \{\alpha_{\omega\cdot (n+1)}, \alpha_{\omega\cdot (n+1) +1}, \ldots, \alpha_{\omega\cdot (n+1) +m}\}$, and $\alpha_{\omega\cdot (n+1) +m}\in \Lambda$ if and only if $\bar{\Lambda}_n^{m+1}=\bar{\Lambda}_n^m$. Hence, $\bar{\Lambda}_n^{m+1}$ is obtained from $\bar{\Lambda}_n$ by adding the first $m+1$ formulas (of the given enumeration) with complexity $n+1$, and so $\bar{\Lambda}_n^{m}$ adds at most $m$ formulas to $\bar{\Lambda}_n$, for $m \in \omega$.
It is worth noting that $\bar{\Lambda}_n^m \in CS(\Omega)$ for every $n,m \in \omega$; in particular, $\bar{\Lambda}_n \in CS(\Omega)$ for every $n \in \omega$. Clearly $\bar{\Lambda}_n^m \subseteq \bar{\Lambda}_n^{m+1}$, $\bar{\Lambda}_{n+1} = \bigcup_{m \in \omega} \bar{\Lambda}_n^m$, and $For(\Omega)= \bigcup_{n \in \omega} \bar{\Lambda}_n$.
\end{remark}

\begin{lemma} [Analyticity lemma for $\mathcal{R}(\mathcal{M}_{IPL})$] \label{analyticityIPL} Let $\Lambda \in CS(\Omega)$ and  $\tilde v_p \in iPLV'(\Lambda)$. Then, there exists $v \in \mathcal{L}_{IPL}$ such that the restriction $v_{|\Lambda}$ of $v$ to the domain $\Lambda$ coincides with $\tilde v_p$.
\end{lemma}
\begin{proof} It is an adaptation of the proof of the Analyticity lemma for $\mathcal{R}(\mathcal{M}'_{S4})$ (recall Lemma~\ref{analyticityS4}) to $\mathcal{R}(\mathcal{M}_{IPL})$.

Suppose first that $\Lambda=For(\Omega)$.  By induction on $n \in \omega$, it will be shown that $\tilde v_p \in \mathcal{L}^{IPL}_n$ for every $n$. By Definition~\ref{def:partial_valuationIPL} of partial valuation in $\mathcal{M}_{IPL}$, $\tilde v_p \in Val^{nd}_+(\mathcal{M}_{IPL})= {\mathcal{L}}_0^{IPL}$. Suppose that $\tilde v_p \in \mathcal{L}^{IPL}_n$ for a given $n \geq 0$, and let $\alpha \in For(\Omega)$ such that $\tilde v_p(\alpha)=\bU$. By Definition~\ref{def:partial'_level_val_IPL} of intuitionistic partial' level valuation,  there exists a level valuation $w$ in $\mathcal{L}_{IPL}$ (and so, $w \in \mathcal{L}^{IPL}_n$) such that $\mathsf{P}_\Lambda(\tilde v_p,w,\alpha)$. Since  $\Lambda=For(\Omega)$, this means that  $\tilde v_p \in \mathcal{L}^{IPL}_{n+1}$. From this, $\tilde v_p \in \mathcal{L}_{IPL}$ and the result clearly holds by taking $v:=\tilde v_p$.

Suppose now that $\Lambda \neq \emptyset $ is a proper subset of $For(\Omega)$.
Consider  an enumeration of $For(\Omega)$ as in Remark~\ref{enum:forIPL}.
The valuation $v$ will be defined by induction on the complexity $n$ of the formulas. More precisely, with terminology as in Definition~\ref{sets:enum:forIPL} and from the observations in Remark~\ref{enum2:forIPL}, an extension $v_n^m$ of $\tilde v_p$ (i.e.,  $\tilde v_p \subseteq v_n^m$) will be defined for every $(n,w) \in \omega\times \omega$ such that $v_n^m \in iPLV'(\bar{\Lambda}_n^m)$ and $v_i^j \subseteq v_n^{m}$ if $(i,j) \leq (n,m)$, where: $(i,j) \leq (n,m)$ iff $i \leq n$ or $i=n$ and $j \leq m$ (and where $\omega \times \omega$ denotes, as usual, the Cartesian product of $\omega$ with itself). \\[1mm]
{\bf Base} $n=m=0$: Observe that $\bar{\Lambda}_0=\bar{\Lambda}_0^0=\Lambda \cup \mathcal{P}$.
Define $v_0^0(\alpha)= \tilde v_p(\alpha)$ if $\alpha \in \Lambda$, and $v_0^0(p)=\bF$ if $p$ is a propositional variable which does not belong to $\Lambda$. Clearly, $\tilde v_p \subseteq v_0^0$ and $v_0^0 \in iPLV'(\bar{\Lambda}_0)=iPLV'(\bar{\Lambda}_0^0)$. Let $v_0:=v_0^0$.\\[1mm]
{\bf Inductive step}: Assume that $v_i^j(\alpha)$ was defined for every $\alpha \in \bar{\Lambda}_i^j$ such that  $v_i^j(\alpha)= \tilde v_p(\alpha)$ if $\alpha \in \Lambda$, $v_i^j \in iPLV'(\bar{\Lambda}_i^j)$, $v_i^j \subseteq v_k^r$ if $(i,j) \leq (k,r) \leq (n,m)$, for given $n,m \geq 0$ (Induction Hypothesis, IH). Now it will be shown how to extend $v_n^m$ to a function $v_n^{m+1}$ with domain $\bar{\Lambda}_n^{m+1}=\bar{\Lambda}_n^m \cup \{\alpha_{\omega\cdot (n+1) +m}\}$. That is, it will be shown how to define a value for $\alpha_{\omega\cdot (n+1) +m}$ while preserving the values assigned by $v_n^m$, in such a manner that $v_n^{m+1}\in iPLV'(\bar{\Lambda}_n^{m+1})$. To start with, define $v_n^{m+1}(\beta)=v_n^m(\beta)$ for every $\beta \in \bar{\Lambda}_n^m$.

Let $\alpha=\alpha_{\omega\cdot (n+1) +m}$. If $\alpha \in \Lambda$ then $\bar{\Lambda}_n^{m+1}=\bar{\Lambda}_n^m$. Then  $v_n^{m+1}:=v_n^m$, and the required hypothesis are obviously satisfied by $v_n^{m+1}$. Suppose now that  $\alpha \not\in \Lambda$. Observe that $v_n^{m+1}(\beta)=v_n^{m}(\beta)$ was already defined, for every strict subformula $\beta$ of $\alpha$. 
Since $v_n^m \in iPLV'(\bar{\Lambda}_n^m)$ then, for every $\delta \in \bar{\Lambda}_n^m$ such that $v_n^m(\delta)=\bU$ there exists $w   \in \mathcal{L}_{IPL}$ such that $\mathsf{P}_{\bar{\Lambda}_n^m}(v_n^m,w,\delta)$. Such a $w$ will be denoted by $w^\delta$ (observe that it is possible to have more than one $w^\delta$ for each $\delta$, and it is possible to have $w^\delta=w^\gamma$ for $\delta \neq\gamma$). Recall that $\bar{\Lambda}_n^m$ is closed under subformulas.
There are three cases to analyze:\\[1mm]
(1) $\alpha=\neg\beta$ such that $\neg^{IPL} (v_n^m(\beta))=\{\bU,\bT\}$, or $\alpha=\beta\to\gamma$ such that $\to^{IPL} (v_n^m(\beta),v_n^m(\gamma))=\{\bU,\bT\}$. There are three subcases to analyze:\\[1mm]
(1.1) There is $\delta \in \bar{\Lambda}_n^m$ with $v_n^m(\delta)=\bU$ and there exists some $w^\delta$ such that $w^\delta(\alpha)=\bF$. Define $v_n^{m+1}(\alpha)=\bU$. Hence $v_n^{m+1} \in iPLV'(\bar{\Lambda}_n^{m+1})$ such that $w^\alpha=w^\delta$.\\[1mm]
(1.2) There is $\delta \in \bar{\Lambda}_n^m $ with $v_n^m(\delta)=\bU$ and there exists some $w^\delta$ such that $w^\delta(\alpha)=\bU$. By Proposition~\ref{prop-level-val}, there exists $w'' \in  \mathcal{L}_{IPL}$ such that $\mathsf{P}_{For(\Omega)}(w^\delta,w'',\alpha)$.  Define $v_n^{m+1}(\alpha)=\bU$. Hence, $v_n^{m+1} \in iPLV'(\bar{\Lambda}_n^{m+1})$ such that $w^\alpha=w''$.\\[1mm]
(1.3) For every $\delta \in \bar{\Lambda}_n^m$ such that $v_n^m(\delta)=\bU$, $w^\delta(\alpha)=\bT$ for every $w^\delta$. Define $v_n^{m+1}(\alpha)=\bT$. Clearly,  $v_n^{m+1} \in iPLV'(\bar{\Lambda}_n^{m+1})$.\\[1mm]
(2) $\alpha=\neg\beta$ such that $\neg^{IPL} (v_n^m(\beta))=\{\bF\}$, or $\alpha=\beta\#\gamma$ such that $\#^{IPL} (v_n^m(\beta),v_n^m(\gamma))  =\{\bF\}$ (for some $\# \in \{\vee,\land, \to\}$).  Define $v_n^{m+1}(\alpha)=\bF$. It is clear that $v_n^{m+1} \in iPLV'(\bar{\Lambda}_n^{m+1})$.\\[1mm]
(3) $\alpha=\beta\#\gamma$ such that $\#^{IPL} (v_n^m(\beta),v_n^m(\gamma))  =\{\bT\}$ (for some $\# \in \{\vee,\land, \to\}$). Define $v_n^{m+1}(\alpha)=\bT$. It is clear that $v_n^{m+1} \in iPLV'(\bar{\Lambda}_n^{m+1})$.\\[1mm]

We have shown in cases (1)-(3) how to extend the domain of $v_n^m$ to the additional formula $\alpha_{\omega\cdot (n+1) +m} \in \bar{\Lambda}_n^{m+1}$, showing that the resulting function $v_n^{m+1}$ is in $iPLV'(\bar{\Lambda}_n^{m+1})$, and $\tilde v_p \subseteq v_i^j \subseteq v_k^r$, for every  $(i,j) \leq (k,r) \leq (n,m+1)$. That is, we show how to increase the superscript $m$. In order to increase the subscript $n$, let $v_{n+1}=v_{n+1}^0 := \bigcup_{m \in \omega} v_n^m$. By the procedure described above, it is immediate to see that $v_{n+1}^0$ is in $iPLV'(\bigcup_{m \in \omega} \bar{\Lambda}_n^m) =iPLV'(\bar{\Lambda}_{n+1})= iPLV'(\bar{\Lambda}_{n+1}^0)$, and $\tilde v_p \subseteq v_i^j \subseteq v_k^r$, for every  $(i,j) \leq (k,r) \leq (n+1,0)$. By repeating the process, it is possible to define $v_n^m$ for every $(n,m) \in \omega\times \omega$ with the required properties. Finally, let $v:=\bigcup_{n\in \omega} v_n$. From the manner in which the construction was made, it follows that  $v$ is in $iPLV'(\bigcup_{n \in \omega}\bar{\Lambda}_{n})= iPLV'(For(\Omega))$, and $\tilde v_p \subseteq v$. By the first part of the proof, we conclude that $v \in \mathcal{L}_{IPL}$.
\end{proof}


\begin{proposition} \label{prop:analiticityIPL}
Let $\Lambda \in FCS(\Omega)$. Then $iPLV(\Lambda)=iPLV'(\Lambda)$.
\end{proposition}
\begin{proof}
Let  $\tilde v_p \in iPLV'(\Lambda)$. Then,  $\tilde v_p \in iPV(\Lambda)$. Suppose that $\alpha \in \Lambda$ such that $\tilde v_p(\alpha)=\bU$.   Then, there exists  $w \in \mathcal{L}_{IPL}$ such that $\mathsf{P}_\Lambda(\tilde v_p,w,\alpha)$. Let $\tilde w_p=w_{|\Lambda}$. By Lemma~\ref{co-analyticityIPL}, $\tilde w_p \in iPLV'(\Lambda)$ such that $\mathsf{P}_\Lambda(\tilde v_p,\tilde w_p,\alpha)$. Now, suppose that there exists $\tilde w'_p\in iPLV'(\Lambda)$ such that $\mathsf{P}_\Lambda(\tilde v_p,\tilde w'_p,\alpha)$. By Lemma~\ref{analyticityIPL}, there exists $v' \in \mathcal{L}_{IPL}$ such that the restriction of $v'$ to $\Lambda$ coincides with $\tilde w'_p $. Hence,  $\mathsf{P}_\Lambda(\tilde v_p,v',\alpha)$. This shows that \\[2mm]
$\begin{array}{lll}
iPLV'(\Lambda)&=& \{\tilde v_p \in iPV(\Lambda) \mid \forall \alpha \in \Lambda \big(\tilde v_p(\alpha) = \bU \mbox{ implies that } \mathsf{P}_\Lambda(\tilde v_p,\tilde w_p,\alpha)\\[1mm] 
&&\hspace*{7mm}\mbox{ for some $\tilde w_p \in iPLV'(\Lambda)$}\big)\}.
\end{array}$
From this, it follows that  $iPLV(\Lambda)=iPLV'(\Lambda)$, by Proposition~\ref{algorithm-iPLV}.
\end{proof}

\begin{remark}
    In Proposition~\ref{algorithm-iPLV} the set $iPLV(\Lambda)$ was characterized constructively, as the output of an algorithm (Algorithm~1). In this subsection, we define an alternative subset $iPLV'(\Lambda) \subseteq iPV(\Lambda)$ in a predicative way. This subset connects the RNmatrix semantics for \ipl\ (based on level valuations) to the truth-table procedure given by $iPLV(\Lambda)$. Analyticity and co-analyticity of $iPLV'(\Lambda)$ for level valuations allow us to show that, for a finite $\Lambda$, $iPLV'(\Lambda)$ satisfies the condition of Remark~\ref{rem:PLVsIPL}. By uniqueness (Proposition~\ref{algorithm-iPLV}), $iPLV'(\Lambda)$ must coincide with  $iPLV(\Lambda)$, i.e., with the truth-table generated by $\Lambda$. But then we obtain analyticity and co-analyticity of $iPLV(\Lambda)$ for level valuations, which implies the soundness and completeness of truth-tables for \ipl, see Theorem~\ref{Theor:sound_compl_IPL0}.
\end{remark}

\begin{definition} \label{def_conseq_truth_tablesIPL}
Let $\Gamma  \cup \{\varphi\}$ be  a finite subset of  $For(\Omega)$. Let $\Lambda$ be the (finite) set of subformulas of $\Gamma \cup \{\varphi\}$.  We say that $\varphi$ is a consequence of $\Gamma$ w.r.t.  intuitionistic truth-tables, denoted by $\Gamma \models_{iPLV} \varphi$, if, for every intuitionistic partial level valuation $\tilde v_p \in iPLV(\Lambda)$, if $\tilde v_p(\beta)=\bT$ for every $\beta \in \Gamma$ then $\tilde v_p(\varphi)=\bT$. In particular, $\emptyset \models_{iPLV} \varphi$, denoted simply by $ \models_{iPLV} \varphi$, if $\tilde v_p(\varphi)=\bT$ for every $\tilde v_p \in iPLV(\Lambda)$, where $\Lambda$ is the set of subformulas of $\varphi$.
\end{definition}

Finally, we establish the correspondence between the restricted Nmatrix  for \ipl\ and the intuitionistic truth-tables:

\begin{theorem} [Soundness and completeness of intuitionistic truth-tables w.r.t.  $\mathcal{R}(\mathcal{M}_{IPL})$] \label{Theor:sound_compl_IPL0} 
Let $\Gamma  \cup \{\varphi\}$ be  a finite subset of  $For(\Omega)$. Then: $\Gamma \models_{\mathcal{R}(\mathcal{M}_{IPL})} \varphi$ if, and only if, $\Gamma \models_{iPLV} \varphi$.
\end{theorem}
\begin{proof} Let $\Lambda$ be the set of subformulas of $\Gamma  \cup \{\varphi\}$. Observe that $\Lambda \in FCS(\Omega)$.
Suppose that $\Gamma\models_{\mathcal{R}(\mathcal{M}_{IPL})} \varphi$. 
Let $\tilde v_p \in iPLV(\Lambda)$ such that $\tilde v_p(\beta)=\bT$ for every $\beta \in \Gamma$. By Proposition~\ref{prop:analiticityIPL}, $\tilde v_p \in iPLV'(\Lambda)$. By Lemma~\ref{analyticityIPL},  there exists $v \in \mathcal{L}_{IPL}$ such that the restriction $v_{|\Lambda}$ of $v$ to $\Lambda$ coincides with $\tilde v_p$. Then, $v(\beta)=\bT$ for every $\beta \in \Gamma$ and so, by hypothesis, $v(\varphi)=\bT$. That is, $\tilde v_p(\varphi)=\bT$. This shows that $\Gamma \models_{iPLV} \varphi$.

Conversely, suppose that  $\Gamma \models_{iPLV} \varphi$. Let $v \in \mathcal{L}_{IPL}$ such that $v(\beta)=\bT$ for every $\beta \in \Gamma$. By Lemma~\ref{co-analyticityIPL}, the restriction $\tilde v_p:=v_{|\Lambda}$ of $v$ to $\Lambda$ belongs to $iPLV'(\Lambda)$. By Proposition~\ref{prop:analiticityIPL},  $\tilde v_p \in iPLV(\Lambda)$ such that  $\tilde v_p(\beta)=\bT$ for every $\beta \in \Gamma$. By hypothesis, $\tilde v_p(\varphi)=\bT$. Then, $v(\varphi)=\bT$ and so $\Gamma\models_{\mathcal{R}(\mathcal{M}_{IPL})} \varphi$.
\end{proof}

\begin{corollary} [Deduction metatheorem for intuitionistic truth-tables] \label{MTD_iPLV}  Let $\Gamma \cup\{\alpha,\beta\}$ be  a finite subset of  $For(\Omega)$. Then, $\Gamma, \alpha \models_{iPLV} \beta$  if, and only if, $\Gamma \models_{iPLV} \alpha \to \beta$. 
\end{corollary}
\begin{proof} It follows from  Theorem~\ref{Theor:sound_compl_IPL0} and Corollary~\ref{MTD_RN_IPL}.
\end{proof}

It is worth noting that the latter result is useful to reduce the complexity of the truth-tables: instead of generating the (reduced) truth-table for a formula $\delta=\beta_1 \to (\beta_2 \to(\ldots \to(\beta_n \to \varphi) \ldots))$ to be tested as a tautology, it suffices generating the (reduced) truth-table for $\{\beta_1,\ldots,\beta_n, \varphi\}$ and then determine if, in every row such that $\beta_1,\ldots,\beta_n$ receive simultaneously the value $\bT$, $\varphi$ also receives the value $\bT$. This reduces the number of rows required to test $\delta$.

Finally, we arrive to the main result of the paper, recalling that {\bf LI} is a standard sequent calculus for \ipl:

\begin{theorem} [Soundness and completeness of  {\bf LI} w.r.t. intuitionistic truth-tables] \label{Theor:sound_compl_IPL1} Let $\Gamma \cup\{\varphi\}$ be  a finite subset of  $For(\Omega)$. Then: $\Gamma\vdash_{IPL} \varphi$ if, and only if, $\Gamma \models_{iPLV}\varphi$.
\end{theorem}
\begin{proof} It follows from  Theorems~\ref{thm:sound:IPL}, \ref{thm:compl:IPL}  and~\ref{Theor:sound_compl_IPL0}.
\end{proof}

The latter result shows that the procedure of constructing (non-deterministic) truth-tables in Definition~\ref{def:intui:truth-tables} determines a decision method for propositional intuitionistic logic \ipl.

\section{Some considerations on complexity}\label{complexity}

Unlike deterministic truth-tables, the size of non-deterministic truth-tables cannot be determined solely by the number of atoms. This is because, in principle, each subformula may (or may not) create new rows. Thus, we must also consider the number and shape of every evaluated subformula.

A workaround for this problem is to establish size limits using boundaries. More formally, we can define a general upper bound for non-deterministic matrices by identifying their branching factor as follows.

\begin{definition}[Branching factor]
 Let $\mathcal{M} = \langle \mathcal{V}, \mathcal{D}, \mathcal{O} \rangle$ be some Nmatrix over a signature $\Theta$. The branching factor of $\mathcal{M}$ is the lowest $\kappa$ such that, for every n-ary $\# \in \Theta_n$ and every $a_1, \ldots, a_n \in \mathcal{V}$, $| \mathcal{O}(\#) (a_1, \ldots, a_n) | \leq \kappa$.
  \end{definition}

\begin{definition}[Upper bound]\label{def:ub}
 Let $sub(\alpha)$ be the set of all subformulas of $\alpha$, $atoms(\alpha)$ the set of subformulas with complexity equal to $0$, $\kappa \geq 1$ the branching factor, and $\theta \leq | \mathcal{V} |$ the number of truth-values used for atomic valuations. Then, the maximum number of rows of the initial matrix for $\alpha$ in $\mathcal{R}(\mathcal{M})$ is given by

    \begin{align*}
 ub_{\mathcal{M}} (\alpha) = \frac{\kappa^{|sub(\alpha)|}}{\kappa^{|atoms(\alpha)|}} \times \theta^{|atoms(\alpha)|}
      \end{align*}
\end{definition}

\begin{example}
In $\mathcal{R}(\mathcal{M}_{IPL})$, $\kappa = 2$ and $\theta = 2$ . Hence,
\[
 ub_{\mathcal{M}_{IPL}} = 2^{|sub(\alpha)|}
\]
\end{example}

\emph{A fortiori}, we can infer that the upper bound for the Nmatrix is also an upper bound for the restricted Nmatrix. Therefore, we can expect no more than $2^{|sub(\alpha)|}$ rows in a truth-table created for $\alpha$ with $\mathcal{R}(\mathcal{M}_{IPL})$. The truth-table may be shorter than $2^{|sub(\alpha)|}$, but not less than $2^{|atoms(\alpha)|}$, given that there are at least $2^{|atoms(\alpha)|}$ (trivially) valid partial valuations in the truth-table.

This exponential growth in the number of subformulas may be problematic; a formula with only one variable can be arbitrarily large and produce immense tables. However, truth-tables are meaningful. By constructing a truth-table, one generates significant information about a given proposition or argument. A future work direction would be improving the average-case runtime by refining the algorithm with state-of-the-art strategies.

\section{Related works}\label{Sect:Related}

The idea of using finite non-deterministic semantics for characterizing intuitionistic logic and modal logic is not new. This section will mention some approaches to this topic proposed in the literature. We do not intend to present a detailed description and comparison between all these approaches here since they are beyond the scope of the present paper.\\[1mm]
{\bf(1)}  Andrea Loparic already obtained in 1977, announced in~\cite{loparic:77} and finally published in~\cite{loparic2010}, a semantical characterization of  Johannson-Kolmogorov minimal logic and intuitionistic propositional logic \ipl\ in terms of a $2$-valued, non-truth-functional valuation semantics which induces a decision procedure for these logics. Similar to Gr\"atz's method (and to the adaptation to  \ipl\ proposed here), Loparic's approach is based on a criterion for distinguishing between what she called {\em semivaluations} and {\em valuations}, the latter being the useful ones. This would correspond to level valuations, since such valuations are associated with $\varphi$-saturated sets, and they are the ones that characterize the logic. The associated algorithm generates a non-deterministic truth-table for a given formula, as in the cases of \sfo\ and \ipl\ presented in this paper. 
Then, using formal criteria based on implicative formulas, it deletes certain rows. Thus, a formula is a tautology of the given logic if and only if it receives the value 1 in all remaining rows.\\[1mm] 
{\bf(2)}  J. Kearns proposed in~\cite{kearns:78} a semantics for (first-order) intuitionistic logic based on the notion of {\em justifications}.
A 3-valued Nmatrix, whose domain is not formed by truth-values but by values of justification ($+$ for `justified', $o$ for `weakly unjustified', and $-$ for `strongly unjustifiable'), is considered for interpreting the connectives of \ipl. A sentence is weakly unjustified if it is not justified by the currently justified sentences but might become justified later. In turn, strongly unjustifiable sentences are those already ruled out by the currently justified sentences.
Given two valuations $v$ and $v'$ over the Nmatrix, $v \preceq v'$ ($v$ {\em is contained in} $v'$) if $v'$ preserves the values $+$ (and so, the values $-$) assigned by $v$ to each formula. Hence, if $v'(\alpha)=o$ then $v(\alpha)=o$, for any $\alpha$. The intuitive idea is that $v'$ is a `possible future' of $v$ (hence $v'$ has more `knowledge'  than $v$, provided that the weak epistemic status $o$ may be corrected to the stronger  $+$ or $-$ in the future).
Kearns constructed a hierarchy of level valuations (as he would later do for modal logics) starting with level 0, which is formed by all valuations on the Nmatrix.
The $(n+1)$th-level valuations are the $n$th-level valuations satisfying certain restrictions defined in terms of the sets $X_n^v=\{v' \mid v'$ is an $n$th-level valuation such that $v \preceq v' \}$, for any $n$th-level valuation $v$. A basic feature is that, if $v$ is an $(n + 1)$th-level valuation and $v'(\alpha) \neq +$ for every $v' \in X_n^v$ then $v(\alpha)=-$. That is, if $\alpha$ is not justified in any possible future of $v$, then $\alpha$ must be strongly unjustifiable at present $v$.
In a second paper~\cite{kearns:IPL:81} on the subject, published two years after~\cite{kearns:78} and seven months after~\cite{kearns_modal_1981}, Kearns improved the interpretation of the justification semantics and simplified its formal presentation by using trees instead of levels.
The {\em contained} relation $\preceq$ is crucial to define the trees whose nodes are valuations. It should be stressed that Kearns's proposal constitutes a 3-valued RNmatrix semantics for \ipl, although it differs significantly from our approach. In particular, Kearns's RNmatrix does not seem to induce a decision procedure for \ipl.\\[1mm]
{\bf(3)} Concerning modal logics, in~\cite{lah:zoh:22}, right after the publication of~\cite{gratz_truth_2022}, O. Lahav and Y. Zohar proposed a decision procedure for systems {\bf K} and {\bf KT} based on 4- and 3-valued RNmatrices, respectively, by using a decision procedure different (but closely related) to Gr\"atz's approach. Indeed, for ${\bf L} \in \{ {\bf K}, {\bf KT}\}$ they considered a sequent calculus $\mathcal{G}_{\bf L}$ and sets $\mathbb{V}_{\bf L}^{\mathcal{F},n}$ of $n$th-level valuations relative to a set $\mathcal{F}$ of formulas (the domain of such valuations) and $n$ is the maximum number of applications of the necessitation rule validated by such valuations. Given a sequent $\Gamma \Rightarrow \varphi$, the expression $\vdash_{\mathcal{G}_{\bf L}}^{\mathcal{F},n}\Gamma  \Rightarrow \varphi$ denotes that there exists a derivation in $\mathcal{G}_{\bf L}$ of  $\Gamma \Rightarrow \varphi$ with a derivation containing exclusively formulas in $\mathcal{F}$ and such that the number of applications of the inference rule of necessitation $\Delta \Rightarrow \psi / \Box\Delta \Rightarrow \Box\psi$ in any of the branches of the derivation is limited by $n$. They proved that  $\vdash_{\mathcal{G}_{\bf L}}^{\mathcal{F},n}\Gamma  \Rightarrow \varphi$ if and only if $\varphi$ is a consequence of $\Gamma$ w.r.t. $n$th-level valuations in $\mathbb{V}_{\bf L}^{\mathcal{F},n}$. Moreover, they obtained a decision procedure for {\bf K} by considering, for each sequent $\Gamma \Rightarrow \varphi$ to be tested, the set of level valuations  $\mathbb{V}_{\bf L}^{\mathcal{F},n}$, where  $\mathcal{F}$ is the set of subformulas of $\Gamma \cup \{\varphi\}$ and $n=4^{|\mathcal{F}|}$. For {\bf KT}, it is enough considering $n=3^{|\mathcal{F}|}$, since the RNmatrix for {\bf KT} is 3-valued. Moreover, that RNmatrix coincides with the one proposed in~\cite{gratz_truth_2022} for {\bf KT}. It is worth noting that~\cite{lah:zoh:22} considers the reduced version of the underlying Nmatrix, expanded to the full signature. This produces the Nmatrix  $\mathcal{M}'_{S4}$ for \sfo\ of Table~\ref{table:nm_S4*-R}, but where $\Box'(2)=\{1,2\}$. 
Despite these similarities, the decision procedure in~\cite{lah:zoh:22} (as mentioned earlier) differs from that in~\cite{gratz_truth_2022}. It consists of generating the set of $n$th-level valuations over $\mathcal{F}$ for appropriate $n$ and $\mathcal{F}$.\\[1mm]
{\bf(4)} More recently, A. Solares-Rojas proposed a 3-valued Nmatrix semantics for \ipl\ based on intuitionistic Kripke models and the novel notion of {\em depth-bounded approximations} (see~\cite[Chapter~5]{Solares-Rojas}). Within this semantic approach to \ipl, a 3-valued Nmatrix \textemdash\ whose truth-values are intuitionistic truth (1), falsity (0), and indeterminacy ($u$) \textemdash\  gives the meaning of the connectives in terms of information that a given agent holds. 
The valuations over the Nmatrix aim to model actual information states, ordered by a certain {\em refinement relation} $\sqsubseteq_3$.
This relation models how the agent's knowledge evolves when new information from reliable external sources is acquired. 
The semantics for \ipl\ is defined via intuitionistic Kripke frames expanded with a family of valuations indexed by the worlds, where the refinement order corresponds to the frame's accessibility relation. Some similarities between this approach and Kearns's (mentioned in point {\bf (2)}) may be observed.
Conceptually, both approaches aim to model the evolution of states of knowledge: the information state of an agent in~\cite{Solares-Rojas}, and the knowledge of an idealized community in~\cite{kearns:78,kearns:IPL:81}. Regarding technical aspects, by identifying $1$, $0$, and $u$ with $+$, $-$, and $o$ respectively, the Nmatrix in~\cite[Chapter~5]{Solares-Rojas} coincides with that in~\cite{kearns:78}, with one exception: while $o \vee o = \{o, +\}$ in Kearns's work, $u \vee u = \{u\}$ (i.e., $o \vee o = \{o\}$) in Solares-Rojas's approach. Additionally, the refinement relation $\sqsubseteq_3$ between valuations coincides with the {\em contained} relation $\preceq$ introduced in~\cite{kearns:78}.
The relationship between both approaches (in particular, the connections with the improved version~\cite{kearns:IPL:81}, which is closer to Kripke and Beth's semantics) deserves to be investigated. Regarding the connections with our approach, although the 3-valued Nmatrices proposed in~\cite[Chapter~5]{Solares-Rojas} and the one we gave in Definition~\ref{def:M_IPL} are quite different, it is not immediate whether the semantics proposed by Solares-Rojas could be seen as given by an RNmatrix. This is the subject of future investigation, as well as the relationship between our approach and the semantics of justification introduced by Kearns.

\section{Concluding remarks} \label{final_remarks}

In this paper, we proposed a new decision procedure for intuitionistic propositional logic (\ipl) based on a 3-valued restricted Nmatrix semantics. This procedure was obtained by abstracting the composition of Gr\"atz's decision procedure for \sfo~\cite{gratz_truth_2022} and G\"odel's faithful translation from \ipl\ to \sfo. A 3-valued Nmatrix with an intuitive interpretation in epistemic terms was proposed for \ipl. By adapting Gr\"atz's algorithm for \sfo, a suitable notion of level valuations was obtained, which allowed us to define a straightforward algorithm for deleting spurious rows in the non-deterministic truth-tables generated by the Nmatrix for any formula. The soundness and completeness proofs for the RNmatrix and the resulting truth-table algorithm for \ipl\ were presented in full technical detail.

It is worth noting that the Nmatrix for \ipl\ introduced in Definition~\ref{def:M_IPL} has a clear and intuitive interpretation in terms of knowledge (see comments after Definition~\ref{def:M_IPL}) and constructibility. Indeed, recall that the valuation associated with a $\varphi$-saturated set $\Delta$ in \ipl\ (given in Definition~\ref{def:val-Delta-IPL}) acts as a 3-valued characteristic function of $\Delta$, and note the observations in Remark~\ref{ren:canval}.

By considering finite and decidable RNmatrix semantics as a generalization of standard finite-valued deterministic truth-tables, the results presented here provide a way to overcome the limitations obtained by G\"odel and Dugundji for \ipl\ and \sfo, respectively, as mentioned in Section~\ref{sect:Intro}.

As a topic of future research, we intend to explore the meaning of the partial valuations corresponding to the rows of the non-deterministic truth-tables in \ipl\ and \sfo. In particular, we are studying the possibility of generating finite Kripke models for \ipl, {\bf KT} and \sfo\ from the obtained truth-tables. This could be connected with {\em finite approximability} (\cite[p. 49 and Ch.~11]{cha:Zakh:97}) or with {\em mini-canonical models} (\cite[Ch.~8]{hughes_cresswell_1996}). Finally, the extension of this methodology to other modal logics, as well as Prawitz's ecumenical systems (see~\cite{prawitz}), will also be investigated. Some steps in this direction were taken in~\cite{lem:ola:pim:con:25}.

\

\noindent{\large \bf Acknowledgments:}
The first and second authors thank the support of the S\~ao Paulo Research Foundation (FAPESP, Brazil) through the PhD scholarship grant \#21/01025-3 and the Thematic Project RatioLog \#20/16353-3. The second author also acknowledges support from the National Council for Scientific and Technological Development (CNPq, Brazil), through the individual research grant \#309830/2023-0. The third author thanks Funda\c c\~ao Carlos Chagas de Amparo \`a Pesquisa do Estado do Rio de Janeiro (FAPERJ, Brazil), through the research grant \#E-26.210.971/2019.

\bibliography{main}{}

\newpage

\section*{Appendix A}

The following truth-tables show that each axiom of the usual Hilbert calculus for \textbf{IPL} is valid in the intuitionistic truth-tables. These tables were automatically generated by a Rocq module for RNmatrices written by R. Leme, see\\

\url{https://github.com/renatoleme/Forest}

\

{
\scriptsize

\begin{multicols}{2}
    \setbox\ltmcbox\vbox\bgroup
    \makeatletter\col@number\@ne
\begin{longtable}[]{@{}cccc@{}}
\toprule
\(q\) & \(p\) & \(q \to p\) & \(p \to (q \to p)\)\tabularnewline
\midrule
\endhead
\(\textbf{F}\) & \(\textbf{F}\) & \(\textbf{U}\) &
\(\textbf{T}\)\tabularnewline
\(\textbf{F}\) & \(\textbf{F}\) & \(\textbf{T}\) &
\(\textbf{T}\)\tabularnewline
\(\textbf{F}\) & \(\textbf{T}\) & \(\textbf{T}\) &
\(\textbf{T}\)\tabularnewline
\(\textbf{T}\) & \(\textbf{F}\) & \(\textbf{F}\) &
\(\textbf{T}\)\tabularnewline
\(\textbf{T}\) & \(\textbf{T}\) & \(\textbf{T}\) &
\(\textbf{T}\)\tabularnewline
\bottomrule
\end{longtable}
\unskip
\unpenalty
\unpenalty\egroup
\unvbox\ltmcbox

\setbox\ltmcbox\vbox\bgroup
\makeatletter\col@number\@ne
\begin{longtable}[]{@{}ccccc@{}}
    \toprule
    \(p\) & \(q\) & \(p \land q\) & \(q \to p \land q\) &
    \(p \to (q \to p \land q)\)\tabularnewline
    \midrule
    \endhead
    \(\textbf{F}\) & \(\textbf{F}\) & \(\textbf{F}\) & \(\textbf{U}\) &
    \(\textbf{T}\)\tabularnewline
    \(\textbf{F}\) & \(\textbf{F}\) & \(\textbf{F}\) & \(\textbf{T}\) &
    \(\textbf{T}\)\tabularnewline
    \(\textbf{F}\) & \(\textbf{T}\) & \(\textbf{F}\) & \(\textbf{F}\) &
    \(\textbf{T}\)\tabularnewline
    \(\textbf{T}\) & \(\textbf{F}\) & \(\textbf{F}\) & \(\textbf{T}\) &
    \(\textbf{T}\)\tabularnewline
    \(\textbf{T}\) & \(\textbf{T}\) & \(\textbf{T}\) & \(\textbf{T}\) &
    \(\textbf{T}\)\tabularnewline
    \bottomrule
    \end{longtable}
    \unskip
    \unpenalty
    \unpenalty\egroup
    \unvbox\ltmcbox
\end{multicols}

\begin{longtable}[]{@{}ccccccccc@{}}
\toprule
\(q\) & \(p\) & \(r\) & \(q \to r\) & \(p \to q\)& \(p \to r\) & \(p \to (q \to r)\) & 
\((p \to q) \to (p \to r)\) & \((p \to (q \to r)) \to ((p \to q) \to (p \to r))\)
\tabularnewline
\midrule
\endhead
\(\textbf{F}\) & \(\textbf{F}\) & \(\textbf{F}\) & \(\textbf{U}\) & \(\textbf{U}\) & \(\textbf{U}\) & \(\textbf{U}\) & \(\textbf{U}\) & \(\textbf{T}\) \tabularnewline
\(\textbf{F}\) & \(\textbf{F}\) & \(\textbf{F}\) & \(\textbf{U}\) & \(\textbf{U}\) & \(\textbf{U}\) & \(\textbf{T}\) & \(\textbf{T}\) & \(\textbf{T}\) \tabularnewline
\(\textbf{F}\) & \(\textbf{F}\) & \(\textbf{F}\) & \(\textbf{U}\) & \(\textbf{U}\) & \(\textbf{T}\) & \(\textbf{T}\) & \(\textbf{T}\) & \(\textbf{T}\) \tabularnewline
\(\textbf{F}\) & \(\textbf{F}\) & \(\textbf{F}\) & \(\textbf{U}\) & \(\textbf{T}\) & \(\textbf{U}\) & \(\textbf{U}\) & \(\textbf{F}\) & \(\textbf{T}\) \tabularnewline
\(\textbf{F}\) & \(\textbf{F}\) & \(\textbf{F}\) & \(\textbf{U}\) & \(\textbf{T}\) & \(\textbf{T}\) & \(\textbf{T}\) & \(\textbf{T}\) & \(\textbf{T}\) \tabularnewline
\(\textbf{F}\) & \(\textbf{F}\) & \(\textbf{F}\) & \(\textbf{T}\) & \(\textbf{U}\) & \(\textbf{U}\) & \(\textbf{T}\) & \(\textbf{T}\) & \(\textbf{T}\) \tabularnewline
\(\textbf{F}\) & \(\textbf{F}\) & \(\textbf{F}\) & \(\textbf{T}\) & \(\textbf{U}\) & \(\textbf{T}\) & \(\textbf{T}\) & \(\textbf{T}\) & \(\textbf{T}\) \tabularnewline
\(\textbf{F}\) & \(\textbf{F}\) & \(\textbf{F}\) & \(\textbf{T}\) & \(\textbf{T}\) & \(\textbf{T}\) & \(\textbf{T}\) & \(\textbf{T}\) & \(\textbf{T}\) \tabularnewline
\(\textbf{F}\) & \(\textbf{F}\) & \(\textbf{T}\) & \(\textbf{T}\) & \(\textbf{U}\) & \(\textbf{T}\) & \(\textbf{T}\) & \(\textbf{T}\) & \(\textbf{T}\) \tabularnewline
\(\textbf{F}\) & \(\textbf{F}\) & \(\textbf{T}\) & \(\textbf{T}\) & \(\textbf{T}\) & \(\textbf{T}\) & \(\textbf{T}\) & \(\textbf{T}\) & \(\textbf{T}\) \tabularnewline
\(\textbf{F}\) & \(\textbf{T}\) & \(\textbf{F}\) & \(\textbf{U}\) & \(\textbf{F}\) & \(\textbf{F}\) & \(\textbf{F}\) & \(\textbf{U}\) & \(\textbf{T}\) \tabularnewline
\(\textbf{F}\) & \(\textbf{T}\) & \(\textbf{F}\) & \(\textbf{T}\) & \(\textbf{F}\) & \(\textbf{F}\) & \(\textbf{T}\) & \(\textbf{T}\) & \(\textbf{T}\) \tabularnewline
\(\textbf{F}\) & \(\textbf{T}\) & \(\textbf{T}\) & \(\textbf{T}\) & \(\textbf{F}\) & \(\textbf{T}\) & \(\textbf{T}\) & \(\textbf{T}\) & \(\textbf{T}\) \tabularnewline
\(\textbf{T}\) & \(\textbf{F}\) & \(\textbf{F}\) & \(\textbf{F}\) & \(\textbf{T}\) & \(\textbf{U}\) & \(\textbf{U}\) & \(\textbf{F}\) & \(\textbf{T}\) \tabularnewline
\(\textbf{T}\) & \(\textbf{F}\) & \(\textbf{F}\) & \(\textbf{F}\) & \(\textbf{T}\) & \(\textbf{T}\) & \(\textbf{T}\) & \(\textbf{T}\) & \(\textbf{T}\) \tabularnewline
\(\textbf{T}\) & \(\textbf{F}\) & \(\textbf{T}\) & \(\textbf{T}\) & \(\textbf{T}\) & \(\textbf{T}\) & \(\textbf{T}\) & \(\textbf{T}\) & \(\textbf{T}\) \tabularnewline
\(\textbf{T}\) & \(\textbf{T}\) & \(\textbf{F}\) & \(\textbf{F}\) & \(\textbf{T}\) & \(\textbf{F}\) & \(\textbf{F}\) & \(\textbf{F}\) & \(\textbf{T}\) \tabularnewline
\(\textbf{T}\) & \(\textbf{T}\) & \(\textbf{T}\) & \(\textbf{T}\) & \(\textbf{T}\) & \(\textbf{T}\) & \(\textbf{T}\) & \(\textbf{T}\) & \(\textbf{T}\) \tabularnewline
\bottomrule
\end{longtable}

\begin{multicols}{2}
    \setbox\ltmcbox\vbox\bgroup
    \makeatletter\col@number\@ne
\begin{longtable}[]{@{}ccccc@{}}
\toprule
\(q\) & \(p\) & \(p \land q\) & \(p \land q \to p\) & \(p \land q \to q\)\tabularnewline
\midrule
\endhead
\(\textbf{F}\) & \(\textbf{F}\) & \(\textbf{F}\) &
\(\textbf{T}\) & \(\textbf{T}\)\tabularnewline
\(\textbf{F}\) & \(\textbf{T}\) & \(\textbf{F}\) &
\(\textbf{T}\) & \(\textbf{T}\)\tabularnewline
\(\textbf{T}\) & \(\textbf{F}\) & \(\textbf{F}\) &
\(\textbf{T}\) & \(\textbf{T}\)\tabularnewline
\(\textbf{T}\) & \(\textbf{T}\) & \(\textbf{T}\) &
\(\textbf{T}\) & \(\textbf{T}\)\tabularnewline
\bottomrule
\end{longtable}
\unskip
\unpenalty
\unpenalty\egroup
\unvbox\ltmcbox
\setbox\ltmcbox\vbox\bgroup
\makeatletter\col@number\@ne
\begin{longtable}[]{@{}ccccc@{}}
\toprule
\(p\) & \(q\) & \(p \lor q\) & \(p \to p \lor q\) & \(q \to p \lor q\)\tabularnewline
\midrule
\endhead
\(\textbf{F}\) & \(\textbf{F}\) & \(\textbf{F}\) &
\(\textbf{T}\) & \(\textbf{T}\)\tabularnewline
\(\textbf{F}\) & \(\textbf{T}\) & \(\textbf{T}\) &
\(\textbf{T}\) & \(\textbf{T}\)\tabularnewline
\(\textbf{T}\) & \(\textbf{F}\) & \(\textbf{T}\) &
\(\textbf{T}\) & \(\textbf{T}\)\tabularnewline
\(\textbf{T}\) & \(\textbf{T}\) & \(\textbf{T}\) &
\(\textbf{T}\) & \(\textbf{T}\)\tabularnewline
\bottomrule
\end{longtable}
\unskip
\unpenalty
\unpenalty\egroup
\unvbox\ltmcbox
\end{multicols}

\newpage

\begin{longtable}[]{@{}ccccccccc@{}}
\toprule \(p\) & \(q\) & \(r\) & \(p \to r\) & \(q \to r\) & \(p \lor q\) & \(p \lor q \to r\) & \((q \to r) \to (p \lor q \to r)\) & \((p \to r) \to ((q \to r) \to	((p \lor q) \to r))\) \tabularnewline
\midrule
\endhead
 \(\textbf{F}\)  &  \(\textbf{F}\)  &  \(\textbf{F}\)  &  \(\textbf{U}\)  &  \(\textbf{U}\)  &  \(\textbf{F}\)  &  \(\textbf{U}\)  &  \(\textbf{U}\)  &  \(\textbf{T}\) \tabularnewline
 \(\textbf{F}\)  &  \(\textbf{F}\)  &  \(\textbf{F}\)  &  \(\textbf{U}\)  &  \(\textbf{U}\)  &  \(\textbf{F}\)  &  \(\textbf{U}\)  &  \(\textbf{T}\)  &  \(\textbf{T}\) \tabularnewline
 \(\textbf{F}\)  &  \(\textbf{F}\)  &  \(\textbf{F}\)  &  \(\textbf{U}\)  &  \(\textbf{T}\)  &  \(\textbf{F}\)  &  \(\textbf{U}\)  &  \(\textbf{F}\)  &  \(\textbf{T}\) \tabularnewline
 \(\textbf{F}\)  &  \(\textbf{F}\)  &  \(\textbf{F}\)  &  \(\textbf{T}\)  &  \(\textbf{U}\)  &  \(\textbf{F}\)  &  \(\textbf{U}\)  &  \(\textbf{T}\)  &  \(\textbf{T}\) \tabularnewline
 \(\textbf{F}\)  &  \(\textbf{F}\)  &  \(\textbf{F}\)  &  \(\textbf{T}\)  &  \(\textbf{T}\)  &  \(\textbf{F}\)  &  \(\textbf{T}\)  &  \(\textbf{T}\)  &  \(\textbf{T}\) \tabularnewline
 \(\textbf{F}\)  &  \(\textbf{F}\)  &  \(\textbf{T}\)  &  \(\textbf{T}\)  &  \(\textbf{T}\)  &  \(\textbf{F}\)  &  \(\textbf{T}\)  &  \(\textbf{T}\)  &  \(\textbf{T}\) \tabularnewline
 \(\textbf{F}\)  &  \(\textbf{T}\)  &  \(\textbf{F}\)  &  \(\textbf{U}\)  &  \(\textbf{F}\)  &  \(\textbf{T}\)  &  \(\textbf{F}\)  &  \(\textbf{T}\)  &  \(\textbf{T}\) \tabularnewline
 \(\textbf{F}\)  &  \(\textbf{T}\)  &  \(\textbf{F}\)  &  \(\textbf{T}\)  &  \(\textbf{F}\)  &  \(\textbf{T}\)  &  \(\textbf{F}\)  &  \(\textbf{T}\)  &  \(\textbf{T}\) \tabularnewline
 \(\textbf{F}\)  &  \(\textbf{T}\)  &  \(\textbf{T}\)  &  \(\textbf{T}\)  &  \(\textbf{T}\)  &  \(\textbf{T}\)  &  \(\textbf{T}\)  &  \(\textbf{T}\)  &  \(\textbf{T}\) \tabularnewline
 \(\textbf{T}\)  &  \(\textbf{F}\)  &  \(\textbf{F}\)  &  \(\textbf{F}\)  &  \(\textbf{U}\)  &  \(\textbf{T}\)  &  \(\textbf{F}\)  &  \(\textbf{U}\)  &  \(\textbf{T}\) \tabularnewline
 \(\textbf{T}\)  &  \(\textbf{F}\)  &  \(\textbf{F}\)  &  \(\textbf{F}\)  &  \(\textbf{U}\)  &  \(\textbf{T}\)  &  \(\textbf{F}\)  &  \(\textbf{T}\)  &  \(\textbf{T}\) \tabularnewline
 \(\textbf{T}\)  &  \(\textbf{F}\)  &  \(\textbf{F}\)  &  \(\textbf{F}\)  &  \(\textbf{T}\)  &  \(\textbf{T}\)  &  \(\textbf{F}\)  &  \(\textbf{F}\)  &  \(\textbf{T}\) \tabularnewline
 \(\textbf{T}\)  &  \(\textbf{F}\)  &  \(\textbf{T}\)  &  \(\textbf{T}\)  &  \(\textbf{T}\)  &  \(\textbf{T}\)  &  \(\textbf{T}\)  &  \(\textbf{T}\)  &  \(\textbf{T}\) \tabularnewline
 \(\textbf{T}\)  &  \(\textbf{T}\)  &  \(\textbf{F}\)  &  \(\textbf{F}\)  &  \(\textbf{F}\)  &  \(\textbf{T}\)  &  \(\textbf{F}\)  &  \(\textbf{T}\)  &  \(\textbf{T}\) \tabularnewline
 \(\textbf{T}\)  &  \(\textbf{T}\)  &  \(\textbf{T}\)  &  \(\textbf{T}\)  &  \(\textbf{T}\)  &  \(\textbf{T}\)  &  \(\textbf{T}\)  &  \(\textbf{T}\)  &  \(\textbf{T}\) \tabularnewline
\bottomrule
\end{longtable}

\begin{longtable}[]{@{}cccccccc@{}}
\toprule
\(p\)& \(q\)& \(q \to p\)& \(\neg p\)& \(\neg q\)& \(q \to \neg p\)& \((q \to \neg p) \to \neg q\)& \((q \to p) \to ((q \to \neg p) \to \neg q)\)\tabularnewline
\midrule
\endhead
 \(\textbf{F}\)  &  \(\textbf{F}\)  &  \(\textbf{U}\)  &  \(\textbf{U}\)  &  \(\textbf{U}\)  &  \(\textbf{U}\)  &  \(\textbf{U}\)  &  \(\textbf{T}\) \tabularnewline
 \(\textbf{F}\)  &  \(\textbf{F}\)  &  \(\textbf{U}\)  &  \(\textbf{U}\)  &  \(\textbf{U}\)  &  \(\textbf{U}\)  &  \(\textbf{T}\)  &  \(\textbf{T}\) \tabularnewline
 \(\textbf{F}\)  &  \(\textbf{F}\)  &  \(\textbf{U}\)  &  \(\textbf{U}\)  &  \(\textbf{U}\)  &  \(\textbf{T}\)  &  \(\textbf{F}\)  &  \(\textbf{T}\) \tabularnewline
 \(\textbf{F}\)  &  \(\textbf{F}\)  &  \(\textbf{U}\)  &  \(\textbf{T}\)  &  \(\textbf{U}\)  &  \(\textbf{T}\)  &  \(\textbf{F}\)  &  \(\textbf{T}\) \tabularnewline
 \(\textbf{F}\)  &  \(\textbf{F}\)  &  \(\textbf{T}\)  &  \(\textbf{U}\)  &  \(\textbf{U}\)  &  \(\textbf{U}\)  &  \(\textbf{T}\)  &  \(\textbf{T}\) \tabularnewline
 \(\textbf{F}\)  &  \(\textbf{F}\)  &  \(\textbf{T}\)  &  \(\textbf{U}\)  &  \(\textbf{T}\)  &  \(\textbf{T}\)  &  \(\textbf{T}\)  &  \(\textbf{T}\) \tabularnewline
 \(\textbf{F}\)  &  \(\textbf{F}\)  &  \(\textbf{T}\)  &  \(\textbf{T}\)  &  \(\textbf{T}\)  &  \(\textbf{T}\)  &  \(\textbf{T}\)  &  \(\textbf{T}\) \tabularnewline
 \(\textbf{F}\)  &  \(\textbf{T}\)  &  \(\textbf{F}\)  &  \(\textbf{U}\)  &  \(\textbf{F}\)  &  \(\textbf{F}\)  &  \(\textbf{U}\)  &  \(\textbf{T}\) \tabularnewline
 \(\textbf{F}\)  &  \(\textbf{T}\)  &  \(\textbf{F}\)  &  \(\textbf{U}\)  &  \(\textbf{F}\)  &  \(\textbf{F}\)  &  \(\textbf{T}\)  &  \(\textbf{T}\) \tabularnewline
 \(\textbf{F}\)  &  \(\textbf{T}\)  &  \(\textbf{F}\)  &  \(\textbf{T}\)  &  \(\textbf{F}\)  &  \(\textbf{T}\)  &  \(\textbf{F}\)  &  \(\textbf{T}\) \tabularnewline
 \(\textbf{T}\)  &  \(\textbf{F}\)  &  \(\textbf{T}\)  &  \(\textbf{F}\)  &  \(\textbf{U}\)  &  \(\textbf{U}\)  &  \(\textbf{T}\)  &  \(\textbf{T}\) \tabularnewline
 \(\textbf{T}\)  &  \(\textbf{F}\)  &  \(\textbf{T}\)  &  \(\textbf{F}\)  &  \(\textbf{T}\)  &  \(\textbf{T}\)  &  \(\textbf{T}\)  &  \(\textbf{T}\) \tabularnewline
 \(\textbf{T}\)  &  \(\textbf{T}\)  &  \(\textbf{T}\)  &  \(\textbf{F}\)  &  \(\textbf{F}\)  &  \(\textbf{F}\)  &  \(\textbf{T}\)  &  \(\textbf{T}\) \tabularnewline
\bottomrule
\end{longtable}

\begin{longtable}[]{@{}ccccc@{}}
\toprule
\(p\) & \(q\) & \(\neg p\) & \(\neg p \to q\) &
\(p \to (\neg p \to q)\)\tabularnewline
\midrule
\endhead
\(\textbf{F}\) & \(\textbf{F}\) & \(\textbf{U}\) & \(\textbf{U}\) &
\(\textbf{T}\)\tabularnewline
\(\textbf{F}\) & \(\textbf{F}\) & \(\textbf{U}\) & \(\textbf{T}\) &
\(\textbf{T}\)\tabularnewline
\(\textbf{F}\) & \(\textbf{F}\) & \(\textbf{T}\) & \(\textbf{F}\) &
\(\textbf{T}\)\tabularnewline
\(\textbf{F}\) & \(\textbf{T}\) & \(\textbf{U}\) & \(\textbf{T}\) &
\(\textbf{T}\)\tabularnewline
\(\textbf{F}\) & \(\textbf{T}\) & \(\textbf{T}\) & \(\textbf{T}\) &
\(\textbf{T}\)\tabularnewline
\(\textbf{T}\) & \(\textbf{F}\) & \(\textbf{F}\) & \(\textbf{T}\) &
\(\textbf{T}\)\tabularnewline
\(\textbf{T}\) & \(\textbf{T}\) & \(\textbf{F}\) & \(\textbf{T}\) &
\(\textbf{T}\)\tabularnewline
\bottomrule
\end{longtable}

}

\end{document}